\documentclass{amsart}
\usepackage{amsfonts}
\usepackage{amsmath}
\usepackage{amssymb}
\usepackage{amsthm}

\usepackage{multicol}
\usepackage[margin=3cm]{geometry} 

\usepackage{arydshln} 
\usepackage{makecell} 

\usepackage[pagebackref, pdftex]{hyperref}

\renewcommand*{\backref}[1]{}
\renewcommand*{\backrefalt}[4]{
  \ifcase #1
  [No citations.]
  \or [#2]
  \else [#2]
  \fi }

\usepackage{graphicx} 
\usepackage{import}

\usepackage{tikz}
\usetikzlibrary{calc, arrows, decorations.markings, decorations.pathmorphing, positioning, decorations.pathreplacing}

\AtBeginDocument{%
   \def\MR#1{}
}

\newcommand{\To}{\longrightarrow}

\newcommand{\R}{\mathbb{R}}

\renewcommand{\S}{\mathscr{S}}

\newcommand{\Z}{\mathbb{Z}}

\numberwithin{equation}{section}

\theoremstyle{plain}
\newtheorem{theorem}[equation]{Theorem}
\newtheorem{thm}[equation]{Theorem}

\newtheorem{lemma}[equation]{Lemma}
\newtheorem{lem}[equation]{Lemma}

\newtheorem{proposition}[equation]{Proposition}
\newtheorem{prop}[equation]{Proposition}

\theoremstyle{definition}
\newtheorem{defn}[equation]{Definition}

\newcommand{\refsec}[1]{Section~\ref{Sec:#1}}
\newcommand{\refdef}[1]{Definition~\ref{Def:#1}}
\newcommand{\reffig}[1]{Figure~\ref{Fig:#1}}

\newcommand{\refeqn}[1]{\eqref{Eqn:#1}}
\newcommand{\reflem}[1]{Lemma~\ref{Lem:#1}}
\newcommand{\refprop}[1]{Proposition~\ref{Prop:#1}}
\newcommand{\refthm}[1]{Theorem~\ref{Thm:#1}}

\usepackage{mathrsfs}
\usepackage{float}
\usepackage{mathtools,commath}
\usepackage{esvect}

\newcommand{\Tr}{\textnormal{Tr}}
\newcommand{\Tw}{\textnormal{Tw}}

\newcommand{\STw}{\textnormal{STw}}
\newcommand{\Pu}{\textnormal{Pu}}

\title{Generalised Kauffman Clock Theorems}

\author{Nguyen Thanh Tung Le}
\address{School of Information Technology, Can Tho University of Technology, 256 Nguyen Van Cu, Ninh Kieu, Can Tho, Vietnam}
\email{lnttung@ctuet.edu.vn}

\author{Daniel V. Mathews}
\address{School of Mathematics, Monash University, 9 Rainforest Walk, Clayton VIC 3800, Australia; School of Physical and Mathematical Sciences, Nanyang Technological University, 21 Nanyang Link, Singapore 637371}
\email{dan.v.mathews@gmail.com}

\begin{document}

\begin{abstract}
Kauffman's clock theorem provides a distributive lattice structure on the set of states of a four-valent graph in the plane. We prove two distinct generalisations of this theorem, for four-valent graphs embedded in more general compact oriented surfaces. The proofs use results of Propp providing distributive lattice structures on matchings on bipartite plane graphs, and orientations of graphs with fixed circulation.
\end{abstract}

\maketitle

\setcounter{tocdepth}{1}

\tableofcontents

\section{Introduction}

\subsection{Overview}

This article considers the \emph{Clock Theorem} of Kauffman, proved in his 1983 book \emph{Formal Knot Theory} \cite{Kauffman}, and some generalisations.

Consider a 4-valent plane graph, such as the graph obtained from a knot diagram by flattening crossings. Such a graph has two more faces than vertices; choosing two adjacent faces as \emph{starred} defines a \emph{Kauffman universe}. As Kauffman noted in \cite[p. 1]{Kauffman}, ``I have taken a perhaps startling. but certainly memorable, set of terms". 

A \emph{state} on a Kauffman universe places a \emph{marker} in one of the 4 corners at each vertex of the graph, so that no face of the graph contains more than one marker, and starred faces are forbidden to have markers. The Clock Theorem, in the formulation of Gilmer--Litherland \cite{Gilmer_Litheland_duality_86}, asserts that the set of states has a nice combinatorial structure, in particular the structure of a \emph{distributive lattice} (in the sense of a partially ordered set with joins and meets). There are natural \emph{transposition} operations moving from one state to another, which provide the covering relation for the lattice. 

In this article we present two generalisations of the Clock Theorem. Both of them apply to more general classes of surfaces than the plane. We define \emph{multiverses} generalising Kauffman universes. Roughly, a multiverse is a 4-valent graph embedded on a compact connected oriented surface, with certain starred faces prohibited from containing state markers: a precise definition is given in \refdef{multiverse}. A multiverse is like a Kauffman universe but may have more endpoints, more boundary components, higher genus, more ``strings" and ``more stars".  See \reffig{exampleofmultiverse} for some examples.

\begin{theorem}[Planar Clock Theorem]\label{Thm:main_thm_1}
Let $U$ be a framed planar multiverse. Then the set of states of $U$ forms a distributive lattice, where plane transpositions provide the covering relation.
\end{theorem}

\begin{theorem}[Arbitrary Genus Clock Theorem]\label{Thm:main_thm_2}
Let $U$ be a framed multiverse. Then the set of states of $U$ with a fixed viable circulation function forms a distributive lattice, where surface transpositions provide the covering relation.
\end{theorem}

More precise and detailed versions of these theorems are given as \refthm{clock_on_genus_0} and \refthm{clock_on_positive_genus}. All notions will be precisely defined in due course; we present a rough idea in this introduction.
\begin{figure}
    \centering
    \includegraphics[width=0.7\textwidth]{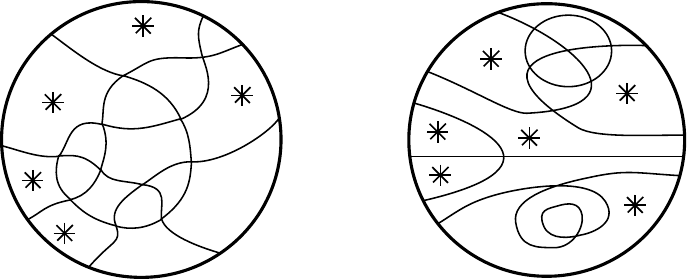}
    \includegraphics[width=\textwidth]{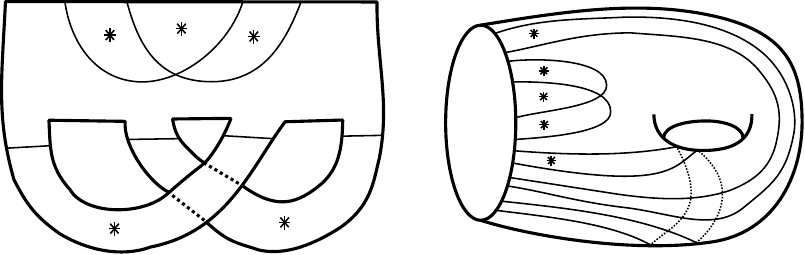}
    \caption{Top left and right: multiverses on a disc. Bottom: two depictions of the same multiverse on a punctured torus.}
    \label{Fig:exampleofmultiverse}
\end{figure}

As the name suggests, a \emph{planar multiverse} lies on a planar surface. The surface may have an arbitrary finite number of boundary components. The graph may have arbitrarily (finitely) many, or few, endpoints on each boundary component of the surface. 

A \emph{multiverse} more generally may lie on an arbitrary compact oriented surface, with arbitrarily high (finite) genus and arbitrarily (finitely) many boundary components. Again, the graph may have arbitrarily (finitely) many, or few, endpoints on each boundary component.

When the multiverse in \refthm{main_thm_1} or \refthm{main_thm_2} is on a disc, and has two endpoints on the boundary of the disc, it is essentially a Kauffman universe: a \emph{universe in string form} or \emph{string universe} in Kauffman's terminology \cite{Kauffman}. 

Although our two main results \refthm{main_thm_1} and \refthm{main_thm_2} clearly involve generalisations of the ideas of Kauffman's Clock theorem, and apply in more general contexts, their strict logical relationship with each other and with Kauffman's clock theorem is a little subtle. Our notion of ``plane transposition" in \refthm{main_thm_1} does \emph{not} immediately reduce to Kauffman's notion of ``transposition" on a Kauffman universe; plane transpositions are more general.
However, it turns out that on a Kauffman universe, the only plane transpositions that can arise are Kauffman transpositions, and in \refsec{applications_to_universes} of this paper we show
the following.
\begin{thm}
\label{Thm:Kauffman_plane_equivalence}
Let $U$ be a Kauffman universe. Then the distributive lattices of $U$ given by Kauffman's Clock Theorem and the Planar Clock \refthm{main_thm_1} are isomorphic, with Kauffman transpositions corresponding to plane transpositions.
\end{thm}
A precise statement is given in \refthm{clock_lattices_isomorphic}. Thus, \refthm{main_thm_1} \emph{is} a generalisation of Kauffman's Clock Theorem, but not obviously so.

It may appear that \refthm{main_thm_2} is a generalisation of \refthm{main_thm_1}, rendering \refthm{main_thm_1} redundant. However, the notions of transpositions involved in these theorems are quite distinct. The notion of ``surface transposition" in \refthm{main_thm_2} is considerably more restricted than the notion of ``plane transposition" in \refthm{main_thm_1}, and neither is in general a subset of the other.
The notion of plane transposition relies on a property unique to planar surfaces, namely that every simple closed curve has an inside and outside. Surface transpositions require possibly multiple curves to be involved, bounding a common surface. While the states in \refthm{main_thm_1} can all be connected by plane transpositions, the states in \refthm{main_thm_2} in general cannot be connected by surface transpositions; only those of a fixed circulation function (which we will define in due course) are connected to each other. Both theorems apply to a framed planar multiverse, but will in general produce distinct lattices: for example \reffig{Hasse_diagram_example1_thm1} and \reffig{Hasse_diagram_example1_thm2}, show the lattices obtained from \refthm{main_thm_1} and \refthm{main_thm_2}, for the same framed planar multiverse.

\begin{figure}
\begin{center}
    \includegraphics[width=\textwidth]{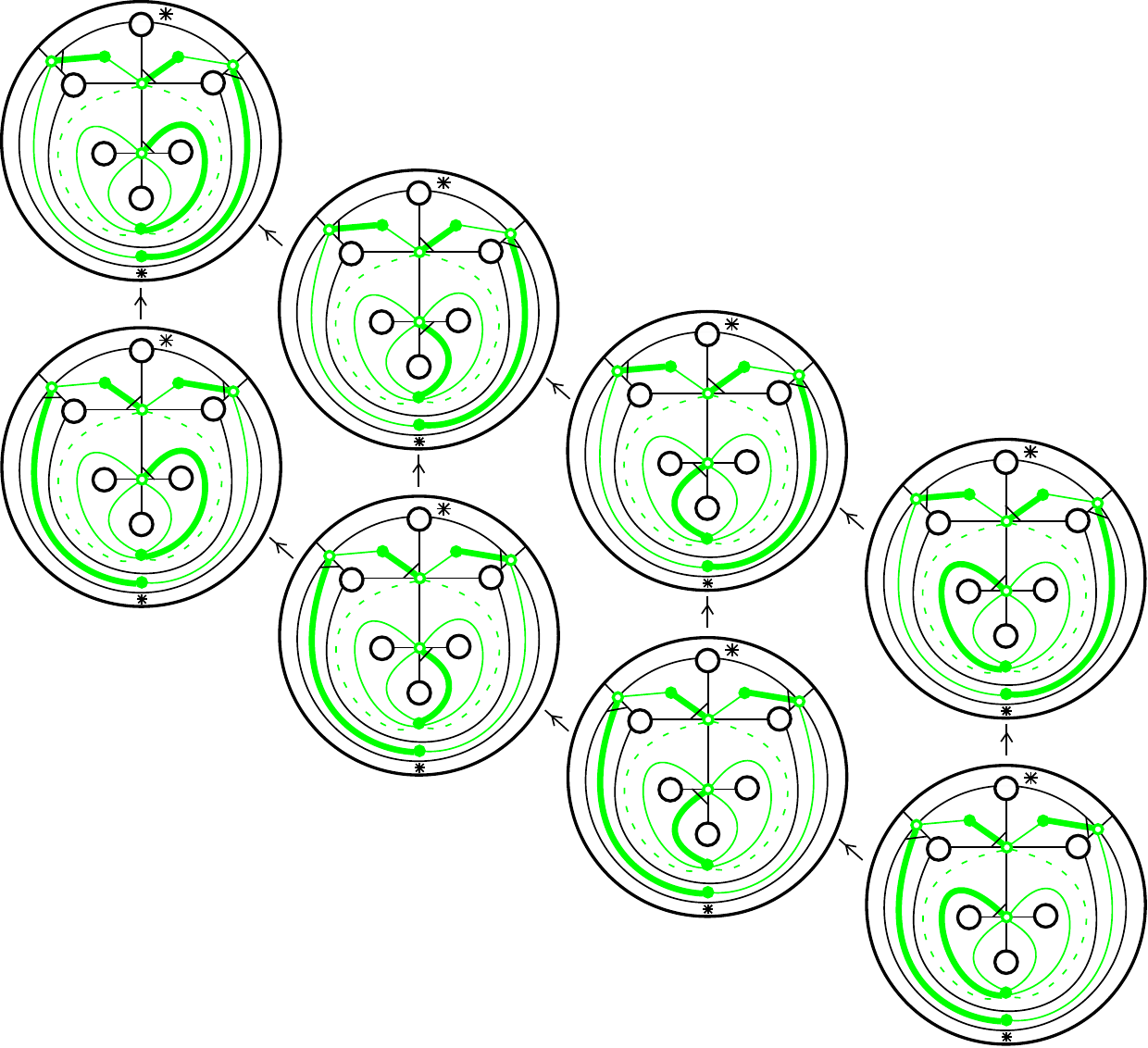}
\end{center}
    \caption{The Hasse diagram of the lattice obtained from \refthm{main_thm_1}, for a framed planar multiverse.}
    \label{Fig:Hasse_diagram_example1_thm1}
\end{figure}

\begin{figure}
\begin{center}
    \includegraphics[width=0.75\textwidth]{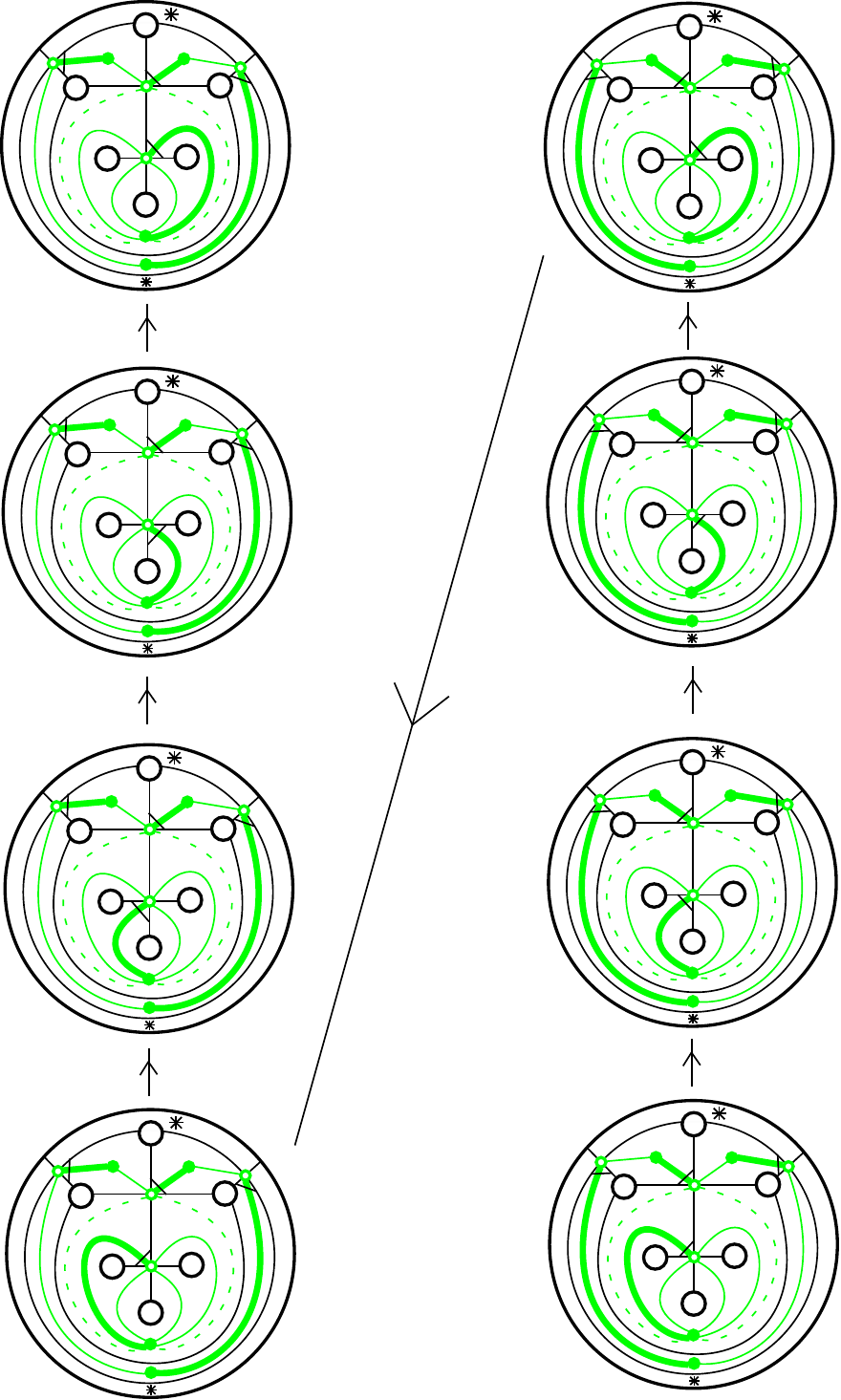}
\end{center}
    \caption{The Hasse diagram of the lattice obtained from \refthm{main_thm_2}, for a framed planar  multiverse.}
    \label{Fig:Hasse_diagram_example1_thm2}
\end{figure}

Moreover, \refthm{main_thm_2} is not, strictly speaking, a generalisation of Kauffman's Clock theorem. When $U$ is a Kauffman string universe, \refthm{main_thm_2} produces a distributive lattice, but it is in general different from the one provided by the original Clock theorem. See for instance \reffig{Hasse_diagram_example2_thm1} and \reffig{Hasse_diagram_example2_thm2}.

In the remainder of this introduction, we briefly give a precise recollection of Kauffman's Clock Theorem, as formulated by Gilmer--Litherland, and a rough idea of our generalisations, as well as general background, existing literature and context.

\begin{figure}
\begin{center}
\includegraphics[width=0.6\textwidth]{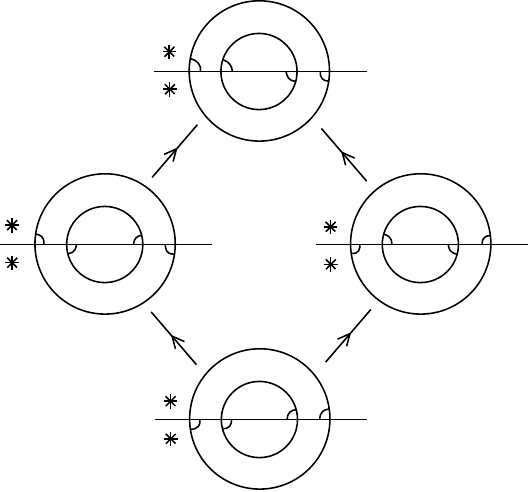}
\end{center}
    \caption{The Hasse diagram of the lattice obtained from Kauffman's clock theorem or \refthm{main_thm_1}, for a Kauffman string universe.}
    \label{Fig:Hasse_diagram_example2_thm1}
\end{figure}

\begin{figure}
\begin{center}
\includegraphics[width=0.25\textwidth]{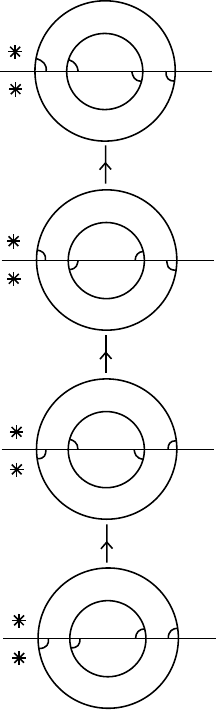}
\end{center}
    \caption{The Hasse diagram of the lattice obtained from \refthm{main_thm_2}, for a  Kauffman string universe.}
    \label{Fig:Hasse_diagram_example2_thm2}
\end{figure}

\subsection{Kauffman's clock theorem}
\label{Sec:string_universes}

We recount some definitions and results from Kauffman \cite{Kauffman}.
\begin{defn}
\label{Def:universe}
A \emph{Kauffman universe}, or just \emph{universe}, is a pair $(U, \mathscr{F})$ where
\begin{enumerate}
\item $U$ is a connected $4$-valent plane graph, and
\item $\mathscr{F} = \{ F_0, F_1 \}$ is a set of two distinct faces of $U$, called \emph{starred} faces, where $F_0$ is the unbounded face, and $F_1$ shares an edge with $F_0$.
\end{enumerate}
\end{defn}

Here, as usual, by a plane graph we mean a graph embedded in the plane $\R^2$, with edges embedded as curves which intersect only at vertex endpoints according to the incidence relations of the graph. The \emph{faces} of the  plane graph $U$ are the connected components of $\R^2 \setminus U$. Precisely one face is unbounded. 
When the faces are understood we often denote the universe simply by $U$.
We mark each starred face with a star. An example of a universe is shown in \reffig{example_string} (left).
\begin{figure}
    \centering
    \includegraphics[width=0.5\linewidth]{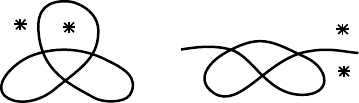}
    \caption{Left: a universe obtained by flattening the crossings of a trefoil knot. Right: the corresponding string universe.}
    \label{Fig:example_string}
\end{figure}

At each vertex $v$ of $U$, there are four adjacent \emph{corners}, namely the four connected components of the complement of $U$ in a small neighbourhood of $v$. Each corner of $v$ lies in a face of $U$, but there may be distinct corners of $v$ which lie in the same face.

A straightforward Euler characteristic argument shows that the number of faces of $U$ is 2 more than the number of vertices. (We prove a general result in \reflem{Euler_char_arg}.) Thus the number of unstarred faces is equal to the number of vertices. Hence the following definition makes sense.
\begin{defn} 
\label{Def:state_universe}
A \emph{state} of $(U, \mathscr{F})$ is a choice of corner at each vertex of $U$, so that each unstarred face is chosen precisely once.
\end{defn}
To draw a state, we place a marker in the corner chosen at each vertex. \refdef{state_universe} precisely requires that each unstarred face of $U$ contains precisely one marker, as in the examples of \reffig{ex_state}.
\begin{figure}
    \centering
    \includegraphics[width=0.5\linewidth]{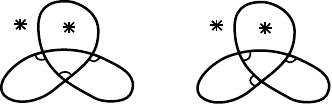}
    \caption{Two states of a universe.}
    \label{Fig:ex_state}
\end{figure}

Given a state $S$ on a universe $U$, we may consider altering the assignment of corners at vertices, by rotating the markers as follows.
\begin{defn}
\label{Def:transposition}
A \emph{clockwise (resp. counterclockwise) Kauffman transposition}, or just \emph{transposition}, on a state $S$ is a simultaneous $90^\circ$ clockwise (resp. counterclockwise) rotation of two distinct markers, which results in another state.
\end{defn}
Suppose $u$ and $v$ are the vertices involved in a transposition on a state $S$. Then the marker which rotates at $u$ rotates out of a face $F$ and into another face $F'$; the marker at $v$ must then rotate out of $F'$ and into $F$. Hence any transposition appears locally as in \reffig{transposition}. Note the box in \reffig{transposition} can contain an arbitrarily complicated (or trivial) 4-valent graph, but the graph inside the box only intersects the boundary of the box at two points, on edges incident to $u$ and $v$ respectively.

\begin{figure}[H]
\begin{center}
\includegraphics[width=\textwidth]{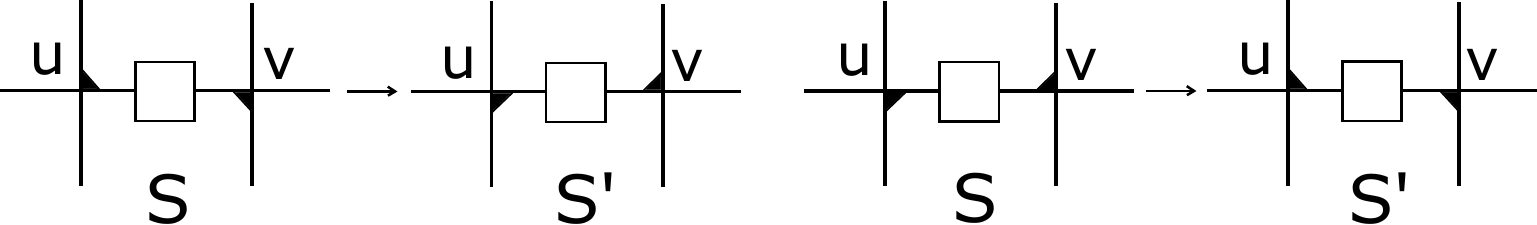}
\caption{Transpositions, clockwise (left) and counterclockwise (right).}
\label{Fig:transposition}
\end{center}
\end{figure}

We use transpositions to define a relation on states as follows.
\begin{defn}
\label{Def:leq_on_states}
Let $S,S'$ be states of $U$. We write $S \leqslant S'$ if there exists a sequence of states $S=S_0, S_1, \ldots, S_n = S'$ of $U$, for some integer $n \geqslant 0$, such that each $S_{j+1}$ is obtained from $S_j$ by a counterclockwise transposition.
\end{defn}
Note that we allow $n=0$, so $S \leqslant S$ for each state $S$. It is clear from the definition that $\leqslant$ is transitive. It is perhaps less clear, but true, that $\leqslant$ is a partial order, so that the symbol $\leqslant$ is justified. In fact, it is a distributive lattice: this is the content of Kauffman's Clock Theorem, as formulated by Gilmer--Litherland. The distributivity of the lattice was first stated in \cite{Gilmer_Litheland_duality_86}. 
\begin{theorem}[Kauffman's Clock Theorem]
\label{Thm:clock_universe}
Let $(U, \mathscr{F})$ be a universe and $\mathscr{S}$ its set of states. Then $\S$, equipped with the relation $\leqslant$, is a distributive lattice. Moreover, a state $S$ is covered by another state $S'$ if and only if $S'$ is obtained from $S$ by a counterclockwise transposition.
\end{theorem}
The lattice here is called the \emph{clock lattice} of $U$.
We will recall the definition of distributive lattice, covering, and the various generalisations we need, in \refsec{lattices}. 

Kauffman equivalently expresses the clock theorem using \emph{universes in string form}, or \emph{string universes}. Given a universe $(U, \mathscr{F})$, select the edge separating the two starred faces and break it. Extending the resulting half-edges infinitely to the left and right yields a universe in string form. Truncating to a sufficiently large closed disc $D$, whose boundary $\partial D$ contains endpoints of the half-edges, we obtain the following.
\begin{defn}[Universe on a disc]
\label{Def:universe_on_disc}
A \emph{universe on a disc} is a triple $(U, D, \mathscr{F})$ where
\begin{enumerate}
\item 
$D$ is a closed disc.
\item 
$U$ is a connected graph embedded in $D$, with $2$ \emph{boundary} vertices of degree $1$, and all other vertices of degree $4$. The boundary vertices lie on $\partial D$, and the other vertices in the interior of $D$.
\item 
$\mathscr{F}$ is a set of two distinct \emph{starred} faces of $U$, namely the $2$ faces of $U$ adjacent to $\partial D$.
\end{enumerate}
\end{defn}
Thus, a universe on a disc can be regarded as a Kauffman string universe, truncated to a disc. 
The two unbounded faces of a string universe correspond to the two faces of a universe on a disc adjacent to its boundary. These are the starred faces $\mathscr{F}$.
An example is shown in \reffig{example_string} (right).

Regarding the disc $D$ as a subset of $\R^2$, \refdef{state_universe} of state generalises to a universe on a disc (we place markers at the corners of degree-4 vertices only), as does \refdef{transposition} of transposition. The states and transpositions on a Kauffman universe are naturally bijective with those on the corresponding universe on a disc. We can then define a  relation $\leq$ on states by \refdef{leq_on_states}, and Kauffman's Clock Theorem (as formulated by Gilmer--Litherland) then has the following equivalent reformulation.
\begin{theorem}[Kauffman's Clock Theorem for universes on a disc]
\label{thm:clock_on_plane}
Let $(U, D, \mathscr{F})$ be a universe on a disc, and  $\mathscr{S}$  its set of states.
Then, $\mathscr{S}$ equipped with $\leqslant$, is  distributive lattice.
Moreover, a state $S$ is covered by another state $S'$ if and only if $S'$ is obtained from $S$ by a counterclockwise transposition.
\end{theorem}
From a universe on a disc, or a string universe, one can recover a Kauffman universe in two ways, by connecting the boundary vertices above or below the diagram. The starred faces then lie on either side of this joined edge.

Universes, whether as in the original \refdef{universe} or in string form on a disc as in \refdef{universe_on_disc}, can be unified by compactifying $\R^2$ or $D$ into the sphere $S^2$. Then a universe can be regarded as a connected 4-valent graph embedded in $S^2$, together with a choice of two distinct faces sharing an edge, and notions of states, transpositions, and clock theorem can be formulated accordingly and equivalently. 

\subsection{Generalised universes, states, and transpositions}

This paper essentially consists of defining generalisations of the notions in Kauffman's Clock Theorem, and showing that analogous Clock Theorems hold. We now briefly give some idea of these generalisations and our approach, but as they involve numerous details, precise definitions and technical justifications are deferred to later sections.

Our multiverses are designed as generalisations of universes on a disc. Thinking of the embedded graph $U$ in a Kauffman universe as a collection of intersecting \emph{strings}, a multiverse generalises a universe on a disc to more strings, embedded on more general surfaces.

To preserve notions of clockwise and counterclockwise, we require a multiverse to lie on an oriented surface $\Sigma$. To preserve a notion of ``exterior" face, we require $\Sigma$ to have a designated \emph{outer} boundary component. On this $\Sigma$ we embed a 4-valent graph $U$, similarly to a universe. To preserve the notion of state as a choice of corner at each vertex, providing a bijection between vertices and unstarred faces, we require an appropriate number of stars for this to be possible. A precise definition of multiverse is given in \refdef{multiverse}. A \emph{planar multiverse} is simply a multiverse where $\Sigma$ is planar. A \emph{state} on a multiverse is defined exactly as on a universe. In our figures, $\partial \Sigma$ is usually drawn in thick black, and the graph $U$ is drawn in thin black. See for example \reffig{exampleofmultiverse} through  \reffig{Hasse_diagram_example1_thm2}. 

While Kauffman transpositions adjust precisely two markers of a state, our notions of transpositions may in general adjust arbitrarily many state markers. We consider certain curves we call \emph{transposition contours}, which are closed curves $\gamma$ alternately passing through vertices and faces of the multiverse; a precise definition is given in \refdef{transposition_contour}. When $\gamma$ passes through a vertex at $v$, it passes through two corners at $v$, and our generalised transpositions move state markers from one of these corners to the other. See \reffig{alternating_cycle} (left). A Kauffman transposition can then be regarded as a particular type of transposition along a contour: see \reffig{Kauffman_as_contour_transposition}.

\begin{figure}
\begin{center}
\includegraphics[width=0.8\textwidth]{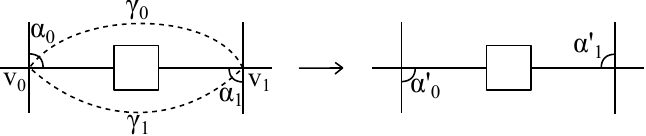}
\end{center}
    \caption{A Kauffman transposition as a contour $2$-transposition.}
    \label{Fig:Kauffman_as_contour_transposition}
\end{figure}

When the faces of $U$ are simply connected, a transposition contour is essentially determined by the vertices and corners through which it passes. However in general it turns out to be necessary to ``frame" transposition contours, ``guiding" them to run along a choice of graph $G$ on $\Sigma$ dual to $U$, which we call a \emph{spine} and which forms what we call a \emph{framing} for the multiverse. A spine is a certain subgraph of a generalised overlaid Tait graph; we define it precisely in \refdef{spine}. In our figures, spines are shown in green. See for example \reffig{Hasse_diagram_example1_thm1} and \reffig{Hasse_diagram_example1_thm2}. Without framings, \refthm{main_thm_1} fails, as we explain in \refsec{n-transpositions}, and \refthm{main_thm_2} cannot be formulated.

On a planar multiverse, a transposition contour has an interior and exterior. The \emph{plane transpositions} of \refthm{main_thm_1} are  transpositions along framed contours $\gamma$, where the state markers rotate through the interior of $\gamma$, and satisfy a further condition relating to corners in the interior of $\gamma$; a precise definition is given in \refdef{n_transposition}.

On a general multiverse, our transpositions are performed along \emph{subsurfaces} $\Psi$ of $\Sigma$ bounded by framed transposition contours, and certain boundary components of $\Sigma$. In a \emph{surface transposition}, as in \refthm{main_thm_2}, state markers again rotate along transposition contours, through the interior of $\Psi$; again further conditions are required, relating to the interior of $\Psi$. A precise definition is given in \refdef{transposition_on_positve_genus}.

Our proofs use several results of Propp \cite{Propp}, providing lattice structures on orientations of graphs, and on matchings of plane bipartite graphs. 
A spine of a multiverse is naturally a bipartite graph, and states of the multiverse correspond to matchings on the spine. Propp introduces a notion of \emph{twisting} a matching of a bipartite graph on an elementary cycle, which provides a covering relation in his lattice. This notion of twisting leads eventually to the notion of plane transposition in \refthm{main_thm_1}, although there are numerous subtleties in translating to multiverses.

From a spine of a multiverse $G$, we also construct a \emph{dual} $G^\perp$, similar to but more general than the notion of duality for plane graphs. Matchings on $G$ (and hence states of $U$) correspond precisely to certain \emph{orientations} on $G^\perp$. Propp uses a notion of \emph{pushing} an \emph{accessibility class} of an oriented graph, due to Pretzel \cite{Pretzel_86}. This notion eventually leads to the notion of surface transposition in \refthm{main_thm_2}, although again, numerous subtleties arise. In particular, surface transpositions preserve the \emph{circulation} of a state, which is a function keeping track of orientations around cycles of $G^\perp$; a precise definition is given in \refdef{circulation_of_state_matching}. \refthm{main_thm_2} yields a separate distributive lattice for each subset of the states with a fixed viable circulation function. Thus, the set of states can be regarded as a ``disconnected lattice", with one connected component for each viable circulation function.

\subsection{Related work and context}

Since its publication in 1983, Kauffman's Clock Theorem has seen numerous applications and generalisations. Kauffman states have become a standard notion in knot theory. The literature on the topic is too vast for us to attempt a comprehensive summary here. We merely mention some of the existing work on the topic, including those results which to our knowledge are most closely related to this paper. 

First,  Kauffman states have been related to various related graph-theoretic concepts, often by using various types of graphs associated to a knot diagram such as Tait graphs, and their spanning trees. 
In 1986, Gilmer--Litherland \cite{Gilmer_Litheland_duality_86} gave a simplified proof of Kauffman's clock theorem, showing that spanning trees of a certain Tait graph correspond bijectively to states, and devising an operation on spanning trees corresponding to a transposition.
In 2014, Cohen--Teicher \cite{Cohen_Teicher_14} gave a formula for the height of the clock lattice for a knot, by considering perfect matchings of an overlaid Tait graph.
Our arguments in this paper rely crucially on a generalised correspondence between states and matchings on the overlaid Tait graph.

Second, Kauffman states have been used directly to prove knot-theoretic results. For instance,
Stoimenow \cite{Stoimenow_det_07} used the fact that the number of states of an alternating link is equal to its determinant, to relate the determinant of an alternating link in $S^3$ to its hyperbolic volume. Kaplan--Krcatovich--O'Brien \cite{KKO_15_resolution_depth} used Kauffman states to give a bound on the resolution depth of the closure of strictly positive braids. Madaus--Newman--Russell \cite{MMR_Dehn_dimer_17} used Kauffman states, and corresponding matchings on an overlaid Tait graph, to study Dehn and Fox colourings of a knot and the related colouring modules.

Third, certain knot polynomials may be expressed as a sum over Kauffman states. A \emph{state sum} formulation of the Alexander--Conway polynomial of a knot was given by Kauffman already in \emph{Formal Knot Theory} in 1983 \cite{Kauffman}, at the same time as the Clock Theorem. Various generalistaions of state sums have been given.
In 2019, Kidwell--Luse \cite{Kidwell_Luse_19} related some terms of the Alexander polynomial of a rational link to twist regions in a particular diagram for the knot, using transpositions within the clock lattice. Most recently, in 2024 G\"{u}g\"{u}mc\"{u}--Kauffman \cite{Gugumcu_Kauffman} constructed ``mock Alexander polynomials" for starred links and linkoids in surfaces, defined as sums over a generaliesd notion of states.

The Jones polynomial also has state sum formulations. 
Kauffman in 1987 \cite{Kauffman_87_State_Jones} gave a state model for the Jones polynomial, although using a more general notion of state; see also \cite{Kauffman_combinatorial_Jones_23}. Thistlethwaite in 1987 \cite{Thistlethwaite_spanning_tree_Jones_87} showed that the Jones polynomial of a link may be expressed as a sum over spanning trees of a Tait graph. This spanning tree model for the Jones polynomial was extended to Khovanov homology independently by 
Wehrli \cite{Wehrli_08} and 
Champanerkar--Kofman \cite{CK_09_Spanning_Khovanov}. Champanerkar--Kofman--Stoltzfus \cite{CKS_07} showed that these spanning trees correspond to certain spanning ribbon subgraphs of a ribbon graph.

Fourth, Kauffman states have become a standard tool in Heegaard Floer theory. Ozsv\'{a}th--Szab\'{o} in 2003 \cite{OS_alternating_03} gave a description of the knot Floer homology chain complex where the generators are Kauffman states; they gave the states a multi-filtration in \cite{OS_04_KFH_genus_bounds_mutation}. This can be regarded as a categorification of the state sum model of the Alexander--Conway polynomial. Manion in \cite{Manion_21} 
related the Heegaard diagrams used in knot Floer homology which yield Kauffman states, to other standard Heegaard diagrams used in knot Floer homology. Kauffman--Silvero in 2016 \cite{Kauffman_Silvero_16} gave a detailed account of the relationship between knot Floer homology and  the Kauffman state sum model of the Alexander--Conway polynomial. See also Ozsv\'{a}th--Szab\'{o}'s 2018 overview \cite{OS_overview_18}.

Kauffman state generators for knot Floer homology have been widely used. For instance, Ozsv\'{a}th--Szab\'{o} used the Kauffman state generators to give a skein exact sequence for knot Floer homology \cite{OS09_cube_of_resolutions}. Kauffman state generators were also used by Lidman--Moore \cite{Lidman_Moore_16} to classify pretzel knots with L-space surgeries; by Varvarezos \cite{Varvarezos_21} to prove that 3-braid knots do not admit purely cosmetic surgeries; and by Troung \cite{Troung_thickness_23} to bound the dealternating number of a knot.

The relationship between Kauffman states and spanning trees of Tait graphs have also been used in the Heegaard Floer context. For instance,
Lowrance \cite{Lowrance_08} used Kauffman state generators and spanning trees of the Tait graph to give bounds on knot Floer width. Baldwin--Levine \cite{Baldwin_Levine_12} gave a combinatorial description of certain knot Floer homology groups in terms of spanning trees of Tait graphs.

More recently, in 2018 Ozsv\'{a}th--Szab\'{o} introduced the ``Kauffman states functor", associating to a knot the homology of a chain complex generated by Kauffman states, and considering type A and D structures, in the sense of bordered Floer homology, associated to the upper and lower parts of a knot diagram  split along a horizontal half plane. This involves a generalised notion of ``upper" Kauffman states. Manion related the algebras involved to Khovanov--Seidel quiver algebras \cite{Manion_17_Khovanov-Seidel} and quantum supergroup representations \cite{Manion_decat_bordered_19}, and Manion--Marengon--Willis gave descriptions via path algebras on quivers \cite{MMW_generators_relations_21}.

Fifth, and related to the Heegaard Floer applications, Kauffman states and related notions arise in 3-dimensional contact topology. In \cite{Mathews_strand_algebras_contact_categories_19}, the second author discusses partial orders on objects of a contact category as reminiscent of Kauffman’s clock theorem. In \cite{Kalman_Mathews_20}, the second author and K\'{a}lm\'{a}n showed that the number of states of a universe is equal to the number of isotopy classes of tight contact structures on a sutured 3-manifold which is topologically a handlebody, and that the trails corresponding to states have a natural interpretation as dividing sets on certain surfaces in the handlebody.

Sixth, several generalisations of Kauffman's theory have been made to singular links. Ozsv\'{a}th--Szab\'{o} \cite{OS09_singular_knots} extended link Floer homology to oriented singular knots in $S^3$, generalising the notion of Kauffman states to singular knots. Fielder \cite{Fiedler_10} extended Kauffman states to singular links and used them to define a Kauffman state model of Jones and Alexander polynomials of singular links. Manion \cite{Manion_singular_21} generalised the Kauffman states functor to singular crossings.

Seventh, recently connections have been found to cluster algebras.
In 2022 Bazier-Matte--Schiffler \cite{Bazier-Matte_Schiffler_22} constructed, from a knot diagram (universe), a quiver representation, which provides a representation-theoretic analogue of Kauffman states, and an associated cluster algebra. They show that the lattice of Kauffman states is isomorphic to the lattice of submodules of the corresponding representation.

Eighth, and finally, several explicit generalisations of Kauffman states and the  clock theorem have been given.

In 2003, Murasugi--Stoimenow \cite{Murasugi_Stoimenow_03} generalised Kauffman states to planar even valence graphs, defining an Alexander polynomial in this more general case.

In 2009, Roberts \cite{Roberts_09} considered knot Floer homology for string links in $D^2 \times I$. Considering these as planar graphs in the unit square with boundary points, and related spanning forests, Roberts shows that these are in bijective correspondence with generators of the Floer complex for the string link, and proves a generalised version of Kauffman's clock theorem in this context, extending spanning forests into a spanning tree to which Kauffman's clock theorem applies.

In 2014, Cohen and Teichner \cite{Cohen_12, Cohen_Teicher_14} showed how to associate a discrete Morse function to a Kauffman state on a given knot diagram in the plane. These can be regarded as perfect matchings on the balanced Tait graph. Celoria--Yerolemou subsequently \cite{Celoria_Yerolemou_21} proved several results about such states and moves on them, including a generalised clock theorem involving  matchings of the overlaid Tait graph, and ``click" and ``clock" moves.

In 2014, Bao \cite{Bao_14} generalised the construction of Heegaard Floer homology for a singular knot to certain bipartite graphs embedded in $S^3$, and defined states on certain graph diagrams in $S^2$, using  used them to define a state sum formula for a generalised Alexander polynomial.

In 2019, Zibrowius \cite{Zibrowius_states_19} studied tangles in $S^3$, defining generalised Kauffman states on them. These generalised states are a special case of states defined in this paper; both are straightforward generalisations of Kauffman's definition.
On such tangles, Zibrowius defined an Alexander polynomials, with a state sum formulation, as well as a Heegaard Floer invariant. He also proved a generalised clock theorem in this context; however the operations used on states there do not correspond precisely to any type of transpositions defined in this article.

In 2020, Bao--Wu \cite{Bao_Wu_MOY_Alexander_20} introduced an Alexander polynomial for MOY graphs, generalised Kauffman states to such graphs, and defined an Alexander polynomial as a state sum. Subsequently \cite{Bao_Wu_21} they related this polynomial to a sum over spanning trees.

In 2023, Celoria \cite{Celoria_uber_23} constructed a filtration on the simplicial homology of a finite simplicial complex using bi-colourings of its vertices, closely related to Kauffman states. Applied to a knot diagram in the plane, one obtains matchings on the overlaid Tait graph.

Perhaps most closely related to our approach is the 2018 work of Hine--K\'{a}lm\'{a}n \cite{Hine_Kalman}. They consider 3-coloured triangulations of the sphere and torus, also known as \emph{trinities}, together with certain matchings on such objects, generalising Kauffman states, and certain moves between them, generalising transpositions. The resulting set of states is somewhat different from ours. They show that the set of states and moves on a planar trinity is a connected distributive lattice. For a toric trinity, the set of states may have multiple components, some of them cyclic, but acyclic components form distributive lattices.

Most recently, as we were finalising this article, G{\"u}g{\"u}mc{\"u} and Kauffman posted a preprint of a clock theorem for knotoids and linkoids \cite{Gugumcu_Kauffman_25}.

\subsection{Structure of this paper}

In \refsec{background}, we recall necessary background results from lattice theory and the work of Propp which forms the basis for our two generalisations of the clock theorem. For Propp's theorem on orientations of graphs, we require notions of \emph{accessibility classes}, \emph{circulations}, \emph{pushing}. For Propp's theorem on matchings of plane bipartite graphs, we require notions of \emph{elementary cycles}, \emph{alternating paths}, and \emph{positive and negative cycles}. We also slightly generalise Propp's results, as we need them in the slightly more general context of disconnected graphs.

In \refsec{multiverses_general}, we introduce our notion of multiverses, and some general notions about multiverses required for our proofs. These include notions which have already been mentioned, such as \emph{Tait graph}, \emph{spine},  \emph{reduced spine}, ane \emph{dual of spine}, as well as several others.

In \refsec{genus_0} we specialise to planar multiverses and prove \refthm{main_thm_1}, applying Propp's theorem on matchings of plane bipartite graphs.

In \refsec{applications_to_universes} we specialise further to universes (on a disc), showing that \refthm{main_thm_1} reduces to Kauffman's Clock Theorem in this case, so that we have a bona fide generalisation.

Finally, in \refsec{positive_genus} we prove \refthm{main_thm_2}, applying Propp's theorem on orientations of graphs.

\subsection{Acknowledgments}

This article derives from part of the first author's  doctoral thesis, supervised by the second author.
The second author is supported by Australian Research Council grant DP210103136.
The second author thanks Tam\'{a}s K\'{a}lm\'{a}n for introducing him to the Kauffman Clock Theorem, and Dionne Ibarra for pointing out some relevant literature. The authors also thank James Propp for helpful comments.

\section{Background}
\label{Sec:background}

\subsection{Lattices}
\label{Sec:lattices}

We now introduce notions from the theory of orders and lattices required for our results. These notions are all standard. We follow \cite{Davey_Priestley_02} in the following definitions. All ordered sets and lattices we consider are finite.

Let $L$ be a partially ordered set. We denote the partial order relation by $\leqslant$. We write $x<y$ if $x \leqslant y$ and $x \neq y$.
\begin{defn}[Cover]
We say $x$ is \emph{covered} by $y$, or equivalently $y$ \emph{covers} $x$, if $x \leqslant y$, and $x \leqslant z < y$ implies $z=x$. We write $x \lessdot y$.
\end{defn}
If $L$ is finite (as in all cases we consider), $x \leqslant y$ if and only if there exists a finite sequence $x_0, \ldots, x_n$ in $L$ such that $x=x_0 \lessdot x_1 \lessdot \cdots \lessdot x_n = y$, for some $n \geqslant 0$.

We depict lattices by \emph{Hasse diagrams}. Each element of a lattice $L$ is drawn as a point, and points are joined by an arrow when one element covers another; if $x \lessdot y$ then an arrow points from $x$ to $y$.

\begin{defn}[Lattice]
\label{Def:lattice}
A \emph{lattice} is a partially ordered set such that every finite non-empty subset has a least upper bound, or \emph{join}, and a greatest lower bound, or \emph{meet}. 
\end{defn}
We denote the join of two elements $x,y$ as $x \vee y$ and the meet as $x \wedge y$.
\begin{defn}[Distributive lattice]
\label{Def:distributive_lattice}
A lattice is \emph{distributive} if any elements $x,y,z$ satisfy the distributive law
\[
x \wedge (y \vee z) = (x \wedge y) \vee (x \wedge z).
\]
\end{defn}

This distributive law is equivalent to its dual (see e.g. \cite[lem. 4.3]{Davey_Priestley_02}):
\[
x \vee (y \wedge z) = (x \vee y) \wedge (x \vee z).
\]

\begin{defn}
A map of partially ordered sets $\phi \colon L \To L'$ is an \emph{isomorphism} if it is bijective and, for all $x,y \in L$, $x \leqslant y$ if and only if $\phi(x) \leqslant \phi(y)$.
\end{defn}

If $\phi$ is an isomorphism $L \To L'$, then $x \lessdot y$ iff $\phi(x) \lessdot \phi(y)$.
Moreover, if one of $L$ or $L'$ has the further structure of a lattice or distributive lattice, then so does the other. 
If $L$ and $L'$ are lattices, then for all $x,y \in L$ we have $\phi(x \vee y) = \phi(x) \vee \phi(y)$ and $\phi(x \wedge y) = \phi(x) \wedge \phi(y)$.
Thus, we can also refer to such a $\phi$ as a \emph{lattice isomorphism} or \emph{distributive lattice isomorphism} as appropriate.

\begin{defn}[Product of partial orders and lattices]
\label{Def:product_orders}
Let $L_1, \ldots, L_n$ be partially ordered sets.
\begin{enumerate}
\item 
The \emph{product partial order} on the Cartesian product of sets $L_1 \times \cdots \times L_n$ is given by
\[
(x_1, \ldots, x_n) \leqslant (y_1, \ldots, y_n)
\quad \text{iff} \quad 
\text{each $x_i \leqslant y_i$ in $L_i$.}
\]
\item 
If $L_1, \ldots, L_n$ are also distributive lattices, then $L_1 \times \cdots L_n$ is also a distributive lattice, with
\begin{align*}
(x_1, \ldots, x_n) \vee (y_1, \ldots, y_n) &= (x_1 \vee y_1, \ldots, x_n \vee y_n) \\
(x_1, \ldots, x_n) \wedge (y_1, \ldots, y_n) &= (x_1 \wedge y_1, \ldots, x_n \wedge y_n).
\end{align*}
\end{enumerate}
\end{defn}
See e.g. \cite[sec. 1.25, 2.15, 4.7]{Davey_Priestley_02}.
In $L_1 \times \cdots \times L_n$, note that $(x_1, \ldots, x_n) \lessdot (y_1, \ldots, y_n)$ if and only if $x_j \lessdot y_j$ for a unique $j$, and $x_i = y_i$ for all $i \neq j$.

\subsection{Cycles and directed cycles in graphs}

We introduce various precise notions needed for our arguments.
We follow Propp \cite{Propp} but introduce some further elaborations. Propp works with graphs without loops, but we  consider graphs with loops. The results carry through without much difficulty.

Let $X$ be a finite graph, possibly disconnected, possibly with multiple edges and loops.  Let $V$ denote its set of vertices. Each edge of $X$ can be oriented two ways (even if it is a loop). An oriented edge has an \emph{inital} vertex $a \in V$ and a \emph{terminal} vertex $b \in V$; we say $e$ is \emph{oriented} or \emph{directed} from $a$ to $b$. Following Propp \cite{Propp}, we can denote an edge $e$ directed from $a$ to $b$ by $(e,a,b)$. (However, this notation can be misleading for loops: when $e$ is a loop at $v \in V$, $e$ has two orientations, but they are both denoted $(e,v,v)$.) The set of directed edges of $X$ is denoted $\vv{E}$. If $\vv{e} \in \vv{E}$ then we denote by $-\vv{e}$ the same edge with the opposite orientation.

\begin{defn}[Directed paths and cycles]
\label{Def:Propp_definitions0} \cite{Propp}
\begin{enumerate}
\item 
A \emph{directed path} in $X$ is a sequence of $n$ directed edges of $X$, of the form 
\begin{equation}
\label{Eqn:directed_path}
(e_1,v_0,v_1),\allowbreak(e_2,v_1,v_2),\allowbreak\dotsc,\allowbreak(e_n,v_{n-1},v_n)
\end{equation}
for $n \geq 1$. The vertices $v_0$ and $v_n$ are respectively called the \emph{initial} and \emph{terminal} vertices. We also allow \emph{null} directed paths where $n=0$, consisting of a single vertex $v_0$.  
\item 
A \emph{directed cycle} is a directed path whose initial and terminal vertices coincide. The set of directed cycles in $X$ is denoted $\mathscr{C}_X$ or just $\mathscr{C}$.
\item 
A directed cycle is \emph{simple} if all its directed edges are distinct.
\item
A directed cycle is \emph{vertex-simple} if all its vertices $v_1, \ldots, v_{n-1}, v_n = v_0$ are distinct.
\end{enumerate}
\end{defn}
We consider a null directed path as a directed cycle, with initial and terminal vertex $v_0$, and no edges; we refer to \emph{null directed cycles} accordingly. A null directed cycle is vacuously simple. A directed cycle of length $1$ consists of an oriented loop. If $X$ has no loops, a non-null directed cycle has length at least $2$. As directed cycles in general need not be simple, for any graph $X$ containing at least one edge, $\mathscr{C}_X$ is infinite.

Note that a simple directed cycle may visit the same edge twice, but not more than twice. If a simple directed cycle visits the same edge twice, the edge must be traversed in one direction, then the other direction. A loop traversed twice, once in each orientation, is a simple directed cycle. A simple directed cycle may visit the same vertex many times.

Note also that a vertex-simple directed cycle is simple. A directed cycle consisting of a single oriented loop is vertex-simple. If a vertex-simple directed cycle visits an edge more than once then it has length $2$ and traverses a single non-loop edge back and forth.

We will also be interested in cycles without orientations. Consider a non-null directed cycle $C$ in $X$, as in \refeqn{directed_path}, with $v_n = v_0$. We consider the following operations on $C$:
\begin{enumerate}
\item \emph{reversal}, which replaces
\[
\vv{e_1}, \vv{e_2}, \ldots, \vv{e_n}
\quad \text{with} \quad
-\vv{e_n}, \ldots, -\vv{e_2}, -\vv{e_1};
\]
\item \emph{cyclic permutation}, which replaces
\[
\vv{e_1}, \vv{e_2}, \ldots, \vv{e_n}
\quad \text{with} \quad
\vv{e_j}, \vv{e_{j+1}}, \ldots, \vv{e_n}, \vv{e_1}, \vv{e_2}, \ldots, \vv{e_{j-1}}
\]
for some $2 \leq j \leq n$.
\end{enumerate} 
Clearly directed cycles obtained from $C$ by reversal or cyclic permutation cover the same edges, in the same or reversed cyclic order.
Reversal and cyclic permutation generate an equivalence relation on directed simple cycles, which we denote $\sim$.
\begin{defn}[Undirected cycle]
\label{Def:cycle}
An \emph{undirected cycle}, or just \emph{cycle}, is an equivalence class of directed cycles under the equivalence relation $\sim$.
\end{defn}
We consider null directed cycles to form singleton equivalence classes and we accordingly obtain \emph{null cycles}.
Any cycle has a well-defined length, which is $0$ precisely when the cycle is null. The cycles of length $1$ are in bijection with loops. In a non-null cycle, each edge has two adjacent edges in the cycle.

\subsection{Orientations of graphs, accessibility, circulation and pushing}
\label{Sec:orientations_accessibility_circulation_pushing}

Again let $X$ be a finite graph, possibly disconnected, which may have multiple edges or loops. 
Again, we slightly generalise Propp's definitions in \cite{Propp} to include graphs with loops.

\begin{defn}
\label{Def:Propp_definitions} \cite{Propp} \
An \emph{orientation} on $X$ is a choice $R$ of orientation on each edge of $X$.
The pair $(X,R)$ forms an \emph{oriented graph}.
\end{defn}

When we have a directed path on an oriented graph, the orientations on edges arising in the directed path may agree with or differ from those arising from $R$, as in the following definition (also from \cite{Propp}).
\begin{defn}[Forward, backward paths and cycles] 
\label{Def:forward_backward_paths_cycles}
Let $(X,R)$ be an oriented graph.
\begin{enumerate}
\item 
Let $\vv{e}$ be an orientation on an edge $e$ of $X$.
If the orientation of $\vv{e}$ agrees with the orientation of $e$ in $R$, then $\vv{e}$ is \emph{forward} relative to $R$; otherwise, $\vv{e}$ is \emph{backward} relative to $R$.
\item 
A directed path on $X$ consisting purely of forward (resp. backward) edges relative to $R$ is called a \emph{forward path} (resp. \emph{backward path}).
\item 
A directed cycle on $X$ consisting purely of forward (resp. backward) edges relative to $R$ is called a \emph{forward cycle} (resp. \emph{backward cycle}).
\end{enumerate}
\end{defn}
We allow forward and backward paths and cycles to be null. We regard a null directed path or cycle as vacuously both forward and backward.

\begin{defn}[Accessibility] \cite{Propp} 
Let $x$ and $y$ be vertices of an oriented graph $(X,R)$. 
\begin{enumerate}
\item 
$y$ is \emph{accessible} from $x$ (relative to $R$) if there is a (possibly null) forward path from $x$ to $y$. 
\item 
$x$ and $y$ are \emph{mutually accessible} (relative to $R$) if $y$ is accessible from $x$ and $x$ is accessible from $y$.
\end{enumerate}
\end{defn}
Thus, $x$ and $y$ are mutually accessible iff there is a (possibly null) forward path from $x$ to $y$ and vice versa. 
It can be seen that mutual accessibility is an equivalence relation. Indeed, $x$ and $y$ are mutually accessible (with respect to $R$) if and only if there is a (possibly null) forward cycle passing through $x$ and $y$. The resulting equivalence classes are called \emph{accessibility classes}.

\begin{defn}[Circulation] \cite{Propp}
\label{Def:circulation}
Let $C$ be a directed cycle in $(X,R)$. 
\begin{enumerate}
\item 
The set of forward (resp. backward) edges of $C$ is denoted $C_R^+$ (resp. $C_R^-$). 
\item
The \emph{circulation} of $R$ around $C$ is $\abs{C_R^+}-\abs{C_R^-}$.
\end{enumerate}
\end{defn}
Thus a directed cycle is forward precisely when its circulation is equal to its length, and backward precisely when its circulation is equal to its negative length.

\begin{defn}[Circulation function] \ \cite{Propp}
\label{Def:circulation_function}
\begin{enumerate}
\item
The \emph{circulation function}, or just \emph{circulation}, of an orientation $R$ on $X$ is the function $c_R \colon \mathscr{C}_X \To \Z$ which assigns to each directed cycle $C$ the circulation of $R$ around $C$.
\item 
A function $c \colon \mathscr{C}_X \To \Z$ is a \emph{feasible circulation function} if it is the circulation function of some orientation on $X$.
\end{enumerate}
\end{defn}
When the orientation is clear, we simply write $c$ rather than $c_R$. Note that any feasible circulation function sends all null directed cycles to $0$. 

Any directed cycle has a well-defined homology class in $H_1(X)$, where we regard the graph $X$ as a 1-dimensional cell complex in the standard way. Two directed cycles in the same homology class have the same circulation, and $c_R$ in fact extends to an abelian group homomorphism $H_1(X) \To \Z$.

As we will see, in general there can be many orientations on $X$ with  the same circulation function.
\begin{lem} 
\label{Lem:accessibility_same_circulation}
\label{Lem:unchanged_cycle}
Let $R,R'$ be orientations on $X$ which have the same circulation function.
Then 
\begin{enumerate}
\item the accessibility classes of $R$ are identical to the accessibility classes of $R'$; and
\item the forward cycles of $R$ are identical to the forward cycles of $R'$.
\end{enumerate}
\end{lem}

\begin{proof}
A directed cycle is forward if and only if its circulation is equal to its length. Hence $R$ and $R'$ have the same forward cycles. Two vertices lie in the same accessibility class if and only if they lie in a forward cycle. 
\end{proof}

Thus, if $C$ is a forward cycle for an orientation with circulation $c$, then $C$ is a forward cycle for every orientation with circulation $c$. In particular, every directed edge of $C$ lies in every orientation $R$ with circulation $c$. Such directed edges are ``forced" by the circulation $c$ in the following sense.
\begin{defn}[$c$-forced and $c$-forbidden directed edges]
\label{Def:forced_forbidden_edges}
Let $\vv{e} \in \vv{E}$ and let $c$ be a feasible circulation function.
\begin{enumerate}
\item
$\vv{e}$ is \emph{$c$-forced} if $\vv{e}$ belongs to every orientation of $X$ with circulation $c$.
\item 
$\vv{e}$ is \emph{$c$-forbidden} if $\vv{e}$ does not belong to any orientation of $X$ with circulation $c$.
\end{enumerate}
\end{defn}
Clearly, $\vv{e}$ is $c$-forced if and only if $-\vv{e}$ is $c$-forbidden.
The above lemma also motivates the following definition.
\begin{defn}
\label{Def:orientations_notation}
Let $c$ be a feasible circulation function on $X$.
\begin{enumerate}
\item The set of orientations on $X$ with circulation $c$ is denoted $\mathscr{R}_X^c$.
\item An \emph{accessibility class} of $c$ is an accessibility class of some (hence any) $R \in \mathscr{R}_X^c$.
\end{enumerate}
\end{defn}
Note in this definition that as $c$ is feasible, $\mathscr{R}_X^c$ is nonempty. 
When the graph $X$ is understood we simply write $\mathscr{R}^c$.

The following is Proposition 4 of  \cite{Propp}. Propp only needed it for graphs without loops; we slightly extend it to graphs with loops.
\begin{proposition}\label{Prop:Propp's_prop}
Let $c$ be a feasible circulation function on $X$, and $\vv{e}$ a directed edge of $X$. The following statements are equivalent.
\begin{enumerate}
\item	
$\vv{e}$ is $c$-forced or $c$-forbidden.
\item	
The endpoints of $\vv{e}$ belong to the same accessibility class of $c$.
\end{enumerate}
\end{proposition}

\begin{proof}
If $\vv{e}$ is not a loop, then Propp's proof from \cite{Propp} applies. If $\vv{e}$ is a directed loop, then it forms a directed cycle, which $c$ sends to $1$ or $-1$, and $\vv{e}$ is $c$-forced or $c$-forbidden accordingly. So (i) holds, and of course the endpoints of $\vv{e}$, being equal, lie in the same accessibility class.
\end{proof}

For an accessibility class $K$, we denote its complement of $K$ in the vertex set of $X$ by $\mathring{K}$.
\begin{defn}[Maximal and minimal accessibility classes] \cite{Propp} \
An accessibility class $K$ of an oriented graph $(X,R)$ is
\begin{enumerate}
\item
\emph{maximal} if, for each edge $e$ connecting an $x \in K$ to a $y \in \mathring{K}$, $R$ orients $e$ from $y$ to $x$;
\item 
\emph{minimal} if, for each edge $e$ connecting an $x \in K$ to a  $y \in \mathring{K}$, $R$ orients $e$ from $x$ to $y$.
\end{enumerate}
We also say $K$ is \emph{maximal} or \emph{minimal relative to} $R$.
\end{defn}
Thus, $K$ is maximal when $R$ orients edges from $\mathring{K}$ to $K$, and minimal when $R$ orients edges from $K$ to $\mathring{K}$.

\begin{defn}[Pushing] \label{def:pushing} \cite{Pretzel_86, Propp} \
Let $K$ be an accessibility class of an oriented graph $(X,R)$.
\begin{enumerate}
\item 
If $K$ is maximal, the operation of reversing the directed edges between $K$ and $\mathring{K}$ is called \emph{pushing down}. 
\item 
If $K$ is  minimal, the operation of reversing the directed edges between $K$ and $\mathring{K}$ is called \emph{pushing up}.
\end{enumerate}
\end{defn}
Thus, pushing down on $K$ results in a new orientation $R'$ on $X$ with respect to which $K$ is a minimal accessibility class. Similarly, pushing up on $K$ results in a new orientation with respect to which $K$ is maximal. It is not difficult to see that such $R$ and $R'$ have the same circulation function, so if $R \in \mathscr{R}^c$ then $R' \in \mathscr{R}^c$. By \reflem{accessibility_same_circulation} then $R$ and $R'$ then have the same accessibility classes.
The notion of pushing down is due to Pretzel, based on earlier work of Mosesian in the Russian literature: see \cite{Pretzel_86} and references therein.

\subsection{Lattice structure on orientations of graphs}

Again let $X$ be a graph, possibly with multiple edges and loops,
but now assume that $X$ is connected.
Let $c$ be a feasible circulation function on $X$, so $\mathscr{R}^c$ is nonempty.
We fix an accessibility class $K_0$ of $c$, which we call \emph{unpushable}; all other accessibility classes are \emph{pushable}.
\begin{defn}
\label{Def:relation_on_orientations}
Let $R,R' \in \mathscr{R}^c$.
We write $R \leqslant R'$ if there exists a sequence of orientations $R=R_0, R_1, \ldots, R_n = R'$, all in $\mathscr{R}^c$, such that each $R_{j+1}$ is obtained from $R_{j}$ by pushing up on a pushable minimal accessibility class of $c$.
\end{defn}
Here the accessibility classes of all $R_j$, including $R$ and $R'$, are all identical to the accessibility classes of $c$.

We now state Propp's theorem on orientations of graphs from \cite{Propp}. Propp only considered graphs without loops, and we provide a very slight generalisation.
\begin{theorem}[Propp \cite{Propp}]
\label{Thm:Propp_pushing0}
Let $X$ be a finite connected graph, possibly with multiple edges and loops. Let $c$ a feasible circulation function, and $K_0$ a fixed unpushable accessibility class of $\mathscr{R}^c$. Then $(\mathscr{R}^c, \leqslant)$ is a distributive lattice.

Moreover, for $R,R' \in \mathscr{R}^c$, $R'$ \emph{covers} $R$ (i.e. $R \lessdot R'$) if and only if $R'$ is obtained from $R$ by pushing up on a pushable minimal accessibility class of $c$.
\end{theorem}
We call this lattice the \emph{Propp lattice} of $X$ with respect to $c$ and $K_0$.

We briefly outline Propp's proof of this theorem, under the assumption that $X$ has no loops. Propp showed that for any feasible circulation function $c$, there exists a function $\overline{F}:\vv{E}\to [0,1]$ with the following properties:
\begin{enumerate}
\item
For any $\vv{e} \in \vv{E}$, $\overline{F}(\vv{e})+\overline{F}(-\vv{e})=1$.
\item
For any $C \in \mathscr{C}$, $\sum_{\vv{e}\in C} \overline{F}(\vv{e})=\frac{1}{2}(\abs{C}+c(C))$.
\item
If $\vv{e}$ is $c$-forced then $\overline{F}(\vv{e})=1$, and if $\vv{e}$ is $c$-forbidden then $\overline{F}(\vv{e})=0$.
\item 
If $\vv{e}$ is neither $c$-forced nor $c$-forbidden, then $0<\overline{F}(\vv{e})<1$.
\end{enumerate}

Fixing a 
vertex $v^* \in K_0$, and an
orientation $R$ with circulation $c$, Propp showed that one can then define a \emph{height function} $H_R:V\to\mathbb{R}$ 
such that
\begin{enumerate}
\item 
$H_R(v^*)=0$
\item 
For $\vv{e}=(e,v,w) \in \vv{E}$,
\begin{equation}
\label{Eqn:height_function_step}
H_R(w)-H_R(v)=
\begin{cases}
1-\overline{F}(\vv{e})   &\text{if } \vv{e}\in R,\\
-\overline{F}(\vv{e})    &\text{if } \vv{e}\notin R.
\end{cases}
\end{equation}
\end{enumerate}\
Moreover,  
every function $H:V\to\mathbb{R}$ satisfying these conditions is the height function $H_R$ of some $R\in\mathscr{R}^c$. 

Height functions $H_R$ have the property that, for all adjacent vertices $v$ and $w$, $H_R(v)=H_R(w)$ if and only if $v$ and $w$ belong to the same accessibility class. Moreover,
for any $R,R'\in\mathscr{R}^c$ and any $v\in V$, $H_R(v)-H_{R'}(v) \in\mathbb{Z}$.

One can define a partial order on height functions, and hence on orientations, in a standard way: for $R,R' \in \mathscr{R}^c$, let $R \leqslant R'$ iff $H_R (v) \leqslant H_{R'}(v)$ for all $v \in V$.

Propp proved that if $H_R$ and $H_{R'}$ are two height functions, then so are their \emph{meet} $H_R \land H_{R'}$ and their \emph{join} $H_R \lor H_{R'}$, where
\begin{align*}
(H_1\land H_2)(v) &=\min(H_1(v),H_2(v))\\
(H_1\lor H_2)(v) &=\max(H_1(v),H_2(v))
\end{align*}
These meet and join operations in fact provide the claimed  distributive lattice structure on $(\mathscr{R}^c, \leqslant)$.

\begin{proof}[Proof of \refthm{Propp_pushing0}]
The same proof as outlined above applies with a few further details. Let $\underline{X}$ be $X$ with loops removed, $\underline{\mathscr{C}}$ the set of directed cycles in $\underline{X}$, and $\underline{c}$ the restriction of $c$ to $\underline{\mathscr{C}}$. So $\underline{X}$ is a graph with a circulation $\underline{c}$ to which Propp's proof directly applies.

A directed loop $\vv{e}$ forms a directed cycle. For any orientation $R$, $\vv{e}$ is either forward or backward with respect to $R$, and $c_R(\vv{e}) = \pm 1$ accordingly. Thus $c(\vv{e}) = \pm 1$ and $\vv{e}$ is $c$-forced or $c$-forbidden accordingly. Hence we can define a function $\overline{F}:\vv{E}\to [0,1]$, defined as in Propp's proof for $\underline{X}$ and $\underline{c}$, and additionally setting $\overline{F}(\vv{e}) = 1$ or $0$ for directed loops accordingly as they are $c$-forced or $c$-forbidden. Combining Propp's proof with the definition on directed loops, the map $\overline{F}$ immediately satisfies conditions (i), (iii) and (iv) above. To see (ii), take a $C \in \mathscr{C}$, and note that it consists of a (possibly null) directed cycle $\underline{C} \in \underline{\mathscr{C}}$ , together with some number of loop edges. By Propp's proof, (ii) holds  for $\underline{C}$, and by our construction (ii) holds for loops. Both sides of the equation in (ii) are additive under concatenating cycles, so (ii) holds for $C$.

Given an orientation $R$ with circulation $c$, let $\underline{R}$ be its restriction to $\underline{X}$. We can then again define a height function $H_R \colon V \To \R$ by applying Propp's proof to $\underline{X}$ and $\underline{R}$. Condition (i) for $H_R$ then immediately holds, as does (ii) for all non-loop directed edges $\vv{e}$. When $\vv{e}$ is a directed loop at a vertex $v$, both sides of \refeqn{height_function_step} are $0$ since $\overline{F}(\vv{e}) = 1$ or $0$ accordingly as $\vv{e}$ is $c$-forced or $c$-forbidden, i.e. $\vv{e} \in R$ or $\vv{e} \notin R$. Since the height function $H_R$ coincides with the height function $H_{\underline{R}}$ for $\underline{X}$, the height functions again yield the desired  distributive lattice structure. Indeed, the distributive lattice structures on $\mathscr{R}^c_X$ and $\mathscr{R}^{\underline{c}}_{\underline{X}}$ are isomorphic.
\end{proof}

\subsection{Embedded graphs}

We will need to consider numerous embedded graphs, so although the notion is standard, we define them precisely as we need. We only consider finite graphs. So let $X$ be a finite graph, possibly disconnected, possibly with multiple edges and loops. We can realise $X$ as a 1-dimensional cell complex, which we denote $\widetilde{X}$. As $X$ is finite, $\widetilde{X}$ is compact.
\begin{defn}[Graph embedding]
\label{Def:graph_embedding}
Let $\Sigma$ be a surface.
A \emph{graph embedding} of $X$ in $\Sigma$ is a proper embedding $\phi \colon \widetilde{X} \hookrightarrow \Sigma$.
\end{defn}
Note for the purposes of this definition, $\Sigma$ need not be compact, orientable, or connected, though we will impose such conditions later on.
By \emph{proper} here we mean that only vertices of $X$ may map to boundary points of $\Sigma$; equivalently, that the interior of every arc of $X$ must lie in the interior of $\Sigma$. (Vertices of $X$ may map to boundary points or interior points of $\Sigma$.)
Under such an embedding, each vertex of $X$ maps to a distinct point of $\Sigma$, and each edge of $X$ maps to a simple arc connecting its two endpoints. Any two such arcs intersect only at common endpoints as prescribed by the incidence relations of $X$. 

The graph $U$ arising in a universe on a disc (\refdef{universe_on_disc}) is an embedded graph in the disc $D$.
Similarly, the graph $U$ arising in a Kauffman universe (\refdef{universe}) is an embedded graph in $\R^2$, also known as a \emph{plane graph}.

We use a standard notion of isotopy for embedded graphs, as in e.g. \cite{Verdiere_Mesmay_14, Ladegaillerie_84}.
\begin{defn}
\label{Def:isotopy_embedded_graphs}
Let $\phi_0, \phi_1 \colon \widetilde{X} \hookrightarrow \Sigma$ be two graph embeddings. Then $\phi_0, \phi_1$ are \emph{isotopic} if there exists a continuous family of graph embeddings $\phi_t \colon \widetilde{X} \hookrightarrow \Sigma$ for $t \in [0,1]$ from $\phi_0$ to $\phi_1$.
\end{defn}

\begin{defn}[Faces, 2-cell embedding]
\label{Def:faces}
Let $\phi \colon \widetilde{X} \hookrightarrow \Sigma$ be a graph embedding.
\begin{enumerate}
\item
The \emph{faces} of $\phi$ are the connected components of $\Sigma \setminus \phi(\widetilde{X})$.
\item 
$\phi$ is a \emph{2-cell embedding} if each face of $\phi$ is homemorphic to a disc.
\end{enumerate}
\end{defn}
The image of $\phi$ is a homeomorphic copy of $\widetilde{X}$ in $\Sigma$. In practice we abuse notation by  referring to $\widetilde{X}$ and its embedded image as $X$, referring to the images of vertices in $\Sigma$ as vertices of $X$, referring to the images of edges as edges of $X$, and referring to the faces of $\phi$ as  faces of $X$.

\refdef{faces} generalises the standard notion of face arising for plane graphs and universes, as in the introduction.

\subsection{Boundaries of faces and boundary cycles}
\label{Sec:boundaries_of_faces}

We now consider boundaries of faces in embedded graphs; again this notion is quite standard, but since there are some subtleties, we define them precisely as we need.

Let $G$ be a graph embedded on an orientable surface $\Sigma$. We assume $G$ is finite, but may be disconnected and may have multiple edges and loops.
We assume $\Sigma$ is compact, and may have boundary. 
For now we also assume that $G$ is embedded in the \emph{interior} of $\Sigma$.

Let $f$ be a face of $G$. Then $f$ is an open subsurface of $\Sigma$, and its closure in $\Sigma$ is obtained by adding finitely many boundary components. A boundary component of $f$ may consist of edges of $G$, an isolated vertex of $G$, 
or a circle of $\partial \Sigma$.

\begin{figure}
\begin{center}
\includegraphics[width=0.35\textwidth]{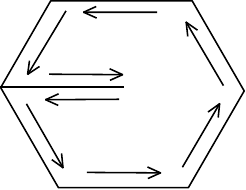}
\end{center}
    \caption{A face with boundary visiting an edge twice.}
    \label{Fig:boundary_with_repeated_edge}
\end{figure}

Consider a boundary component of $f$ consisting of edges of $G$. By walking around this boundary component, from a chosen vertex, in a chosen direction, we obtain a directed cycle $C$. Note this may involve walking along an edge twice. However, as $\Sigma$ is orientable, if $C$ visits an edge $e$ of $G$ twice, then $e$ is visited in opposite directions: see \reffig{boundary_with_repeated_edge}. Thus $C$ is a simple directed cycle (\refdef{Propp_definitions0}). However $C$ need not be vertex-simple. A boundary component consisting of an isolated vertex of $G$ can also be regarded as a null cycle, leading to the following definition.
\begin{defn}[Boundary cycle] 
\label{Def:boundary_cycle}
Let $f$ be a face of $G$.
\begin{enumerate}
\item
A \emph{boundary directed cycle} of $f$ is a simple directed cycle $C$ in $G$, obtained by walking around a boundary component of $f$, from a chosen vertex, in a chosen direction.
\item 
A \emph{boundary cycle} of $f$ is the simple (undirected) cycle $\overline{C}$ represented by a boundary directed cycle of $f$.
\end{enumerate}
\end{defn}
A null boundary cycle (directed or undirected) arises precisely when $G$ has an isolated vertex.

In general, a face $f$ may have many boundary components. Only those boundary components which lie along $G$ form boundary cycles.
A boundary component of $f$ along $G$ will in general yield many boundary directed cycles, as $C$ may start at an arbitrary vertex of $f$ and proceed in either direction around the boundary of $f$. These boundary directed cycles  form an equivalence class under $\sim$ as in \refdef{cycle}. Hence there is a unique boundary cycle $\overline{C}$ at each boundary component of $f$ which lies along $G$. 

\subsection{Boundary and elementary cycles of plane graphs}
\label{Sec:plane_boundary_elementary}

We now consider Propp's notion of \emph{elementary cycle} in \cite{Propp}. However, Propp only considers graphs $G$ which are connected and without loops, so we need a slight generalisation.

Let $G$ be a finite graph, possibly disconnected, possibly with multiple edges and loops.
We now suppose $G$ is embedded in $\R^2$, i.e. $G$ is a plane graph.

The faces of $G$ now all have genus zero. Each face has finitely many boundary components. Some faces may be nested inside each other. There is precisely one unbounded face. 

A bounded face $f$ of $G$ has a distinguished \emph{outermost} boundary component. 
Possibly after a small truncation of $f$ to avoid issues with double edges as in \reffig{boundary_with_repeated_edge}, the closure $\overline{f}$ of $f$ in $\R^2$ has boundary consisting of disjoint simple closed curves, one of which is outermost. By the Jordan curve theorem, each of these simple closed curves has an interior and exterior in $\R^2$. The outermost boundary component is distinguished by the fact that it contains $f$ in its interior.

\begin{defn}[Outer boundary, elementary cycle] \
\label{Def:outer_boundary_elementary}
\label{Def:elementary_cycle}
Let $f$ be a bounded face of $G$.
\begin{enumerate}
\item The \emph{outer boundary} of $f$ is the unique boundary component of $f$ which contains $f$ in its interior.
\item The \emph{outer boundary cycle} of $f$ is the simple (undirected) cycle $C$ of $G$ represented by any boundary directed cycle around the outer boundary of $f$. We say $C$ \emph{encircles} $f$.
\item An \emph{elementary cycle} of $G$ is an outer boundary cycle of some non-outer face of $G$.
\end{enumerate}
\end{defn}
Note that outer boundary cycles and elementary cycles are non-null.

A  directed cycle around the outer boundary of $f$ encircles $f$ in a \emph{clockwise} or \emph{counterclockwise} direction. A boundary (undirected) cycle does not have an orientation. 
Elementary cycles can thus be oriented clockwise or counterclockwise.

Each non-outer face has a unique elementary cycle encircling it, and each elementary cycle encircles a unique non-outer face. Hence non-outer faces and elementary cycles are naturally in bijection.

When $G$ is connected (as in Propp \cite{Propp}), then 
each bounded face of $G$ is homeomorphic to a disc, and the unbounded face is homeomorphic to a punctured disc. 
Each bounded face then has a single boundary component consisting of edges of $G$, yielding an elementary cycle.

\begin{lem}
\label{Lem:elementary_cycles_components}
Let $G$ be a finite plane graph (possibly with multiple edges and loops),
with connected components $G_1, \ldots, G_n$.
The set of elementary cycles of $G$ is the disjoint union of the sets of elementary cycles of the $G_j$.     
\end{lem}
In other words, each elementary cycle of a $G_j$ is an elementary cycle of $G$, and each elementary cycle of $G$ is an elementary cycle of precisely one $G_j$. 

\begin{proof}
Let $C$ be an elementary cycle of $G$. Then $C$ is the outer boundary cycle of some non-outer face $f$ of $G$. The cycle $C$ lies in some component $G_j$ of $G$, and $f$ forms part of some non-outer face $\overline{f}$ of $G_j$. Indeed, $C$ is also the outer boundary cycle of 
$\overline{f}$.
So $C$ is an elementary cycle of $G_j$. (Clearly $C$ is not an elementary cycle of any other component of $G$.) 

Conversely, suppose $C$ is an elementary cycle of $G_j$. Then $C$ is the outer boundary cycle of a non-outer face $f$ of $G_j$. This $f$ is a union of faces of $G$, and one of these faces of $G$ is adjacent to $C$, with $C$ as its outer boundary component. Thus $C$ is an elementary cycle of $G$.
\end{proof}

\subsection{Matchings, alternating paths and cycles}
\label{Sec:matchings}

We again follow Propp \cite{Propp} in the following definitions, which he considered for connected graphs without loops. Generalising them to our context involves some subtleties.

So, let $X$ be a finite graph, possibly disconnected, possibly with multiple edges and loops.  
\begin{defn}[Matching] \label{Def:matching}
A \emph{perfect matching}, or just \emph{matching}, $M$ on $X$ is a set of edges $\{e_1, \ldots, e_n\}$ of $X$ such that the $2n$ endpoints of the $e_j$ include every vertex of $X$ precisely once.
The set of all matchings on $X$ is denoted $\mathscr{M}_X$.
\end{defn}
Thus, if $X$ has a matching with $n$ edges, then $X$ has precisely $2n$ vertices. A loop can never occur in a matching. When $X$ has no loops, a matching may be defined as a set $M$ of of edges of $X$ such that each vertex belongs to precisely one edge of $M$. 

The following notion will be useful in the sequel.
\begin{defn}[Forced and forbidden edges]
\label{Def:forced_forbidden}
An edge of $X$ is
\begin{enumerate}
\item
\emph{forced} if it appears in all $M \in \mathscr{M}_X$;
\item 
\emph{forbidden} if it appears in no $M \in \mathscr{M}_X$.
\end{enumerate}
\end{defn}

\begin{defn}[Alternating paths and cycles]
\label{Def:alternating_path}
Let $M \in \mathscr{M}_X$.
A directed path, directed cycle, or undirected cycle in $X$ is \emph{alternating} relative to $M$ if it is non-null, and its edges alternately do and do not belong to $M$.
\end{defn}
Note that while undirected cycles do not have a direction, each edge does have two well-defined adjacent edges, hence speaking of edges alternately belonging and not belonging to $M$ makes sense. An alternating cycle must have positive even length.

A loop can never occur in an alternating cycle. For if a loop $e$ based at a vertex $v$ were in an alternating cycle $C$ for a matching $M$, then $e$ cannot be in $M$, so both edges in $C$ adjacent to $e$ must be in $M$, but both these edges are adjacent to $v$, so we have two distinct endpoints of edges of $M$ equal to $v$, contradicting $M$ being a matching.

An alternating cycle may be represented by a directed cycle
\begin{equation}
\label{Eqn:alternating_cycle_notation}
(e_1, v_0, v_1), \ldots, (e_{2n}, v_{2n-1}, v_{2n} = v_0)
\quad \text{where} \quad
e_j \in M \text{ for $j$ odd}, \quad
e_j \notin M \text{ for $j$ even}.
\end{equation}

Given an alternating cycle $A$ relative to a matching $M$, consider the operation of removing from $M$ all edges which lie in the cycle $A$, and adding to $M$ the edges of $A$ which do not belong to $M$. In the notation of \refeqn{alternating_cycle_notation}, we remove the $e_j$ with $j$ odd, and add the $e_j$ with $j$ even. This replaces $M$ with a different set $M'$ of edges of $M$.
In other words, regarding $A$ as a set of edges, we remove $A \cap M$ from $M$ and add $A \setminus M$, so that
\begin{equation}
\label{Eqn:matching_notation}
M' = (M \setminus A) \cup (A \setminus M) = M+A, 
\quad \text{or} \quad
M' = M \setminus \{e_j \mid j \text{ odd}\} \cup \{ e_j \mid j \text{ even} \}.
\end{equation}
Here $M+A$ denotes the the boolean addition of $M$ and $A$.

\begin{figure}
\begin{center}
\includegraphics[width=0.45\textwidth]{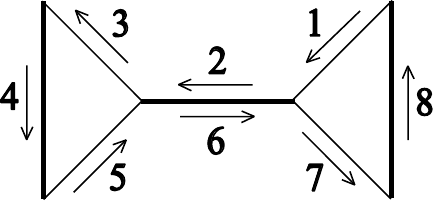}
\end{center}
    \caption{A non-vertex-simple alternating cycle. Thick edges  denote edges of the matching.}
    \label{Fig:non-vertex-simple_alternating}
\end{figure}

Note that in general an alternating cycle need not be vertex-simple: see e.g. \reffig{non-vertex-simple_alternating}. Moreover, for a general alternating cycle $A$, such as in \reffig{non-vertex-simple_alternating}, the set of edges $M'$ described above need not be a matching. The following lemma describes when $M'$ is a matching.
\begin{lem}
\label{Lem:when_twisting_works}
Let $M \in \mathscr{M}_X$ and let $A$ be a simple directed cycle of $X$, alternating relative to $M$, as in \refeqn{alternating_cycle_notation}. Let $M' = M+A$ as in \refeqn{matching_notation}.
Then $M'$ is a matching on $X$ if and only if $A$ is vertex-simple.
\end{lem}

\begin{proof}
Suppose that $A$ is not vertex-simple, so two vertices of $A$ coincide. By a cyclic permutation if necessary, we may assume that the two coincident vertices are $v_0$ and $v_k$ for some $k \neq 0$ mod $2n$ (we take all indices mod $2n$).  Note also $k \neq  \pm 1$ mod $2n$, since $v_0 = v_{\pm 1}$ would imply $e_0$ or $e_1$ being a loop, and as noted above loops cannot occur in $A$.

Each vertex $v_{2j}$ and $v_{2j+1}$ of $A$ belongs to the edge $e_{2j+1}$ of $M$. In particular, each vertex $v_j$ in $A$ belongs to an edge of the matching in $A$, which is either $e_j$ or $e_{j+1}$.

The vertex $v_0$ lies in the edge $e_1$ of $M$, which is directed $(e_1, v_0, v_1)$ in $A$. Similarly, $v_k$ lies in the edge $e_l$ of $M$, where $l$ is the odd element in $\{ k, k+1 \}$. In $A$, accordingly as $l=k$ or $k+1$, $e_l$ appears in $A$ directed as $(e_k, v_{k-1}, v_k)$ and $(e_{k+1}, v_k, v_{k+1})$. As $k \neq 0,1$ mod $2n$, we have $k,k+1 \neq 1$, hence $l \neq 1$ mod $2n$. So the directed edges $(e_l, v_{l-1}, v_l)$ and $(e_1, v_0, v_1)$ are at different positions in $A$. But as $M$ is a matching, $e_1, e_l \in M$, and $v_0$ coincides with $v_{l-1}$ or $v_l$, the edges $e_1$ and $e_l$ are the same edge. As $A$ is simple, we conclude that $(e_1, v_0, v_1)$ and $(e_l, v_{l-1}, v_l)$ form the same edge with opposite directions. In particular $k=l$, $k$ is odd, and $v_1 = v_{k-1}$.

We now consider proceeding from $e_1$ along $A$ in the direction of $e_2, e_3, \ldots$ up to $e_k$; and proceeding from $e_k$ along $A$ in the direction of $e_{k-1}, e_{k-2}, \ldots$ up to $e_1$. Note that these directed paths commence from $e_1 = e_k$ in the same direction. We claim these directed paths cannot coincide. Indeed, if they coincide for $m+1$ edges, where $k=2m+1$, then we have equalities of directed edges
\[
(e_1, v_0, v_1) = (e_k, v_k, v_{k-1}), \;
(e_2, v_1, v_2) = (e_{k-1}, v_{k-1}, v_{k-2}), \;
\ldots, \;
(e_{m+1}, v_{m}, v_{m+1}) = (e_{m+1}, v_{m+1}, v_m),
\]
so that $v_m = v_{m+1}$ and $e_{m+1}$ is a loop, contradicting the fact that loops cannot occur in alternating cycles.

Thus, there exists a least positive integer $c$ such that $e_{1+c} \neq e_{k-c}$. Indeed, this $c$ satisfies $1 \leq c < m$, so $1+c < k-c$. Then $(e_{c}, v_{c-1}, v_c) = (e_{k-c+1}, v_{k-c+1}, v_{k-c})$, but $e_{1+c} \neq e_{k-c}$. Thus 
$v_c = v_{k-c}$. As $k$ is odd, $1+c = k-c$ mod $2$. Thus either both $e_{1+c}, e_{k-c} \in M$ or both $e_{1+c}, e_{k-c} \in M'$. But these two edges have in common the vertex $v_c = v_{k-c}$. So $M$ and $M'$ cannot both be matchings; as $M$ is a matching, then $M'$ is not.

For the converse, suppose $A$ is vertex-simple. If any undirected edges of $A$ coincide, then $A$ has length $2$ and consists of the same edge traversed in both directions, in which case $e_1 = e_2$, so both $e_1, e_2 \in M$ or $e_1, e_2 \notin M$, contradicting $A$ being alternating. Thus all $2n$ edges of $A$ are distinct, and since $M$ matches the $2n$ distinct vertices of $A$ in pairs, so too does $M'$.
\end{proof}

\begin{defn}[Twisting]
\label{Def:twisting}
The operation of replacing a matching $M$ with a matching $M' = M+A$  as above, where $A$ is a vertex-simple alternating cycle relative to $M$, is called \emph{twisting} $M$ at $A$.
\end{defn}

\subsection{Matchings on plane bipartite graphs, positive and negative cycles}
\label{Sec:matchings_plane_bipartite}

We again follow Propp \cite{Propp} in the following definitions, which generalise to our context without difficulty.

Let $G$ be a finite plane graph, possibly disconnected, possibly with multiple edges and loops.
As in \refsec{plane_boundary_elementary}, $G$ has a well-defined elementary cycles. 
Let $M$ be a matching on $G$. 
We consider elementary cycles which are alternating with respect to $M$.

We now additionally assume that $G$ is \emph{bipartite}, with vertices coloured black and white. Then every cycle in $G$ has even length, and $G$ has no loops.
The matching $M$ provides a bijection between the black and white vertices of $G$.
Consider an alternating elementary cycle $A$ relative to $M$ encircling a face $f$.
\begin{lem}
\label{Lem:alternating_elementary_cycles_super-simple}
Let 
$M \in \mathscr{M}_G$, and let $A$ be an alternating elementary cycle of $G$ relative to $M$. Then $A$ is vertex-simple and all its edges are distinct.
\end{lem}

\begin{proof}
Represent $A$ with a directed cycle as in \refeqn{alternating_cycle_notation}. If $A$ contains a repeated edge $e$ then, as $A$ is simple, $e$ is visited twice, once in each direction. As $G$ is bipartite, one visit proceeds from a white to black and one visit from a black to white vertex. Thus the two edges $e_j, e_k$ of $A$ visiting $e$  have $j,k$ of opposite parity. But $e_j \in M$ precisely when $j$ is odd, so $e$ is simultaneously in and not in $M$, a contradiction.

Thus $A$ has all distinct edges. The proof of \reflem{when_twisting_works} shows that if two vertices of $A$ coincide, then two edges coincide also. Thus $A$ has distinct vertices.
\end{proof}
Thus, by \reflem{when_twisting_works} we can twist a matching $M \in \mathscr{M}_G$ at an alternating elementary cycle $A$ and obtain another matching $M'=M+A$.

Combining the structures of bipartite colouring and alternating matching, we endow elementary cycles with orientations as follows.
\begin{defn}[Positive, negative cycle] \cite{Propp}
\label{Def:pos_neg_cycle}
Let $M \in \mathscr{M}_G$, and let $A$ be an alternating elementary cycle relative to $M$ encircling a face $f$.
Direct the edges of $A$ that belong to $M$ (which we denote $A \cap M$) from their black vertices to their white vertices. These edges then either all encircle $f$ in a clockwise direction, or all encircle $f$ in a counterclockwise direction.
\begin{enumerate}
\item 
$A$ is a \emph{positive cycle} relative to $M$ if $A \cap M$ 
encircles $f$ in the counterclockwise direction.
\item 
$A$ is a \emph{negative  cycle} relative to $M$ if $A \cap M$
encircles $f$ in the clockwise direction.
\end{enumerate}
\end{defn}

If we twist a matching $M$ at a positive cycle $A$ and obtain a matching $M'$, then $A$ becomes negative relative to $M'$. Similarly, if $A$ is negative relative to $M$, then $A$ is positive relative to $M'$. 
\begin{defn}[Twisting up and down] \cite{Propp}
\label{Def:twisting_up_down}
The operation of twisting a matching $M$ at an alternating elementary cycle $A$ to obtain a matching $M'$ is called 
\begin{enumerate}
\item
\emph{twisting down}, if $A$ is positive relative to $M$ (hence negative relative to $M'$);
\item 
\emph{twisting up}, if $A$ is negative relative to $M$ (hence  positive relative to $M'$).
\end{enumerate}
We say $M'$ is obtained from $M$ by twisting up/down at $A$, or at $f$, where $f$ is the face encircled by $A$.
\end{defn}
Thus twisting down makes a positive cycle negative, and twisting up makes a negative cycle positive.

Note that twisting up and down can only happen at a face $f$ which is \emph{non-degenerate}, in the sense that the elementary cycle along its outer boundary is vertex-simple. Otherwise, by \reflem{alternating_elementary_cycles_super-simple}, twisting will not yield a matching.

As $G$ has only finitely many faces, it has only finitely many elementary cycles, and hence from a matching $M$ there are only finitely many twisting up or down operations.
Moreover, given the two matchings $M,M'$ related by a twisting up or down, one can recover the elementary cycle involved from the change between the matchings, and also the direction (up or down).

\subsection{Lattice structure on matchings}
\label{Sec:lattices_matchings}

Again let $G$ be a finite bipartite plane graph (hence without loops), possibly disconnected, possibly with multiple edges. Recall the set of matchings on $G$ is denoted $\mathscr{M}_G$. Let $f_0$ be the unbounded face of $G$, so that the bounded faces are precisely those other than $f_0$. 

\begin{defn} \cite{Propp}
\label{Def:relation_on_matchings1}
Define a relation $\leqslant$ on $\mathscr{M}_G$ by $M \leqslant M'$ if there exists a sequence $M=M_0, \ldots, M_n = M'$ in $\mathscr{M}_G$ for some $n \geq 0$, such that each $M_{j+1}$ is obtained from $M_{j}$ by twisting up at a bounded face.
\end{defn}
In other words, $M \leqslant M'$ if $M'$ is obtained from $M$ by a sequence of twisting up operations.  The trivial sequence, when $n=0$, means that $M \leqslant M$.

Propp in \cite{Propp} showed that this $\leqslant$ provides a distributive lattice structure, under two further assumptions: that $G$ is connected, and that each edge of $G$ belongs to some matchings but not others, i.e. is neither forced nor forbidden in the sense of \refdef{forced_forbidden}. We state his theorem and then discuss these assumptions.
\begin{theorem}[Propp \cite{Propp}] 
\label{Thm:Propp_matching}
Let $G$ be a finite connected plane bipartite graph (possibly with multiple edges). 
Suppose that $G$ has no forced or forbidden edges. Then $(\mathscr{M}_G, \leqslant)$ is a distributive lattice.

Moreover, for any two matchings $M,M' \in \mathscr{M}_G$, $M$ is covered by $M'$ (i.e. $M \lessdot M'$) if and only if $M'$ is obtained from $M$ by twisting up at a bounded face.
\end{theorem}
We call this lattice the \emph{Propp lattice} of $G$.

We need this result for disconnected $G$, and moreover, without the assumption of no forced or forbidden edges. In \cite[example 28]{Propp}, Propp discusses these issues, noting that \refthm{Propp_matching} fails without the  assumption, but that a generalisation to disconnected graphs is possible, and that the assumption can be effectively removed since there exists a (possibly trivial) subgraph of $G$ whose connected components satisfy the assumption. We consider such matters in detail in \refsec{reductions} where we define a notion of \emph{reduction}, removing edges not in any matching. Then in \refsec{lattice_matchings_reduction} we prove a generalisation of this theorem, for possibly disconnected graphs, without the assumption, as \refprop{finite_bipartite_plane_lattice}.

As discussed in \refsec{string_universes}, a plane graph can be regarded as an embedded graph on $S^2$ by one-point compactification. The above theorem can then be restated for $G$ embedded in $S^2$, with the bounded face replaced by an arbitrarily chosen face on which twisting down is prohibited.

We briefly mention an outline of Propp's proof of \refthm{Propp_matching}, since we apply similar ideas later in \refsec{positive_genus} to prove \refthm{main_thm_2}. Propp proved \refthm{Propp_matching} by applying \refthm{Propp_pushing0} to a dual $G^\perp$ of $G$. This dual $G^\perp$ has a vertex for each face of $G$, and an edge $e^\perp$ for each edge $e$ of $G$, such that $e$ and $e^\perp$ intersect in a single point. Although $G^\perp$ is a connected plane graph, its embedding into $\R^2$ is not unique. 
Choosing an embedding arbitrarily,
we then have an elementary cycle of $G^\perp$ encircling each face of $G^\perp$, i.e. encircling each vertex of $G$.

The dual $G^\perp$ has a natural \emph{standard orientation} $R^0$, orienting each edge of $G^\perp$ such that every elementary cycle of $G^\perp$ encircling a black (resp. white) vertex of $G$ is oriented clockwise (resp. counterclockwise). 
Let $v_0^{\perp}$ be the vertex of $G^\perp$ dual to the unbounded face $f_0$ of $G$. 

For each $M \in \mathscr{M}_G$, we also have a natural orientation $R^M$ of $G^\perp$, called its \emph{prescribed orientation}, obtained from $R^0$ by reversing orientations on the edges of $G^\perp$ dual to the edges of $M$. 
It can be shown that all prescribed orientations of $G$ have the same circulation $c$, independent of the matching $M$. Recall (\refdef{orientations_notation}) $\mathscr{R}^c_{G^\perp}$ denotes the set of orientations of $G^\perp$ with circulation $c$.

Propp proved that the map $\mathscr{M}_G \To \mathscr{R}^c_{G^\perp}$ sending the matching $M \mapsto R^M$ is a bijection.
Moreover, the positive (resp. negative) cycles of $G$ relative to $M$ correspond precisely to the maximal (resp. minimal) accessibility classes of $(G^\perp, R^M)$.
Twisting up operations on $M$ at faces other than $f_0$ correspond bijectively to pushing up operations on $R^M$ at accessibility classes other than the singleton accessibility class $\{v_0^\perp\}$.
Applying \refthm{Propp_pushing0} to $G^\perp$ yields \refthm{Propp_matching}.

\subsection{Reductions}
\label{Sec:reductions}

In order to generalise \refthm{Propp_matching}, we introduce a notion of \emph{reduction}. Let $X$ be a finite graph, possibly disconnected, possibly with multiple edges and loops.
\begin{defn}
\label{Def:reduction}
The \emph{reduction} $X_0$ of $X$ is the graph with the same vertex set as $X$, obtained from $X$ by removing each forbidden edge.
\end{defn}
Since a loop cannot appear in a matching, $X_0$ has no loops.  If $X$ has no matchings, i.e. $\mathscr{M}_X = \emptyset$, then $X_0$ has no edges.

\begin{lem}
\label{Lem:remove_no_matchings}
Let $X_0$ be the reduction of $X$.
Then we have the following.
\begin{enumerate}
\item 
The matchings of $X$ are precisely the matchings of $X_0$, i.e. $\mathscr{M}_X = \mathscr{M}_{X_0}$. 
\item 
If $\mathscr{M}_X \neq \emptyset$, then each connected component $Y$ of $X_0$ satisfies precisely one of the following two possibilities.
\begin{enumerate}
\item $Y$ consists of a single edge $e$ between two distinct vertices, and $e$ is forced in $X$.
\item Each edge in $Y$ is neither forced nor forbidden in $X$.
\end{enumerate}
\end{enumerate}
\end{lem}
Thus, every forced edge of $X$ becomes an isolated component of $X_0$. Every forbidden edge of $X$ is removed to obtain $X_0$. And every edge which is neither forced nor forbidden matchings remains so in some connected component of $X_0$. (This is a hypothesis of Propp's \refthm{Propp_matching}.)

The hypothesis that $\mathscr{M}_X \neq \emptyset$ in (ii) is necessary. If $X$ has no matchings, then $X_0$ has no edges, and each connected component $Y$ of $X_0$ is a single point, to which neither (a) nor (b) applies.

\begin{proof}
As $X$ and $X_0$ have the same vertex sets, and every edge that appears in a matching of $X$ is retained in $X_0$, then every matching on $X$ is also a matching on $X_0$; and as the edges of $X_0$ are a subset of the edges of $X$, every matching of $X_0$ is a matching of $X$. 
This gives (i).  

Let $e$ be a forced edge of $X$, with endpoints $v,w$. Then in any matching, the vertices $v$ and $w$ are paired by $e$. So if $e'$ is an edge other than $e$ incident with $v$ or $w$ then $e'$ is forbidden. Thus $e$ and $v,w$ form a component of type (a) in $X_0$.

Now consider a connected component $Y$ of $X_0$, which is not of type (a), and let $e$ be an edge of $Y$. Then $e$ lies in some matchings of $X$ (as it has not been removed), but does not lie in all matchings (otherwise it would arise in a component of type (a)). Thus $Y$ is of type (b).
\end{proof}

\subsection{Lattice of matchings on disconnected bipartite graphs}
\label{Sec:matchings_disconnected}
\label{Sec:lattice_matchings_reduction}

We now generalise Propp's \refthm{Propp_matching} to show that $\mathscr{M}_G$ has a distributive lattice structure, even when $G$ is disconnected and may contain forced and forbidden edges.

We will need the following lemma, which applies to components of type (a) in \reflem{remove_no_matchings}.
\begin{lem}
\label{Lem:single_edge_lattice}
Suppose $G$ is a plane bipartite graph consisting of a single edge between two distinct vertices. Then $\mathscr{M}_G$, with the relation $\leqslant$ of \refdef{relation_on_matchings1} is a distributive lattice.
\end{lem}

\begin{proof}
Indeed, $\mathscr{M}_G$ is a singleton set, consisting of a single matching $M$. Since $G$ has no bounded  faces, it has no elementary cycles, hence no twisting operations. The only sequence of twisting down operations is the vacuous one. Thus $\mathscr{M}_G$ forms a trivial distributive lattice.
\end{proof}

Now, let $G$ be a finite bipartite plane graph, possibly disconnected, possibly with multiple edges, and let $G_0$ be its reduction. 
Let the connected components of $G_0$ be $G_1, \ldots, G_n$.
All of $G$, $G_0$, and each $G_j$ for $1 \leq j \leq n$, are finite bipartite plane graphs, so have well-defined sets of matchings and relations $\leqslant$, as in \refsec{matchings} to \refsec{lattices_matchings}. 

By \reflem{remove_no_matchings}(i) we have equalities of sets $\mathscr{M}_G = \mathscr{M}_{G_0}$. Moreover, a matching on $G_0$ is equivalent to a matching $M_j$ on each $G_j$ for $1 \leq j \leq n$. We thus have a bijection of sets
\begin{equation}
\label{Eqn:product_bijection}
\mathscr{M}_{G_0} \cong \mathscr{M}_{G_1} \times \cdots \times \mathscr{M}_{G_n}.
\end{equation}
This bijection of \refeqn{product_bijection} holds even if $G$ has no matchings. In this case, some $G_j$ has no matchings, so the left hand side is the empty set, as is one of the factors on the right, and hence the product on the right hand side is also empty.

If $G$ has no matchings, then $G_0$ has no edges, and $\mathscr{M}_{G_0} = \emptyset$, which we can regard as a trivial distributive lattice. If $G$ has at least one matching, then each $G_j$ for $1 \leq j \leq n$ is of type (a) or (b) as in \reflem{remove_no_matchings}(ii). By \reflem{single_edge_lattice} in case (a), and \refthm{Propp_matching} in case (b), each $(\mathscr{M}_{G_j}, \leqslant)$ for $1 \leq j \leq n$ is a distributive lattice. Hence their product also obtains a distributive lattice structure, using \refdef{product_orders}. Thus both sides of \refeqn{product_bijection} have relations $\leqslant$, and the right hand side has a distributive lattice structure.

We now show that the bijection of \refeqn{product_bijection} extends to an isomorphism of distributive lattices, generalising \refthm{Propp_matching} as follows.
\begin{prop}
\label{Prop:finite_bipartite_plane_lattice}
Let $G$ be a finite bipartite plane graph, possibly disconnected, possibly with multiple edges, and let $G_0$ be its reduction. Let the connected components of $G_0$ be $G_1, \ldots, G_n$.
Then the bijection of sets \refeqn{product_bijection} extends to an isomorphism of distributive lattices
\[
\mathscr{M}_{G_0} \cong \mathscr{M}_{G_1} \times \cdots \times \mathscr{M}_{G_n},
\]
where the relation $\leqslant$ on $\mathscr{M}_{G_0}$ and each $\mathscr{M}_{G_j}$ arises from \refdef{relation_on_matchings1}, and the relation $\leqslant$ on $\mathscr{M}_{G_1} \times \cdots \times \mathscr{M}_{G_n}$ arises from \refdef{product_orders}.

Moreover, for any two matchings $M,M' \in \mathscr{M}_{G_0}$, $M$ is covered by $M'$ (i.e. $M \lessdot M'$) if and only if $M'$ is obtained from $M$ by twisting up at a bounded face of $G_0$.
\end{prop}

Since $\mathscr{M}_G = \mathscr{M}_{G_0}$ as sets, we can regard \refprop{finite_bipartite_plane_lattice} as providing a distributive lattice structure on $\mathscr{M}_G$, generalising the restricted case of \refthm{Propp_matching}, which we refer to as the \emph{Propp lattice} 
of $G$.

\begin{proof}
We first treat the trivial case where $\mathscr{M}_G = \emptyset$. Then $G_0$ is a finite set of isolated vertices, each $G_j$ is a single vertex, and all matching sets are empty. The desired isomorphism is then an isomorphism of vacuous lattices. We may henceforth assume that $\mathscr{M}_{G} = \mathscr{M}_{G_0} \neq \emptyset$, and hence all $\mathscr{M}_{G_j} \neq \emptyset$. Throughout this proof, $j$ denotes an index between $1$ and $n$.

By \reflem{elementary_cycles_components}, $G_0$ and the $G_j$ have the same set of elementary cycles on which to perform twisting down operations. Thus a twisting up operation on a matching of $G_0$ is equivalent to a twisting up operation on a matching $M_j$ of some $G_j$. More precisely, suppose $M \in \mathscr{M}_{G_0}$ corresponds to $(M_1, \ldots, M_n) \in \mathscr{M}_{G_1} \times \cdots \times \mathscr{M}_{G_n}$. Then a twisting up operation on $M$ yields $M' \in \mathscr{M}_{G_0}$ corresponding to $(M'_1, \ldots, M'_n) \in \mathscr{M}_{G_1} \times \cdots \times \mathscr{M}_{G_n}$, where $M'_j$ is obtained from $M_j$ by a twisting up on $G_j$, for a unique $j$, and $M'_i = M_i$ for all $i \neq j$. Conversely, a twisting down operation on $M_j$, resulting in $M'_j$, yields a twisting down from $M$ to $M'$, corresponding to $(M'_1, \ldots, M'_n)$ where all $M'_i = M_i$ for $i \neq j$.

Suppose $M, M' \in \mathscr{M}_{G_0}$, corresponding to $(M_1, \ldots, M_n), (M'_1, \ldots, M'_n) \in \mathscr{M}_{G_1} \times \cdots \times \mathscr{M}_{G_n}$, satisfy $M \leqslant M'$. Then by \refdef{relation_on_matchings1}, there is a sequence $M = M^0, M^1, \ldots, M^m = M'$ in $\mathscr{M}_{G_0}$, such that each $M^{i+1}$ is obtained from $M^{i}$ by twisting up on some negative cycle in $G$. Each such twisting up amounts to a twisting up on a negative cycle in  some $G_j$. Thus after performing all the twisting down operations, and arriving at $M'$, we have $M_1 \leqslant M'_1$, $M_2 \leqslant M'_2$, $\ldots$, $M_n \leqslant M'_n$, and hence $(M_1, \ldots, M_n) \leqslant (M'_1, \ldots, M'_n)$ in the distributive lattice $\mathscr{M}_{G_1} \times \cdots \times \mathscr{M}_{G_n}$.

Conversely, suppose $(M_1, \ldots, M_n) \leqslant (M'_1, \ldots, M'_n)$ in the distributive lattice $\mathscr{M}_{G_1} \times \cdots \times \mathscr{M}_{G_n}$, with corresponding matchings $M,M' \in \mathscr{M}_{G_0}$. Then each $M'_j$ is obtained from $M_j$ by a sequence of twisting up operations on negative cycles of $G_j$. Thus $(M'_1, \ldots, M'_n)$ is obtained from $(M_1, \ldots, M_n)$ by a sequence of twisting up operations on negative cycles on various $G_j$. This sequence of twisting up operations can also be performed on  corresponding matchings in $G_0$, by changing the matching on the appropriate $G_j$ at each step, and leaving all the matchings on the $G_i$ for $i \neq j$ unchanged at that step. Hence $M \leqslant M'$.

Thus the bijection \refeqn{product_bijection} preserves relations $\leqslant$. As $\mathscr{M}_{G_1} \times \cdots \times \mathscr{M}_{G_n}$ has a distributive lattice structure, then we have an isomorphism of distributive lattices as desired.

As noted after \refdef{product_orders}, in the distributive lattice structure on the product $\mathscr{M}_{G_1} \times \cdots \times \mathscr{M}_{G_n}$, we have $(M_1, \ldots, M_n) \lessdot (M'_1, \ldots, M'_n)$ if and only if $M_j \lessdot M'_j$ for a unique $j$, and $M_i = M'_i$ for all $i \neq j$. The component $G_j$ involved cannot be of type (a) of \reflem{remove_no_matchings}, hence must be of type (b), so by \refthm{Propp_matching}, $M'_j$ is obtained from $M_j$ by twisting up. A twisting up on a single $G_j$ is equivalent to a twisting up on $G_0$. Thus for $M,M' \in \mathscr{M}_{G_0}$, $M \lessdot M'$ if and only if $M'$ is obtained from $M$ by twisting up.
\end{proof}

\section{Multiverses in general}
\label{Sec:multiverses_general}

\subsection{Multiverse surface, interior and exterior}
\label{Sec:multiverse_surface}

As discussed in \refsec{string_universes}, Kauffman universes are equivalent to universes on discs. Defined in \refdef{universe_on_disc}, these are graphs embedded on discs with various properties and decorations. Kauffman's formal knot theory repeatedly uses a notion of ``exterior" as in the Jordan curve theorem, for instance in choosing starred faces adjacent to the exterior. More generally, when a compact orientable genus zero surface is  embedded in $\R^2$, it obtains an orientation, and one of the the boundary components is distinguished as outermost. A simple closed curve on the surface then obtains an interior and exterior. This motivates the following definitions.
\begin{defn}[Multiverse surface]
\label{Def:plane_surface}
A \emph{multiverse surface} $(\Sigma, \partial \Sigma_0)$ is a  compact connected oriented surface $\Sigma$, together with a distinguished boundary component $\partial \Sigma_0$, called the \emph{outer} boundary.
\end{defn}
When the distinguished boundary component is understood or irrelevant, we simply write $\Sigma$ for a multiverse surface. By definition, a multiverse surface $\Sigma$ has nonempty boundary. If $\Sigma$ has a single boundary component, it must be the outer boundary. 

\begin{defn}[Interior, exterior of curve]
\label{Def:interior_exterior}
Let $(\Sigma, \partial \Sigma_0)$ be a multiverse surface, and let $\gamma$ a separating simple closed curve in the interior of $\Sigma$. The \emph{exterior} (resp. \emph{interior}) of $\gamma$ is the component of $\Sigma \setminus \gamma$ containing $\partial \Sigma_0$ (resp. not containing $\partial \Sigma_0$).
\end{defn}

When a graph is embedded (\refdef{graph_embedding}) entirely in the interior of a multiverse surface $\Sigma$, similar considerations apply. 
As the graph is disjoint from $\partial \Sigma$, there is a unique face (\refdef{faces}) adjacent to each boundary component.
\begin{defn}[Outer face] 
\label{Def:outer_face}
\label{Def:outer_face_of_spine}
Let $G$ be a graph embedded in the interior of a multiverse surface $(\Sigma, \partial \Sigma_0)$.
The \emph{outer} face of $G$ is the face containing $\partial \Sigma_0$.
\end{defn}

In general, a face of a graph embedded on a multiverse surface may have multiple boundary components.

\subsection{Multiverse graphs}

In \refdef{universe_on_disc} we defined a universe on a disc. A multiverse graph involves a similar graph embedding on a multiverse surface.
\begin{defn}[Multiverse graph]
\label{Def:plane_multiverse_graph}
A \emph{multiverse graph} is a finite graph $U$, possibly disconnected, possibly with loops and multiple edges, embedded in a multiverse surface $(\Sigma, \partial \Sigma_0)$. Each vertex of $U$ has degree $1$ or $4$, with the degree $1$ vertices on $\partial \Sigma$, and the vertices of degree $4$ in the interior of $\Sigma$.
\end{defn}
Note that by \refdef{graph_embedding} of a graph embedding, which includes a properness requirement, the interior of every arc lies in the interior of $\Sigma$.

A multiverse graph may be embedded entirely in the interior of $\Sigma$, in which case all vertices have degree $4$. It may also have all its vertices on $\partial \Sigma$, in which case all vertices have degree $1$.

\begin{defn}[Vertices, edges and faces]
\label{Def:vef_labels}
Let $U$ be a multiverse graph on a multiverse surface $(\Sigma, \partial \Sigma_0)$. Consider its vertices, edges and faces.
\begin{enumerate}
\item 
A \emph{boundary vertex} is a vertex of degree $1$. The number of boundary vertices is denoted $V_\partial$.
\item 
An \emph{interior vertex} is a vertex of degree $4$. The number of interior vertices is denoted $V_{int}$.
\item 
A \emph{free edge} is an edge of which has an endpoint on $\partial \Sigma$.
\item 
An \emph{interior edge} of $U$ is an edge which has both endpoints in the interior of $\Sigma$.
\item 
The number of faces of $U$ is denoted $F$.
\end{enumerate}
\end{defn}
As the terminology suggests, boundary vertices lie on $\partial \Sigma$ and interior vertices lie in the interior of $\Sigma$.  Boundary vertices can lie on the distinguished outer boundary component $\partial \Sigma_0$ or other boundary components. A free edge can have one or both endpoints on $\partial \Sigma$, which may may not include the outer boundary.

Since the sums of degrees of vertices must be even we have the following.
\begin{lem}
\label{Lem:multiverse_boundary_vertices_even}
In a multiverse graph, $V_\partial$ is even. We write $V_\partial = 2N$. \qed
\end{lem}
A multiverse graph on $\Sigma$ can be regarded as a collection of ``strings" on $\Sigma$, as a knot diagram with flattened crossings. Proceeding directly across each degree-4 vertex from edge to opposite edge, $U$ can be regarded as a union of closed curves (``closed strings"), and arcs between vertices on $\partial \Sigma$ (``open strings). Double counting open strings by their endpoints, $N$ is the number of open strings.

\subsection{Definition of multiverse}
\label{Sec:multiverse_defn}

Having defined a multiverse graph, to obtain a structure analogous to \refdef{universe_on_disc}, it remains to specify starred faces. Following \refdef{universe_on_disc}, we require them to be adjacent to the outer boundary. 

\begin{defn}[Multiverse] 
\label{Def:multiverse}
\label{Def:multiverse_on_surface_with_genus}
A \emph{multiverse} is a triple $(U, \Sigma, \mathscr{F})$ where
\begin{enumerate}
\item $\Sigma$ is a multiverse surface with outer boundary $\partial \Sigma_0$;
\item 
$U$ is a multiverse graph on $\Sigma$; 
\item 
$\mathscr{F}$ is a set of $F-V_{int}$ distinct faces of $U$ adjacent to $\partial \Sigma_0$, called \emph{starred faces}.
\end{enumerate}
\end{defn}
As with Kauffman universes, in diagrams we mark starred faces with stars, and refer to the other faces as \emph{unstarred}.  When $\Sigma$ has genus $0$, we call $(U, \Sigma, \mathscr{F})$ a \emph{planar multiverse}. When $U$ is 2-cell embedded in $\Sigma$, we call $(U, \Sigma, \mathscr{F})$ a \emph{2-cell embedded multiverse}.
When $\mathscr{F}$ or $\Sigma$ are understood or not relevant, we refer to a multiverse as $U$.

Part (iii) of the definition requires that $F \geq V_{int}$, and moreover, that there be at least $F-V_{int}$ distinct faces adjacent to the outer boundary. These conditions are by no means automatically satisfied, so condition (iii) restricts the possible graphs which may arise. We discuss these constraints further in \refsec{topology_multiverses}. 

\reffig{exampleofmultiverse} shows two examples of planar multiverses on a disc (the boundary of the disc is not shown). The left example has $U$ connected with $V_\partial = 2N = 8$, hence $N=4$ open strings. There are no closed strings, and we have $V_{int} = 11$, $F = 16$, hence $5$
stars. The right example has $U$ disconnected, with loops and multiple edges. It has $V_{\partial} = 2N = 10$, hence $N=5$ open strings, and  $1$ closed string. It also has $V_{int} = 9$, $F=15$, hence $6$ 
stars.

When $\Sigma$ is a disc, $U$ is connected, and $N=1$, one can show (see \reflem{Euler_char_arg}) that $F - V_{int} = 2$, giving 2 starred faces, and \refdef{multiverse} reduces to \refdef{universe_on_disc} of a universe on a disc.

\subsection{States and transposition contours}
\label{Sec:states_contours}

By design, the number of unstarred faces in a multiverse is equal to the number of interior vertices.
Hence \refdef{state_universe} of a state generalises immediately as follows, as suggested for universes on a disc in \refsec{string_universes}.
\begin{defn}[State]\label{Def:state}
A \emph{state} of a multiverse $(U, \Sigma, \mathscr{F})$ is a choice of corner at each interior vertex of $U$, so that each unstarred face is chosen precisely once.
The set of states of $U$ is denoted $\mathscr{S}_U$.
\end{defn}
Again, we denote states by placing a marker in each chosen corner.
There may be corners at which markers always or never arise.
\begin{defn}
\label{Def:forced_forbidden_corner}
Let $(U, \Sigma, \mathscr{F})$ be a multiverse. Let $\alpha$ be a corner at an interior vertex of $U$.
\begin{enumerate}
\item 
$\alpha$ is \emph{forced} if every state of $U$ places a marker at $\alpha$.
\item 
$\alpha$ is \emph{forbidden} if no state of $U$ places a marker at $\alpha$.
\end{enumerate}
\end{defn}
Clearly any corner in a starred face is forbidden. The same terminology as \refdef{forced_forbidden}, considering matchings on a graph, is intentional, and in \reflem{forced_forbidden_corners_edges} we justify it.

We will define two distinct generalisations of Kauffman transpositions (\refdef{transposition}), namely the \emph{plane transpositions} of \refthm{main_thm_1} and the \emph{surface transpositions} of \refthm{main_thm_2}. 
Both generalisations involve certain simple closed curves which we call \emph{transposition contours}, or just \emph{countours}.

To describe a contour, consider a state $S$ of a multiverse $(U, \Sigma, \mathscr{F})$, and consider $n \geq 1$ distinct interior (i.e. degree-4) vertices of $U$. Label these vertices $v_j$ over $j \in \Z/n\Z$. Let $\alpha_j$ be the corner at $v_j$ chosen by $S$, and let $F_j$ be the (necessarily unstarred) face of $U$ containing $\alpha_j$. Thus we have $n$ distinct faces $F_j$ over $j \in \Z/n\Z$. See \reffig{alternating_cycle} (left) for an illustration.

\begin{figure}
    \centering
    \includegraphics[width=0.45\textwidth]{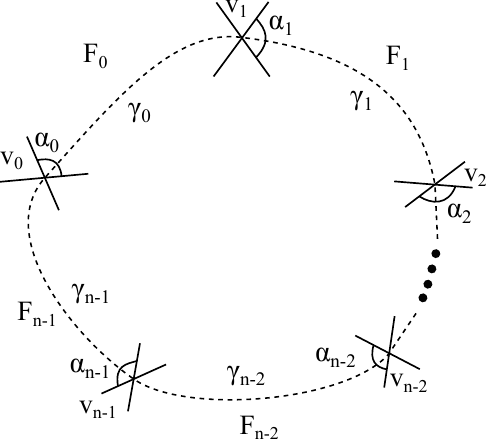}
    \includegraphics[width=0.45\textwidth]{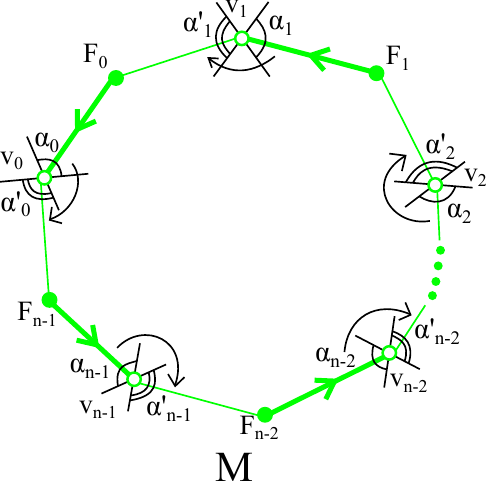}
    \caption{Left: A transposition contour. Right: The corresponding alternating cycle. Edges of $U$ are shown in black. The reduced spine is shown in green. Edges in the matching are drawn thick.}\label{Fig:alternating_cycle}
\end{figure}

\begin{defn}
\label{Def:transposition_contour}
Let $(U, \Sigma, \mathscr{F})$ be a multiverse, and $S$ a state.
An \emph{$n$-transposition contour}, or just \emph{transposition contour} for $S$ is a simple closed curve $\gamma$ on $\Sigma$, 
for which there exist $v_j, \alpha_j, F_j$ as above,
and $n \geq 1$ arcs $\gamma_j$ over $j \in \Z/n\Z$, where each $\gamma_j$:
\begin{enumerate}
\item has endpoints $v_j$ and $v_{j+1}$;
\item has interior lying in $F_j$;
\item is incident to $v_j$ through $\alpha_j$;
\end{enumerate}
such that $\gamma$ is made by joining the arcs $\gamma_j$.
The $v_j$ are the \emph{vertices} and the $F_j$ are the \emph{faces} of $\gamma$.
\end{defn}
Note that we allow $n=1$ in the above, in which case $\gamma = \gamma_1$ is a loop at a single vertex $v_1$, adjacent to the face $F_1$ through two distinct corners.

From a state and a transposition contour, we can obtain another state as follows. Let $\alpha'_{j+1}$ be the corner through which $\gamma_j$ is incident to $v_{j+1}$. Then each $F_j$ is incident to $v_j$ through the corner $\alpha_j$, and to $v_{j+1}$ through the corner $\alpha'_{j+1}$. Thus, at each vertex $v_j$ the faces $F_{j-1}$ and $F_j$ are incident through the corners $\alpha'_j, \alpha_j$ respectively. Replacing each $\alpha_j$ with $\alpha'_j$ turns the state $S$ into another state $S'$.
\begin{defn}[Contour $n$-transposition] \label{Def:partial_transposition}
Let $\gamma$ be a contour for a state $S$ on a multiverse $(U, \Sigma, \mathscr{F})$ as above. Then the operation of replacing each marker $\alpha_j$ in $S$ with $\alpha'_j$ to obtain a state $S'$ is called \emph{contour $n$-transposition}, or just \emph{contour transposition} of $S$ along $\gamma$.
\end{defn}

A Kauffman transposition (\refdef{transposition}) is a specific type of contour $2$-transposition. We may draw a contour as in \reffig{Kauffman_as_contour_transposition}, through the $2$ vertices and faces involved. However, in a Kauffman transposition, the corners $\alpha_j, \alpha'_j$ are adjacent at each $v_j$. This may not be the case in general. 

Transposition contours are in general far from unique. Even if we specify all the data of vertices $v_j$, faces $F_j$ and corners $\alpha_j, \alpha'_j$ involved, each $\gamma_j$ may not even be determined uniquely up to homotopy relative to endpoints. We  define two equivalence relations on contours, one using incidence relations in $U$, the other similar to \refdef{isotopy_embedded_graphs} of isotopy for graph embeddings.
\begin{defn}[Corner-equivalent, isotopic contours]
\label{Def:equivalent_contours}
\label{Def:isotopy_contour}
Let $\gamma, \gamma'$ be transposition contours for a state $S$. Then $\gamma, \gamma'$ are:
\begin{enumerate}
\item     
\emph{corner-equivalent} if they involve the same vertices $v_j$, faces $F_j$, and corners $\alpha_j, \alpha'_j$;
\item
\emph{isotopic} if there is a continuous family of transposition contours $\gamma_t$ for $t \in [0,1]$ from $\gamma_0 = \gamma$ to $\gamma_1 = \gamma'$.
\end{enumerate}
\end{defn}

Both corner-equivalence and isotopy provide equivalence relations on transposition contours. Two isotopic contours $\gamma_0, \gamma_1$ must pass through the same vertices and corners through the isotopy, hence are corner-equivalent. Two corner-equivalent contours are in general not isotopic, but we will see in \refsec{spines_framings} that in some circumstances they are isotopic. 

The state $S'$ obtained by contour transposition along $\gamma$ only depends on the incidence relations between faces and corners in $\gamma$, and hence only on the corner-equivalence class of $\gamma$.

\subsection{Tait graph and spine}
\label{Sec:Tait_spine}

Given a multiverse $(U, \Sigma, \mathscr{F})$, consider building a bipartite graph embedded in $\Sigma$, with vertices coloured black and white, as follows.

Place a white vertex at each interior vertex of $U$, and a black vertex in each face of $U$. We consider \emph{adjacency triples} $(w,\alpha,b)$ where $w$ is a white vertex (hence a 4-valent vertex of $U$), $\alpha$ is one of the 4 corners of $U$ at $w$, and $b$ is the black vertex in the same face of $U$ as $\alpha$. In other words, an adjacency triple $(w, \alpha, b)$ registers the incidence of a vertex $w$ of $U$ with a face $b$ of $U$ through a corner $\alpha$ of $w$. We place an edge for each adjacency triple, resulting in the following construction, similar to the notion of overlaid Tait graph used e.g. in \cite{Celoria_uber_23, Cohen_Teicher_14}. 
\begin{defn}[Tait graph]\label{Def:Tait}
A \emph{Tait graph} of a multiverse graph $U$ on a multiverse surface $\Sigma$ is a bipartite graph $T$ embedded on $\Sigma$ as follows:
\begin{enumerate}
\item
The \emph{white vertices} of $T$ are the interior vertices of $U$.
\item 
One \emph{black vertex} of $T$ is placed in each face of $U$.
\item 
$T$ has an edge for each adjacency triple $(w, \alpha, b)$, with endpoints $w$ and $b$, embedded as an arc with interior in the face containing $b$, and departing $w$ through the corner $\alpha$.
\end{enumerate}
\end{defn}
It is common to construct a Tait graph of a knot diagram using a chequerboard colouring of the faces, so that the graph described below is an ``overlay" of the two Tait graphs from each colour. However, in a multiverse there is in general no such colouring.

Note the above definition does not make use of starred faces; it is associated to a multiverse graph, rather than a universe. Using only the unstarred faces of the multiverse, we define the following  graph, which is  crucial in the sequel.
\begin{defn}[Spine]\label{Def:spine}
A \emph{spine} of a multiverse $(U, \Sigma, \mathscr{F})$ is the induced subgraph $G$ of a Tait graph $T$ on its white vertices, and black vertices in unstarred faces.
\end{defn}
Thus, a spine has a white vertex for each interior vertex of $U$, a black vertex in each unstarred face of $U$, and an edge at each corner of an interior vertex of $U$ in an unstarred face.

For each corner $\alpha$ of a vertex of $U$, there is at most one adjacency triple including $\alpha$, and hence at most one edge of $G$ or $T$. Thus an edge of a Tait graph or spine may be identified by its corner $\alpha$.

Similar to our discussion of contours in \refsec{states_contours}, neither a Tait graph nor a spine of $(U, \Sigma, \mathscr{F})$ is generally unique. Hence we refer to \emph{a} (rather than \emph{the}) Tait graph or spine. While a Tait graph or spine is not generally unique, the underlying (non-embedded) graph only depends on incidence relations in the multiverse, and thus \emph{is} uniquely determined by the multiverse. 

There is a natural notion of isotopy of Tait graphs or spines, following \refdef{isotopy_embedded_graphs} for embedded graphs and \refdef{isotopy_contour} for contours. 
\begin{defn}[Isotopy of Tait graphs and spines] \
\label{Def:isotopy_Tait_graph}
Two spines (resp. Tait graphs) are \emph{isotopic} if there is a continuous family of spines (resp. Tait graphs) from one to the other.
\label{Def:isotopy_spine}
\end{defn}
In such an isotopy, each white vertex is fixed at an interior vertex of $U$. Black vertices must stay in the interior of a face $f$ but may move freely inside $f$. Each edge may move freely, subject to the constraints of proceeding from a prescribed white vertex, through a prescribed corner, within a prescribed face, to a prescribed black vertex, without intersecting other edges.

When a face $f$ of $U$ is not a disc, the embedding of edges of $T$ or $G$ in $f$
is not unique, even up to homotopy fixing endpoints. However, when $f$ is homeomorphic to a disc, then $T$ and $G$ are unique in $f$ up to isotopy. 

A spine $G$ of a multiverse need not be connected, and may have multiple edges. For instance the top right example of \reffig{exampleofmultiverse} has disconnected spine (since the starred faces separate the unstarred faces) and has multiple edges around the loop in $U$. However, as $G$ is bipartite, it does not have loops. 

As both the Tait graph and the spine are embedded in the interior of the multiverse surface $\Sigma$, by \refdef{outer_face} they each have a distinguished outer face.

In a Tait graph, every white vertex has degree $4$; in a spine, 
every white vertex has degree at most $4$.

Since $G$ is obtained from $T$ by removing vertices and edges in starred faces of $U$, and starred faces are adjacent to $\partial \Sigma$, we immediately have the following, which will be useful in the sequel.
\begin{lem}
\label{Lem:where_Tait_spine_coincide}
Let $(U, \Sigma, \mathscr{F})$ be a multiverse, $T$ a Tait graph, and $G$ a spine subgraph of $T$. Let $\Pi$ be the subset of $\Sigma$ formed by the closure of all faces of $U$ not adjacent to $\partial \Sigma_0$. Then $G$ and $T$ coincide on $\Pi$. \qed
\end{lem}

\subsection{Spines and framings}
\label{Sec:spines_framings}

A spine provides a way to ``frame" or ``guide" transposition contours, and so we make the following crucial definition.
\begin{defn}
\label{Def:framed_multiverse}
A \emph{framed multiverse} is a quadruple $(U, \Sigma, \mathscr{F}, G)$ consisting of a multiverse $(U, \Sigma, \mathscr{F})$ together with a spine $G$. We also refer to $G$ as a \emph{framing} of $(U, \Sigma, \mathscr{F})$.
\end{defn}

We now discuss how to use the spine $G$ to ``frame" a transposition contour. So let $(U, \Sigma, \mathscr{F}, G)$ be a framed multiverse, $S$ a state, and $\gamma$ an $n$-transposition contour for $S$. Thus as in \refsec{states_contours} and \refdef{transposition_contour}, $\gamma$ consists of arcs $\gamma_j$, over $j \in \Z/n\Z$, passing through interior vertices $v_j$, corners $\alpha_j$ and $\alpha'_j$, and faces $F_j$. Each arc $\gamma_j$ travels from $v_j$ via the corner $\alpha_j$, through face $F_j$, to $v_{j+1}$ via the corner $\alpha'_{j+1}$.

Consider the restriction of $G$ to the closure of the face $F_j$. There is a white vertex $w$ at each interior vertex of $U$ on the boundary of $F_j$, and a single black vertex which we denote $b_j$. There is an edge connecting $b_j$ to each white vertex $w$ through each corner $\alpha$ of $F_j$ at $w$. In other words, its edges are given by adjacency triples $(w, \alpha, b_j)$, and are uniquely identified by corners.
\begin{defn}
\label{Def:framed_contour}
Let $(U, \Sigma, \mathscr{F}, G)$ be a framed multiverse and $S$ a state. A \emph{framed transposition contour} is an transposition contour $\gamma$ for $S$ such that each arc $\gamma_j$ is homotopic in $F_j$, relative to endpoints, to the directed path in $G$ from $v_j$ to $v_{j+1}$ along the edges corresponding to corners $\alpha_j, \alpha'_{j+1}$.
\end{defn}
Thus, a framed transposition contour must be guided, or ``framed", by the edges of the spine $G$. \reffig{contours_framed_and_not} shows examples of transposition contours, framed and not. 

\begin{figure}
\begin{center}
\includegraphics[width=0.6\textwidth]{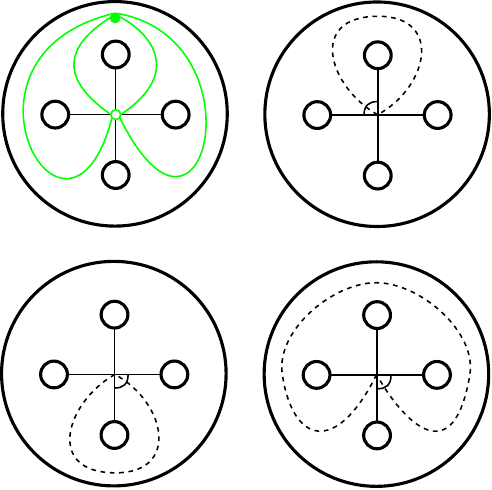}
\end{center}
    \caption{Top left: a framed planar multiverse, with spine/framing shown in green. Top and bottom right: Framed transposition contours. Bottom left: A transposition contour which is not framed.}
    \label{Fig:contours_framed_and_not}
\end{figure}

Because they are constrained to be homotopic in each face $F_j$ to arcs of a spine, we immediately have the following statement about framed contours.
\begin{lem}
\label{Lem:framed_contour_equivalent}
Two framed transposition contours are corner-equivalent if and only if they are isotopic.
\qed
\end{lem}

When each unstarred face of $U$ is a disc, $G$ is unique up to isotopy, so there is a unique framing up to isotopy. Hence all transposition contours are framed. So contours are corner-equivalent if and only if they are isotopic. 

In particular, when $U$ is 2-cell embedded, it has a unique framing up to isotopy, all contours are framed, and contours are corner-equivalent if and only if they are isotopic.

\subsection{States and matchings on a spine}
\label{Sec:states_matchings}

We now consider the set of matchings (\refdef{matching}) $\mathscr{M}_G$ of a spine $G$ of a multiverse $(U, \Sigma, \mathscr{F})$. 

Given a state $S$ of $U$, we may find a matching $M$ of $G$ as follows.
Each white vertex $w$ of $G$ lies at a 4-valent vertex of $U$, at which $S$ assigns a corner $\alpha$. This corner $\alpha$ lies in an unstarred face of $U$ corresponding to a black vertex $b$ of $G$. We thus obtain an adjacency triple $(w, \alpha, b)$, hence an edge of $G$, at each white vertex $w$ of $G$. Let $M$ be the set of these edges. By construction, each white vertex of $G$ belongs to precisely one edge of $M$. As $S$ is a state, each unstarred face of $U$ contains one marker, and hence every black vertex of $G$ belongs to precisely one edge of $M$. Hence $M$ is a matching.

\begin{lemma}\label{Lem:states_and_matchings}
Let $U$ be a multiverse and $G$ a spine. The association of matchings to states defined above is a bijection $\mathscr{S}_U \To \mathscr{M}_G$.
\end{lemma}

\begin{proof}
Indeed, every edge $(w,\alpha, b)$ of the matching $M$ provides a corner $\alpha$ in which to place a marker. As $M$ is a matching and $G$ is bipartite, the edges of $M$ provide a  bijection between black vertices and white vertices of $G$, i.e. between interior vertices of $U$ and unstarred faces of $U$. The corners $\alpha$ given by $M$ thus form a state of $U$, providing an inverse $\mathscr{M}_G \To \mathscr{S}_U$. 
\end{proof}

We can now justify the common  terminology of \refdef{forced_forbidden} and \refdef{forced_forbidden_corner}.
\begin{lem}
\label{Lem:forced_forbidden_corners_edges}
Let $U$ be a multiverse and $G$ a spine. A corner $\alpha$ of $U$ in an unstarred face is forced (resp. forbidden) if and only if the corresponding edge $e$ of $G$ is forced (resp. forbidden).
\end{lem}

\begin{proof}
A state of $U$ choosing the corner $\alpha$ corresponds to a matching of $G$ containing $e$. Every (resp. no) state of $U$ chooses $\alpha$, if and only if every (resp. no) matching of $G$ contains $e$.
\end{proof}

Since matchings only use the underlying graph of $G$, the set $\mathscr{M}_G$ does not depend on the choice of spine.

We can apply the reduction construction of \refdef{reduction} to a spine.
\begin{defn}[Reduced spine]\label{Def:reduced_spine}
A \emph{reduced spine} $G_0$ of $U$ is the reduction of a spine $G$ of $U$.
\end{defn}

Combining \reflem{states_and_matchings} with \reflem{remove_no_matchings}(i) then yields bijections
\begin{equation}
\label{Eqn:states_matchings_reduced_spine}
\mathscr{S}_U \cong \mathscr{M}_G \cong \mathscr{M}_{G_0}.
\end{equation}
In the case that $\mathscr{M}_G = \emptyset$, then $G_0$ simply consists of isolated vertices, and \refeqn{states_matchings_reduced_spine} is a bijection of empty sets.

\subsection{Alternating cycles and transposition contours}
\label{Sec:alternating_cycles_contours}

The notions of alternating cycle for a matching $M$, and framed transposition contour for a state $S$, are closely related. We now discuss how. Let $(U, \Sigma, \mathscr{F}, G)$ be a framed multiverse, $S$ a state of $U$, and $M$ the corresponding matching of $G$.

\begin{lem}
\label{Lem:alternating_cycles_contours}
A vertex-simple alternating cycle of $G$ relative to $M$, regarded as a simple closed curve on $\Sigma$, is a framed transposition contour for $S$.
\end{lem}

\begin{proof}
Let $C$ be a vertex-simple alternating cycle of $G$. The vertices of $C$ alternate between white vertices (corresponding to interior vertices $v_j$ of $U$) and black vertices (corresponding to unstarred faces $F_j$ of $U$), through edges corresponding to corners $\alpha_j, \alpha'_{j+1}$ of $U$. Letting $C$ have length $2n$, we may take $j \in \Z/n\Z$. We may direct $C$ so that it passes in order through $v_1, \alpha_1, F_1, \alpha'_2, v_2, \alpha_2 F_2, \ldots, v_n, \alpha_n, F_n, \alpha'_1$, and so that the edges corresponding to $\alpha_j$ are precisely the edges of $M$ in $C$.

The directed cycle $C$ is cut by the vertices $v_j$ into $n$ paths $\gamma_j$ of length $2$, with each path proceeding from $v_j$ to $v_{j+1}$ via edges corresponding to $\alpha_j$ and $\alpha'_{j+1}$, with $\alpha_j \in M$ and $\alpha'_{j+1} \notin M$. As the matching $M$ corresponds to the state $S$, the first edge of $\gamma_j$ corresponds to the corner $\alpha_j$ selected by $S$ at $v_j$, and indeed is incident to $v_j$ through the corner $\alpha_j$. As $C$ is vertex-simple, all the $v_j$ and $F_j$ are distinct over all $j \in \Z/n\Z$. Thus the $\gamma_j$ form the arcs of a transposition contour $\gamma$ as in \refdef{transposition_contour}. Each $\gamma_j$ is homotopic to (indeed coincides with) the edges of $G$ corresponding to $\alpha_j, \alpha'_{j+1}$, so $\gamma$ is framed.
\end{proof}

For a converse statement, and indeed equivalence, let $\gamma$ be a framed transposition contour for $S$, with arcs $\gamma_j$, vertices $v_j$, corners $\alpha_j, \alpha'_j$, and faces $F_j$ as in \refdef{transposition_contour} and \refsec{states_contours}, with $\gamma_j$ lying in the face $F_j$, having endpoints at corners $\alpha_j, \alpha'_{j+1}$ of vertices $v_j, v_{j+1}$, and homotopic to the edges of $G$ corresponding to $\alpha_j, \alpha'_{j+1}$. We can ``straighten" $\gamma$ into a cycle on $G$ as follows.
\begin{defn} [Straightening] \
\label{Def:straightening}
\begin{enumerate}
\item 
The \emph{straightening} of $\gamma_j$ is the path $C_j$ of length $2$ in $G$ from $v_j$ to $v_{j+1}$ along edges corresponding to corners $\alpha_j, \alpha'_{j+1}$.
\item 
The \emph{straightening} of $\gamma$ is the cycle of $G$ formed by the $C_j$.
\end{enumerate}
\end{defn}

\begin{lem}
\label{Lem:straightening_is_alternating_cycle}
\label{Lem:contours_alternating_cycles_equiv}
The straightening $C$ of $\gamma$ is a vertex-simple alternating cycle relative to $M$, isotopic to $\gamma$.
Straightening provides a bijection between isotopy classes of framed transposition contours for $S$, and vertex-simple alternating cycles of $G$ relative to $M$.
\end{lem}

\begin{proof}
The edges of $G$ corresponding to corners $\alpha_j$ selected by $S$ are precisely the edges of $M$ in $C$, so $C$ is alternating. As the $v_j$ are distinct, $C$ is vertex-simple. By \reflem{alternating_cycles_contours} then $C$ is a contour for $S$, and as $\gamma$ is framed, $C$ is isotopic to $\gamma$. A framed contour $\gamma'$ isotopic to $\gamma$ is also corner-equivalent to $\gamma$, hence has the same straightening. 
We conclude that regarding the cycle $C$ as a curve provides a ``forgetful" map from alternating cycles to transposition contours, to which straightening provides an inverse.
\end{proof}

When each unstarred face of $U$ is a disc, then as discussed in \refsec{spines_framings}, $U$ has a unique framing $G$, all contours are framed, and contours are corner-equivalent if and only if they are isotopic. Straightening then provides a bijection between isotopy classes of transposition contours for $S$, and vertex-simple alternating cycles of $G$ relative to $M$. In particular, this bijection arises if $U$ is 2-cell embedded.

\subsection{Topology of multiverses}
\label{Sec:topology_multiverses}

We now discuss some of the topological constraints on a multiverse.

\begin{lem}
\label{Lem:trivial_2-cell_embed}
Suppose $(U, \Sigma, \mathscr{F})$ is a 2-cell embedded multiverse. If some boundary component of $\Sigma$ is disjoint from $U$, then $U$ is empty and $\Sigma$ is a disc.
\end{lem}

\begin{proof}
The boundary component $C$ disjoint from $U$ lies in a single face $f$, which must be a disc with boundary $C$. 
So $f$ is a connected component of $\Sigma$, and as $\Sigma$ is connected then $f = \Sigma$. Hence $U$ is empty.
\end{proof}

For the following statement and proof, we recall the notation $F,V_{int},N, V_\partial, V_{int}$ of \refdef{vef_labels} and \reflem{multiverse_boundary_vertices_even}.
\begin{lem}
\label{Lem:euler_char_nearly_2-cell}
\label{Lem:Euler_char_arg}
Let $U$ be a multiverse graph embedded in a multiverse surface $(\Sigma, \partial \Sigma_0)$, such that every face $f$ of $U$ is either:
\begin{enumerate}
\item a disc, or
\item an annulus, one of whose boundary components is a component of $\partial \Sigma$, with the other boundary component being a cycle of $U$.
\end{enumerate}
Let $b$ be the number of boundary components of $\Sigma$ forming boundary components of annular faces. Then
\[
F - V_{int} = N + \chi(\Sigma) + b.
\]
\end{lem}
Here  $\chi(\Sigma)$ denotes the Euler characteristic of $\Sigma$. 

\begin{proof}
We first prove the result in the case $b=0$, so that (ii) never arises, and $U$ is 2-cell embedded.

If $\Sigma$ has a boundary component $C$ disjoint from $U$, then \reflem{trivial_2-cell_embed} implies $U$ is empty and $\Sigma$ is a disc. Then $F=1$, $V_{int} =0$, $N=0$ and $\chi(\Sigma) = 1$, so the result holds. We may thus assume that every boundary component of $\Sigma$ contains vertices of $U$. 

Since $U$ is 2-cell embedded, the vertices, edges and faces of $U$, together with $2N$ arcs along $\partial \Sigma$, form a cell decomposition of $\Sigma$. There are $V_{int} + 2N$ vertices and $F$ faces in this cell decomposition. 
Counting each edge twice via the degrees of vertices, the number $E$ of edges in the decomposition satisfies 
$2E = 4V_{int} + 3V_\partial = 4 V_{int} + 6N$, so $E = 2 V_{int} + 3N$.  Euler's formula gives
\[
(V_{int}  + 2N) - (2V_{int} + 3N ) + F = \chi(\Sigma),
\]
which simplifies to $F-V_{int}=N+\chi(\Sigma)$, proving the result when $b=0$.

Now consider the general case. Fill in the $b$ boundary components bounding annular faces with discs. This extends $\Sigma$ to a compact oriented surface $\overline{\Sigma}$ of the same genus, with fewer boundary components, satisfying $\chi(\overline{\Sigma}) = \chi(\Sigma) + b$. On $\overline{\Sigma}$, $U$ is 2-cell embedded, and we obtain a cell decomposition of $\overline{\Sigma}$ as in the $b=0$ case. The Euler characteristic argument above again applies (even if we fill every component of $\partial \Sigma$ with a disc), and we obtain 
$F-V_{int} = N + \chi(\overline{\Sigma}) = N + \chi(\Sigma) + b$.
\end{proof}

Consider the special case where $\Sigma$ is a disc and $U$ is connected. If $N \geq 1$, then $U$ is automatically 2-cell embedded, and we have $F-V_{int}=N+1$. If $N=1$ then we precisely have a universe on a disc (\refdef{universe_on_disc}), and $F-V_{int}=2$, as mentioned in \refsec{string_universes}. 
If $N=0$, then $U$ lies entirely in the interior of the disc and $b=1$, so $F-V_{int}=2$. This is the situation in a Kauffman universe (\refdef{universe}), truncating $\R^2$ to a disc. 

A useful consequence is the following statement, showing that certain isolated components in multiverses prohibit them from having any states. 
\begin{prop}
\label{Prop:multiverse_embedded_in_disc_no_states}
Let $(U, \Sigma, \mathscr{F})$ be a multiverse, and suppose $U$ has a connected component embedded in a disc $D$ in the interior of $\Sigma$.
Then $\mathscr{S}_U = \emptyset$.
\end{prop}
In particular, if $U$ is multiverse on a disc $D$ and $\mathscr{S}_U \neq \emptyset$, then $U$ has no components in the interior of $D$

\begin{proof}
Taking an innermost component if necessary, we have a connected multiverse graph $U_0$ embedded in the interior of a disc $D$, which satisfies the hypotheses of \reflem{euler_char_nearly_2-cell}. Thus numbers of interior vertices $V_{int}$ and faces $F$ of $U_0$ in $D$ satisfy $F=V_{int}+2$. One of these faces is the exterior face of $U_0$, leaving $F-1 = V+1$ non-outer faces of $U_0$. In any state of $U$, these $V+1$ faces must have markers from the $V$ vertices of $U_0$, a contradiction.
\end{proof}

\subsection{Decomposing along a 2-cell embedded multiverse and spine}
\label{Sec:decompose_universe_spine}

Let $(U, \Sigma, \mathscr{F})$ be a 2-cell embedded multiverse. Cutting $\Sigma$ along $U$ decomposes $\Sigma$ into discs. We may now consider further cutting along the edges of a spine $G$, to obtain a finer description of the resulting pieces. We suppose $U$ is nonempty, so by \reflem{trivial_2-cell_embed}, each boundary component of $\Sigma$ contains vertices of $U$.

Let $f$ be a face of $U$. Then $f$ is a disc, and may  be regarded as a polygon. The vertices of the polygon are vertices of $U$, which may be degree-4 interior vertices (hence white vertices of $G$), or degree-1 boundary vertices. The sides of the polygon are edges of $U$ or arcs of $\partial \Sigma$. 
Each side of $f$ is of one of three types.
\begin{enumerate}
\item \emph{Interior} edges (\refdef{vef_labels}); these have endpoints which are degree-4 vertices of $U$, or equivalently, white vertices of $G$.
\item \emph{Free} edges (\refdef{vef_labels}); these have at least one endpoint on $\partial \Sigma$.
\item \emph{Boundary} edges, arcs of $\partial \Sigma$.
\end{enumerate}
Proceeding around the boundary of $f$, each boundary edge is adjacent to a free edge on either side. Interior vertices of $U$ (white vertices of $G$) are precisely those not adjacent to a boundary edge.

If $f$ is starred, then $f$ contains no vertices or edges of $G$, so $f$ is not decomposed further. Each side of $f$ is an interior, free, or boundary edge, with at least one boundary edge, since starred faces are adjacent to the outer boundary.

Now assume $f$ is unstarred. The interior of $f$ then contains a single black vertex $v$ of $G$, as well as edges of $G$ connecting $v$ to each corner of $f$ which is an interior vertex of $U$.

It is possible that there are no edges of $G$ in $f$. In this case all vertices of $f$ lie on $\partial \Sigma$, and the boundary of $f$ consists of alternating boundary and free edges. In this case $G$ has an isolated vertex in $G$, so $\mathscr{M}_G = \emptyset$, hence by \reflem{states_and_matchings} $\mathscr{S}_U = \emptyset$. Cutting $G$ out of $f$ results in a punctured disc.

Otherwise, there is some number $m \geq 1$ of edges of $f$ in $G$, and cutting $f$ along these $m$ edges of $G$ cuts $f$ into $m$ smaller faces. Each such face $f'$ can again be regarded as a polygon, which contains two consecutive edges around its boundary which are edges of $G$, and the remaining edges around its boundary form a sequence of edges $e_1, \ldots, e_r$ around the boundary of $f$ from one white vertex of $G$ to the next. If $e_1$ is an interior edge then both its endpoints are white vertices of $G$, so $r=1$. Otherwise the edges $e_j$ alternate between free and boundary edges, beginning and ending with free edges, so $r$ is odd, $r=2s-1$, and $f'$ is a $(2s+1)$-gon.

We summarise this discussion in the following lemma.
\begin{lem}
\label{Lem:spine_decomposition}
Let $(U, \Sigma, \mathscr{F}, G)$ be a nonempty 2-cell embedded framed multiverse. Cutting $\Sigma$ along $U$ and $G$, the resulting faces are all of precisely one of the following types.
\begin{enumerate}
\item A triangle, whose sides are two edges of $G$ and an interior edge of $U$.
\item A $(2s+1)$-gon for some $s \geq 2$, whose sides consist of two consecutive edges of $G$, and $2s-1$ edges alternating between $s$ free edges of $U$ and $s-1$ arcs of $\partial \Sigma$.
\item A starred face of $U$, whose sides consist of at least one arc of $\partial \Sigma_0$, at least one free edge, and possibly interior edges of $U$.
\item A punctured $2s$-gon for some $s \geq 1$, whose sides alternate between $s$ free edges of $U$ and $s$ boundary arcs of $\partial \Sigma$.
\end{enumerate}
If there is a punctured disc then $\mathscr{S}_U = \emptyset$.
\qed
\end{lem}

Note this statement fails when $U$ is empty: then $\Sigma$ is a disc, there is a single starred face, and $G$ is empty; this fits none of the cases above.

Triangles of type (i) can degenerate. The two sides which are edges of $G$ may coincide; this happens if and only if the side which is an edge of $U$ is in fact a loop of $U$. In this case, the triangle is the face inside the loop.

\subsection{2-cell embedded spines}
\label{Sec:spines_almost_2-cell}

We will show that a spine of a 2-cell embedded multiverse is ``almost" 2-cell embedded, in the following sense.

\begin{defn}[Almost 2-cell embedding]\label{def:almost_2-cell}
An \emph{almost $2$-cell embedding} of a graph $X$ in a surface $\Sigma$ is an embedding of $X$ in $\Sigma$ such that
\begin{enumerate}
\item 
each face of $X$ that is not adjacent to a boundary component of $\Sigma$ is a disc, and
\item 
each boundary component of $\Sigma$ is adjacent to a unique face, which is not a disc.
\end{enumerate}
\end{defn}
Note that if (i) were to apply to all faces of $X$, this would be the usual definition of 2-cell embedded. Condition (ii) instead requires quite different behaviour near a boundary component. It will be satisfied, for instance, if $X$ is disjoint from $\partial \Sigma$, and includes edges and vertices which, around each boundary component, proceed around a boundary-parallel curve, cutting off an annular collar. Then, each boundary component of $\Sigma$ will be adjacent to a single annular face. But it is also satisfied in more general circumstances. For instance, two boundary components of $\Sigma$ may be adjacent to the same (non-disc) face. More trivially, if $\Sigma$ is not a disc and $X$ is the empty graph, then the empty embedding is trivially almost 2-cell embedded.

\begin{lemma}\label{Lem:spine_is_2-cell}
Let $(U, \Sigma, \mathscr{F}, G)$ be a nonempty 2-cell embedded framed multiverse. Then we have the following.
\begin{enumerate}
\item Each face of $G$ not adjacent to $\partial \Sigma$ is a 4-gon, and in particular a disc.
\item Each boundary component of $\Sigma$ is adjacent to a unique face of $G$, which is not a disc.
\end{enumerate}
In particular, $G$ is almost $2$-cell embedded.
\end{lemma}

\begin{proof}
Cutting $U$ along $\Sigma$ and $G$ results in the faces of types (i)--(iv) of \reflem{spine_decomposition}.
We obtain the faces of $G$ by gluing together these faces along the edges of $U$.
Note that types (ii)--(iv) are all adjacent to the boundary of $\Sigma$.

Let $f$ be a face of $G$ not adjacent to $\partial\Sigma$. Then $f$ is obtained by gluing triangles of type (i). But each triangle of type (i) contains precisely one edge of $U$. Hence $f$ is obtained from two triangles of type (i), glued along an edge of $U$. Thus $f$ is a 4-gon and, in particular, a disc.

Let $C$ be a boundary component of $\Sigma$. By construction, $G$ is disjoint from $C$, 
and hence $C$ is adjacent to a unique face $f$ of $G$. The face $f$ has $C$ as a boundary component, but as $U$ is nonempty, $f$ has at least one other boundary component, and hence is not a disc.
\end{proof}

To summarise: a spine $G$ is an embedded graph in the interior of $\Sigma$, whose faces away from the boundary  are quadrilaterals. Note that in such quadrilaterals, some of the 4 edges or 4 vertices may coincide.

\subsection{Dual of spine}
\label{Sec:dual_of_spine}

Consider a framed multiverse $(U, \Sigma, \mathscr{F}, G)$. The spine $G$ is embedded in $\Sigma$, and we can construct a dual in a standard way, as follows.

\begin{defn}[Dual of spine]
\label{Def:dual_of_spine}
A \emph{dual} of a spine $G$ is a graph $G^\perp$ embedded in $\Sigma$, constructed as follows.
\begin{enumerate}
\item One vertex of $G^\perp$ is placed in the interior of each face of $G$.
\item $G^\perp$ has an edge $e^\perp$ for each edge $e$ of $G$. The edge $e^\perp$ joins the vertices of $G^\perp$ in the faces of $G$ on either side of $e$, avoids $\partial \Sigma$ and vertices of $G$, and is drawn to intersect $e$ exactly once.
\end{enumerate}
\end{defn}
We say the edge $e^\perp$ of $G^\perp$ is \emph{dual} to the edge $e$ of $G$.
If the same face of $G$ appears on both sides of $e$, then $e^\perp$ is a loop. 
In our diagrams, $G$ is drawn in green and $G^\perp$ in red.

As $G$ is embedded in the interior of the multiverse surface $\Sigma$, it has a distinguished outer face (\refdef{outer_face_of_spine}), and we have the following.
\begin{defn}
\label{Def:outer_vertex}
The \emph{outer vertex} of a dual of spine $G^\perp$ is the vertex placed in the outer face of $G$.
\end{defn}

Note that the embedded graph $G^\perp$ described above is not unique, even up to homotopy. For instance, consider a face $f_1$ of $G$ adjacent to $\partial \Sigma$. If $U$ is nonempty, then $f_1$ is not a disc, and contains some edge $e$ of $G$ in its boundary with a face $f_2$ on the other side. (Possibly $f_1 = f_2$.) Let $f_1^\perp, f_2^\perp$ be the vertices of $G^\perp$ dual to $f_1, f_2$ respectively. Since $f_1$ is not a disk, the edge $e^\perp$ dual to $e$ can be drawn in at least two ways which are not homotopic to each other in $f_1$ or $f_2$ relative to endpoints. See for example \reffig{ex_dual_of_spine}. Thus, $G^\perp$ is not unique as an embedded graph on $\Pi$. 

However, $G^\perp$ is uniquely determined as a graph by $G$, since it is determined by incidence relations between edges and faces of $G$. 

Note that $G$ may contain isolated vertices, but these have no effect on $G^\perp$: if we remove isolated vertices from $G$, and then construct $G^\perp$ as above, it is also a dual for the original $G$.

	\begin{figure}
		\centering
		\includegraphics[width=0.6\textwidth]{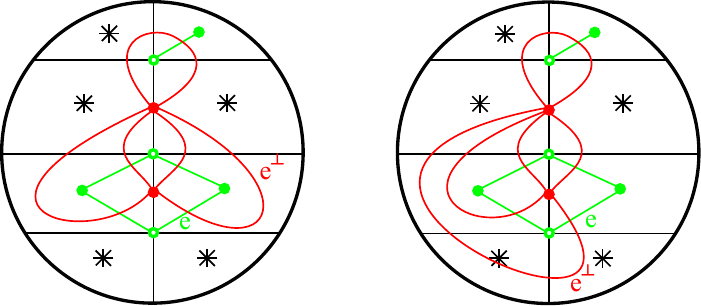}
		\caption{A multiverse on a disc. A spine $G$ is drawn in green. Two distinct duals $G^\perp$ are drawn in red.}
        \label{Fig:ex_dual_of_spine}
	\end{figure}

By definition, the vertices of $G^\perp$ are in bijection with faces of $G$, and edges of $G^\perp$ are in bijection with edges of $G$. However, faces of $G^\perp$ are generally not in bijection with vertices of $G$. For instance, in \reffig{ex_dual_of_spine} there are faces of $G^\perp$ containing more than one vertex of $G$.

One might hope that, at least if $U$ is 2-cell embedded, so that by \reflem{spine_is_2-cell}  $G$ is almost 2-cell embedded, then $G^\perp$ is also close to 2-cell embedded in some sense. However, \reffig{non-2-cell_dual} shows an example of a 2-cell embedded multiverse $U$ on a surface $\Sigma$, together with a spine $G$ and dual $G^\perp$, such that $G^\perp$ has a face, away from $\partial \Sigma$, with nonzero genus.

\begin{figure}
\begin{center}
\includegraphics[width=0.9\textwidth]{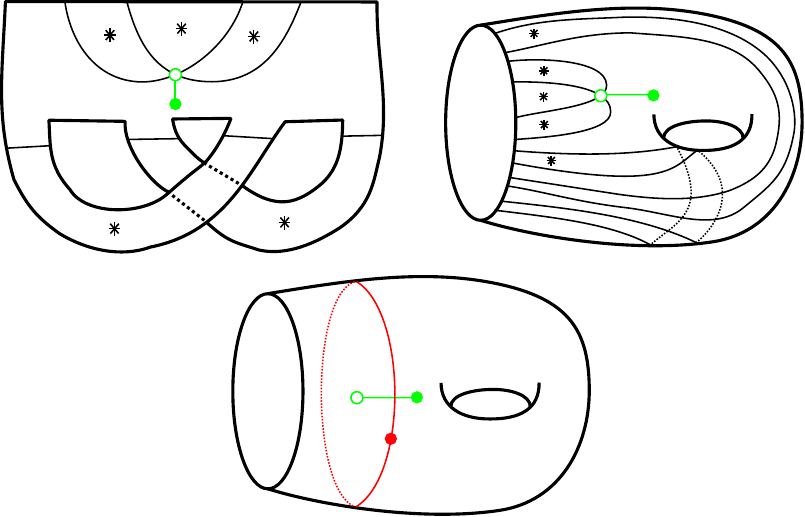}
\end{center}
    \caption{A 2-cell embedded multiverse $U$ (black) with a spine $G$ (green) and dual $G^\perp$ (red).}
    \label{Fig:non-2-cell_dual}
\end{figure}

\begin{lem}
\label{Lem:dual_of_spine_connected}
A dual of a spine $G^\perp$ is a connected graph.
\end{lem}

\begin{proof}  
Let $v_1, v_2$ be vertices of $G^\perp$, corresponding to faces $f_1, f_2$ of $G$. A multiverse surface $\Sigma$ is connected, so there is a path in $\Sigma$ connecting $f_1$ and $f_2$. This path may be chosen to avoid vertices of $G$. It passes through various edges and faces of $G$, which correspond to various edges and vertices of $G^\perp$, yielding a path in $G^\perp$ from $v_1$ to $v_2$. 
\end{proof}

Around any vertex $v$ of $G$, some vertices of $G$ depart: let them be, in cyclic order (say, counterclockwise), $e_1, \ldots, e_n$. As $G$ is bipartite, it contains no loops, so all the $e_j$ are disjoint. The dual edges $e_1^\perp, \ldots, e_n^\perp$ then form a simple directed cycle in $G^\perp$, and its equivalence class forms a simple cycle in $G$.
\begin{defn}
\label{Def:basic_cycle}
The simple cycle in $G^\perp$ represented by $e_1^\perp, \ldots, e_n^\perp$
is called the \emph{basic cycle} of $G^\perp$ around $v$.
\end{defn}
Basic cycles are in some ways like elementary cycles of plane graphs (\refdef{elementary_cycle}). However, unlike elementary cycles, they may not bound faces. For instance, in the left multiverse of \reffig{ex_dual_of_spine}, the bottom vertex of $G$ has a basic cycle around it, but this basic cycle does not bound a face.

Note that every edge $e^\perp$ of $G^\perp$ appears twice in basic cycles: if the dual edge $e$ of $G$ has vertices $v,w$, then $e^\perp$ lies in the basic cycles around $v$ and $w$.

The edges of $G$ adjacent to a vertex $v$ may be cyclically ordered in clockwise or counterclockwise fashion. As a multiverse surface is oriented, there is a well defined notion of clockwise and counterclockwise at each point. These cyclic orderings of vertices about $v$ yield orientations of edges of a basic cycle, and we can make the following definition.
\begin{defn}
\label{Def:orientations_basic_cycles}
If $e_1, \ldots, e_n$ are the edges about $v$ in counterclockwise (resp. clockwise) order, then the directed basic cycle $e_1^\perp, \ldots, e_n^\perp$ about $v$ is oriented \emph{counterclockwise} (resp. clockwise).
\end{defn}

Note that the clockwise orientation of a basic cycle may not appear ``globally clockwise" in a multiverse. For example, in the left diagram of \reffig{ex_dual_of_spine}, the basic cycle about the bottom vertex of $G$, oriented counterclockwise, appears clockwise as a curve on the disc.

\section{Clock theorem in genus zero}
\label{Sec:genus_0}

In this section we consider planar multiverses $(U, \Pi, \mathscr{F})$, i.e. where $\Pi$ has genus zero, and prove \refthm{main_thm_1}. Since $\Pi$ can be embedded in $\R^2$, we may regard $U$ and a spine $G$ as plane  graphs. Our results in the plane case essentially rely on the Jordan curve theorem, so that every cycle has an interior and exterior. 

The rough idea is to apply Propp's \refthm{Propp_matching} to the spine $G$, which provides the set of matchings with a distributive lattice structure, and \reflem{states_and_matchings}, which identifies matchings of $G$ with states of $U$. However, a spine may be disconnected, and may contain forced or forbidden edges, violating the hypotheses of \refthm{Propp_matching}. Thus, we consider the reduced spine (\refdef{reduced_spine}), and our generalisation \refprop{finite_bipartite_plane_lattice} of Propp's theorem.

\reffig{Hasse_diagram_example1_thm1} and \reffig{Hasse_diagram_example2_thm1} provide examples of planar multiverses and their lattices of states given by \refthm{main_thm_1}, which may be useful guides to our constructions. \reffig{Hasse_diagram_example1_thm1} shows the  spine of the multiverse (and the reduced spine, which is obtained by removing the dotted green edges), with states and matchings.

\subsection{Faces of graphs embedded in plane surfaces, boundaries, and elementary cycles}
\label{Sec:elementary_on_plane_surface}
\label{Sec:cycles_twisting_plane_surfaces}

Let $\Pi$ be a planar multiverse surface. Then $\Pi$ can be regarded as a subsurface of $\R^2$, realising its outer boundary (\refdef{plane_surface}) as its outermost boundary component. Indeed, $\Pi$ can be so regarded in a unique way up to isotopy, using the fact that $\Pi$ is oriented. As $\Pi$ has genus $0$, every simple closed curve in $\Pi$ is separating, hence has an interior and exterior (\refdef{interior_exterior}), agreeing with the usual notions from the Jordan curve theorem. Using  the orientation on $\Pi$ and the Jordan curve theorem in $\R^2$, every oriented simple closed curve in $\Pi$ can be regarded as  counterclockwise/positive or clockwise/negative.

Now let $G$ be a finite graph (possibly disconnected, possibly with multiple edges and loops) embedded in the interior of $\Pi$. As a graph embedded in the interior of a multiverse surface, $G$ has an outer face (\refdef{outer_face}). Since we regard $\Pi \subset \R^2$ uniquely up to isotopy, $G$ can be regarded as a plane graph, unique up to isotopy (\refdef{isotopy_embedded_graphs}). The faces of $G$, regarded as an embedded graph in $\Sigma$, are naturally in bijection with the faces of $G$, regarded as a plane graph. Under this bijection, the outer face of $G$ (\refdef{outer_face}) corresponds to the unbounded face of $G$ as a plane graph. As a finite graph embedded in an orientable surface, as discussed in \refsec{boundaries_of_faces}, each face of $G$ has well-defined boundary cycles (\refdef{boundary_cycle}). As a plane graph, as discussed in \refsec{plane_boundary_elementary}, bounded (i.e. non-outer) faces of $G$ have well-defined \emph{outer} boundary cycles, which form elementary cycles (\refdef{outer_boundary_elementary}). Moreover, elementary cycles have well-defined clockwise and counterclockwise orientations. As in \refsec{matchings}, we may consider matchings and alternating cycles on $G$.

Suppose now additionally that $G$ is bipartite. Then $G$ can be regarded as a plane bipartite graph, again unique up to isotopy, so \refsec{matchings_plane_bipartite} applies, and we may consider positive and negative cycles relative to a matching (\refdef{pos_neg_cycle}), and twisting a matching up and down (\refdef{twisting_up_down}) on alternating elementary cycles of $G$. Moreover, \refsec{lattices_matchings} applies, so that set $\mathscr{M}_G$ of matchings of $G$ has a well-defined relation $\leqslant$ from \refdef{relation_on_matchings1}, to which the generalisation \refprop{finite_bipartite_plane_lattice} of Propp's \refthm{Propp_matching} applies, endowing $\mathscr{M}_{G_0}$ (here $G_0$ is the reduction of $G$), and hence $\mathscr{M}_G$ (equal to $\mathscr{M}_{G_0}$ as a set), with the structure of a distributive lattice, the Propp lattice of $G$.

A spine/framing, or reduced spine, of a planar multiverse $\Pi$ is precisely a graph of this type: it is a finite bipartite graph embedded in the interior of $\Pi$.  Hence all the above applies, and we have a Propp lattice. We summarise this discussion in the following statement.

\begin{prop}
\label{Prop:reduced_spine_lattice}
Let $(U, \Pi, \mathscr{F}, G)$ be a framed planar multiverse, and $G_0$ the reduced spine. Then $\mathscr{M}_{G_0}$, with the relation $\leqslant$ of \refdef{relation_on_matchings1}, is a distributive lattice. Moreover, for any two matchings $M, M' \in \mathscr{M}_{G_0} = \mathscr{M}_G$, $M \lessdot M'$ if and only if $M'$ is obtained from $M$ by twisting up at a non-outer face of $G_0$.
\qed
\end{prop}

\subsection{Plane transpositions}
\label{Sec:n-transpositions}

We now define a generalisation of Kauffman transpositions (\refdef{transposition}) to planar multiverses. Just like Kauffman's original notion of transposition, these send states to states, by rotating markers at vertices, and they can be clockwise or counterclockwise. However, this generalised notion can involve an arbitrary number $n \geq 1$ of vertices. They are a special case of the countour $n$-transpositions of \refdef{partial_transposition}.
We call them \emph{plane $n$-transpositions} or just \emph{plane transpositions}.

Let $(U, \Pi, \mathscr{F}, G)$ be a framed planar multiverse, and $S$ a state. Let $n \geq 1$ and let $\gamma$ be a framed $n$-transposition contour (\refdef{framed_contour}) for $S$, involving interior vertices $v_j$ of $U$ over $j \in \Z/n\Z$. As in \reffig{alternating_cycle} (left), let $\alpha_j$ be the corner at $v_j$ chosen by $S$, and $F_j$ be the unstarred face of $U$ containing $\alpha_j$, so that $\gamma$ consists of arcs $\gamma_j$ passing from $v_j$ via $\alpha_j$ through $F_j$ to $v_{j+1}$. Let $\alpha'_{j+1}$ be the corner through which $\gamma_j$ is incident to $v_{j+1}$.
Then contour $n$-transposition on $S$ along $\gamma$ sends $S$ to the state $S'$ obtained by replacing each $\alpha_j$ with $\alpha'_j$. 

As a  simple closed curve $\gamma$ on a plane surface, $\gamma$ is separating, and has an interior in $\Pi$ as per \refdef{interior_exterior}.
Since $\gamma$ passes through $v_j$ via the two corners $\alpha_j$ and $\alpha'_j$, these corners lie partly in the interior and partly in the exterior of $\gamma$. The other two corners of $v_j$ lie either in the interior of $\gamma$, or in the exterior of $\gamma$. There is a direction, clockwise or counterclockwise, such that the rotation from $\alpha_j$ to $\alpha'_j$ about $v_j$ in this direction passes through the interior of $\gamma$. This direction is the same at each $v_j$. We consider rotating the marker at each $v_j$ in this direction, from $\alpha_j$ (the marker in $S$) around to $\alpha'_j$ (the marker in $S'$), through the interior of $\gamma$. The rotation angle at each vertex is $90^\circ$, $180^\circ$ or $270^\circ$, and may be different at different vertices. 
\begin{defn}[Plane transposition]\label{Def:n_transposition}
Suppose that $(U, \Pi, \mathscr{F}, G)$, $S$, $S'$, $\gamma$,
$\alpha_j$, $\alpha'_j$, 
are as above and further satisfy the following condition:
any corner
at any $v_j$ which lies entirely in the interior of $\gamma$ is forbidden (\refdef{forced_forbidden_corner}).

\emph{Plane $n$-transposition}, or just \emph{plane transposition}, of $S$ along $\gamma$ is the contour $n$-transposition of $S$ along $\gamma$, yielding $S'$. The plane $n$-transposition is \emph{clockwise} (resp. \emph{counterclockwise}) accordingly as the markers rotate clockwise (resp. counterclockwise) through the interior of $\gamma$ from $\alpha_j$ to $\alpha'_j$.
\end{defn}
Thus, plane transposition is contour transposition, along a framed contour, in the planar case, where interior corners of the contour are forbidden. 

Since planar transposition is determined by the framed contour $\gamma$, we can refer to plane transposition on the curve $\gamma$.
Alternatively, plane transposition is determined by corners $\alpha_j$ and $\alpha'_j$, or by the states $S$ and $S'$. In particular, if there is a plane transposition taking a state $S$ to the state $S'$, then the vertices $v_j$, corners $\alpha_j, \alpha'_j$ and their number $n$ are determined, as is $\gamma$ up to corner-equivalence (\refdef{equivalent_contours}), hence (as $\gamma$ is framed, by \reflem{framed_contour_equivalent}) up to isotopy.

We regard two plane transpositions as equivalent if they involve the same corners $\alpha_j, \alpha'_j$ or equivalently, the same states $S,S'$ (or equivalently, framed contours $\gamma$ which are corner-equivalent or isotopic). From a state $S$ there are only finitely many plane transpositions possible.  Between any two states $S,S'$ there is at most one $n$ for which there exists a plane $n$-transposition between them, and if such a plane $n$-transposition exists, then it is unique up to equivalence, and is of a specified direction (clockwise or counterclockwise). Note that this relies on $\gamma$ being framed: in general there may be two distinct contours such that contour transposition on each takes $S \mapsto S'$, with one being clockwise and the other counterclockwise. See the bottom two diagrams of \reffig{contours_framed_and_not} for an example.

We define a relation $\leqslant$ on the set of states $\mathscr{S}_U$ of $U$, similar to \refdef{leq_on_states}.
\begin{defn}
\label{Def:leq_on_plane_multiverse_state}
Let $U$ be a framed planar multiverse and let $S,S' \in \mathscr{S}_U$. We write $S \leqslant S'$ if there exists a sequence of states $S=S_0, S_1, \ldots, S_m = S'$ of $U$, for some integer $m \geq 0$, such that each $S_{j+1}$ is obtained from $S_j$ by a counterclockwise plane transposition.
\end{defn}

The requirement for $U$ to be framed, and for plane transpositions to be along framed contours is necessary. Without  it, \refthm{main_thm_1} fails. For example, \reffig{framing_countrexample} shows a cycle of counterclockwise ``plane transpositions" that satisfy all conditions except that they cannot all be framed.
\begin{figure}
\begin{center}
\includegraphics[width=0.6\textwidth]{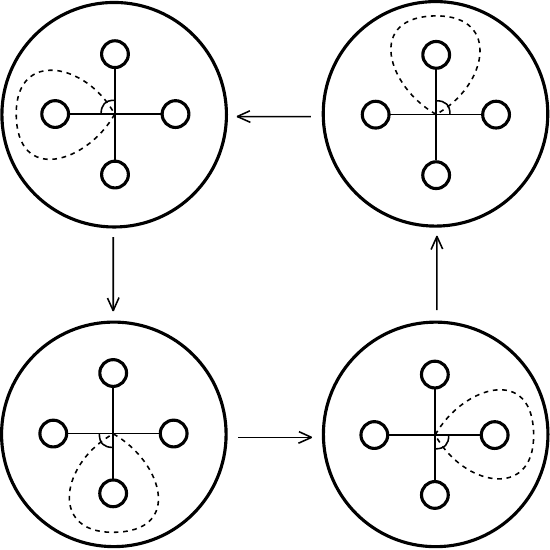}
\end{center}
    \caption{Counterexample to \refthm{main_thm_1} without framings.}
    \label{Fig:framing_countrexample}
\end{figure}

\subsection{Equivalence of twisting and transpositions on planar multiverses}
\label{Sec:twisting_transpositions}

Again let $(U, \Pi, \mathscr{F}, G)$ be a framed planar multiverse, and $G_0$  a reduced spine, so as in 
\refsec{elementary_on_plane_surface}, $G_0$ has well-defined notions of outer face, elementary cycle, positive and negative alternating cycles, twisting up and down, and a relation $\leqslant$ on $\mathscr{M}_{G_0}$, which has a distributive lattice structure as in \refprop{reduced_spine_lattice}. 

By \reflem{states_and_matchings} and \reflem{remove_no_matchings}(i), states of $U$ are in bijection with matchings of $G_0$ as in \refeqn{states_matchings_reduced_spine}. Let $S$ be a state of $U$, and $M$ its corresponding matching on $G_0$. 

Denote by $\Tr_S$ be the set of all counterclockwise plane transpositions that can be done on $S$. As discussed at the end of \refsec{n-transpositions}, $\Tr_S$ is a finite set. Moreover, between any two states there is at most one plane $n$-transposition, for one value of $n$, in a specific direction (clockwise or counterclockwise). Similarly, denote by $\Tw_M$ be the set of all twisting up operations that can be done on $M$; again, this is a finite set. Between two matchings there is at most one twisting up or down, in a specific direction (up or down), on a specific elementary cycle. These two sets are naturally bijective as follows.
\begin{lemma}\label{Lem:transpositions_and_twisting}
Let $S,S'$ be states of $U$ corresponding to matchings $M,M'$ of $G_0$.
There is a counterclockwise (resp. clockwise) plane transposition taking $S$ to $S'$, if and only if there is a twisting up (resp. down) taking $M$ to $M'$.
\end{lemma}
Hence there is a bijection
\[
p \colon \Tr_S \To \Tw_M
\]
which takes the counterclockwise plane transposition $S \mapsto S'$ to the twisting up $M \mapsto M'$.
Considering \refdef{leq_on_plane_multiverse_state} and \refdef{relation_on_matchings1} of the relations $\leqslant$ on $\mathscr{S}_U$ and $\mathscr{M}_{G_0}$, it then follows that the bijection of sets $\mathscr{S}_U \cong \mathscr{M}_{G_0}$ of \refeqn{states_matchings_reduced_spine} respects these relations. 

\begin{proof}
Let $T \in \Tr_S$ be a counterclockwise plane transposition taking $S \mapsto S'$; we construct a twisting up $O \in \Tw_M$ taking $M \mapsto M'$.
Let $v_j, \alpha_j, \alpha'_j, F_j, \gamma_j, \gamma$ for $j \in \Z/n\Z$ be the vertices, corners, faces, arcs, and framed contour involved in $T$, as in \refsec{n-transpositions}.
Recall that each vertex $v_j$ is adjacent to faces $F_j$, $F_{j-1}$ via corners $\alpha_j, \alpha'_j$ respectively.
Note each $v_j$ is also a white vertex of $G$ and $G_0$. Label the vertex of $G$ and $G_0$ in each $F_j$ also by $F_j$, and the edge of $G$ in the corner $\alpha_j$ or $\alpha'_j$ also by $\alpha_j$ or $\alpha'_j$. By the correspondence $\mathscr{S}_U \cong \mathscr{M}_{G_0}$ of \refeqn{states_matchings_reduced_spine}, the corners $\alpha_j$ of $U$ are marked in $S$ and the edges $\alpha_j$ lie in $M$; the corners $\alpha'_j$ of $U$ are marked in $S'$ and the edges $\alpha'_j$ lie in $M'$. Since all the edges $\alpha_j, \alpha'_j$ lie in at least one matching $M$ or $M'$, they also lie in $G_0$.

Let $A$ be the straightening of $\gamma$ (\refdef{straightening}). By \reflem{contours_alternating_cycles_equiv}, $A$ is a vertex-simple alternating cycle of $G$ relative to $M$. Its vertices are the $v_j$ and $F_j$, and its edges are the $\alpha_j$ and $\alpha'_j$. As all its vertices and edges also lie in $G_0$, $A$ is also a vertex-simple alternating cycle of $G_0$ relative to $M$. See \reffig{alternating_cycle}. We will show that twisting up on $A$ provides the desired $O$. 

We claim $A$ is elementary, i.e. (\refdef{outer_boundary_elementary}) that $A$ encircles a face $f$. For this, it is sufficient to prove that at each vertex $v_j$ or $F_j$ of $A$, there are no edges of $G_0$ departing the vertex into the interior of $A$. Since $A$ separates its interior from the exterior, and the matching $M$ (or $M'$) matches all the vertices along the cycle $A$, it must also match the vertices of $G$ in the interior of $A$ to each other. In particular, the number of black and white vertices of $G$ inside $A$ must be equal. Thus, in any matching $N$ of $G$, the number of edges connecting black vertices in the interior of $A$ to white vertices of $A$ (i.e. the $v_j$), and the number of edges connecting white vertices in the interior of $A$ to black vertices of $A$ (i.e. the $F_j$), must be equal; let this number be $k(N)$. But by \refdef{n_transposition} of a plane transposition, at any $v_j$, any corner 
in the interior of $\gamma$ is forbidden (\refdef{forced_forbidden_corner}), hence the corresponding edge of $G$ is forbidden (\reflem{forced_forbidden_corners_edges}), so
does not appear in any matching. Thus for any matching $N$ we have $k(N) = 0$. Hence all edges of $G$ from $v_j$ or $F_j$ into the interior of $A$ are forbidden, and do not appear in $G_0$. So $A$ is elementary as claimed.

Observe that $A$ is alternating relative to both $M$ and $M'$,
and that twisting $M$ at $A$ yields $M'$.
Since $T$ is counterclockwise, the marker at each $v_j$ rotates counterclockwise through the interior of $A$ from its position in $S$ to its position in $S'$. Thus, in $M$, the edges $\alpha_j$, directed from black to  white vertices, encircle $f$ in the clockwise direction, making $A$ negative relative to $M$. 
Similarly, $A$ is positive relative to $M'$,
and we have constructed an $O \in \Tw_M$ taking $M \mapsto M'$.

Reversing the above, we now construct, from an $O \in \Tw_M$ taking $M \mapsto M'$ by twisting up on a negative cycle $A$, a $T \in \Tr_S$ taking $S \mapsto S'$.
The elementary cycle $A$ of $G_0$ encircles a face $f$ and is alternating relative to $M$ and $M'$. Let its white vertices be $v_j$ and its black vertices $F_j$ in cyclic order, for $j \in \Z/n\Z$,
so that the edge $\alpha_j$ joining $F_j$ and $v_j$ belongs to $M$, and the edge $\alpha'_j$ joining $F_{j-1}$ and $v_j$ belongs to $M'$; we denote the corresponding faces of $U$ by $F_j$ and the corners of $U$ at $v_j$ by $\alpha_j, \alpha'_j$ also.
So $S$ has a marker at each $v_j$ in corner $\alpha_j$ of face $F_j$, while $S'$ has a marker at each $\alpha'_j$ in face $F'_j$. 

By \reflem{alternating_cycles_contours}, when $A$ is regarded as a simple closed curve $\gamma$, it is a framed transposition contour for $S$. Contour transposition along $A$ takes $S \mapsto S'$.
As $A$ is elementary in $G_0$, encircling a face $f$, no edge of $G_0$ passes from a $v_j$ or $F_j$ into the interior of $A$. Hence any edges of $G$ passing from a $v_j$ or $F_j$ into the interior of $A$ are forbidden. Thus, by \reflem{forced_forbidden_corners_edges}, at any $v_j$, any corner in the interior of $\gamma$ is forbidden. So there is a plane transposition $T$ along $\gamma$ taking $S \mapsto S'$.

The fact that $O$ is a twisting up implies that, in $M$, the edges $\alpha_j$ of $G_0$, when directed from black to white vertices, encircle $f$ in the clockwise direction.
Similarly, in $M'$, the edges $\alpha'_j$, directed from black to white vertices, encircle $f$ in the counterclockwise direction.

It follows that $T$ is a counterclockwise plane transposition as desired. The constructions of twisting up from counterclockwise plane transposition and vice versa yield mutually inverse bijections.
\end{proof}

We note that the original notion of transposition, from Kauffman's \refdef{transposition}, also makes sense in any multiverse. On a planar multiverse, we can also translate it as follows.
\begin{lem}
\label{Lem:Kauffman_transposition_twisting}
Let $(U, \Pi, \mathscr{F}, G)$ be a framed planar multiverse, and $S,S'$ states of $U$ corresponding to matchings $M,M'$ of $G$. There is a Kauffman transposition taking $S$ to $S'$ if and only if the following conditions hold:
\begin{enumerate}
\item $M$ is obtained from $M'$ by twisting at an alternating cycle $A$ of $G$ of length $4$; and
\item at each white vertex $v$ of $A$, no edges of $G$ depart from $v$ into the interior of $A$.
\end{enumerate}
\end{lem}
Note that in this lemma, the alternating cycle $A$ need not be elementary, and hence the twisting need not be  twisting up or down in the sense of \refdef{twisting_up_down}. 
We need $\Pi$ to be planar in order to identify an interior of $A$.

\begin{proof}
Let $u,v$ be the two vertices of $U$ involved in the transposition (as in \reffig{transposition}). Let $Q,R$ be the the faces in which markers at $u,v$ respectively lie in $S$. Then $u,v$ may be regarded as white and $Q,R$ as black vertices of $G$. These four vertices, together with the four edges between them corresponding to the corners involved in the transposition from $S$ to $S'$, naturally form a cycle $A$ in $G$ of length $4$, which proceeds $u$--$Q$--$v$--$R$--$u$. The matching $M$ contains the edges $u$--$Q$ and $v$--$R$, while $M'$ contains $u$--$R$ and $v$--$Q$. In particular, $M'$ is obtained from $M$ by twisting at the alternating cycle $A$.  In a transposition, markers rotate precisely $90^\circ$, so that there are no faces between $Q$ and $R$ at $u$ or $v$, and hence there are no edges in $G$ departing from $u$ or $v$ into the interior of $A$.

For the converse, suppose we have $M,M'$ and $A$ as in (i) and (ii). Let the cycle $A$ be $u$--$Q$--$v$--$R$--$u$ where $u,v$ are white and $Q,R$ are black, and let $M$ contain $u$--$Q$ and $v$--$R$, so $M'$ contains $u$--$R$ and $v$--$Q$. In the state $S$ we then have a marker at $u$ in a corner incident with $Q$, and a marker at $v$ in a corner incident with $R$, corresponding to two edges of $A$; in the state $S'$, we have a marker at $u$ in a corner incident with $R$, and a marker at $v$ in a corner incident with $Q$. By condition (ii), the two corners occupied at $u$ by $S$ and $S'$ are adjacent, and the two corners occupied at $v$ by $S$ and $S'$ are adjacent. Thus $S$ and $S'$ are related by a $90^\circ$ rotation at $u$ and $v$. Accordingly as $A$ (oriented with the order of vertices written above) encloses its interior with counterclockwise or clockwise orientation, the markers both rotate counterclockwise or both rotate clockwise, and the states are related by a Kauffman transposition.
\end{proof}

\subsection{Planar clock theorem}
\label{Sec:planar_clock_thm}

We can now state a generalised clock theorem for planar multiverses, a precise version of \refthm{main_thm_1}. Recall that the set of states $\mathscr{S}_U$ of a framed planar multiverse $(U, \Pi, \mathscr{F}, G)$ has a relation $\leqslant$ from \refdef{leq_on_plane_multiverse_state}.
The matchings $\mathscr{M}_{G_0}$ of the reduced spine $G_0$ have the relation $\leqslant$ of \refdef{relation_on_matchings1}, and by \refprop{reduced_spine_lattice}, $\mathscr{M}_{G_0}$ then forms a distributive lattice.

Recall also that we have the bijection $\mathscr{S}_U \cong \mathscr{M}_{G_0}$ of \refeqn{states_matchings_reduced_spine}. As discussed after the statement of \reflem{transpositions_and_twisting}, this bijection preserves the relations $\leqslant$ on the sets. Hence the same is true for $\mathscr{S}_U$ and we have a bijection of distributive lattices.
\begin{theorem}[Planar clock theorem]\label{Thm:clock_on_genus_0}
Let $(U, \Pi, \mathscr{F}, G)$ be a framed planar multiverse. 
Then $\mathscr{S}_U$, equipped with the relation $\leqslant$ of \refdef{leq_on_plane_multiverse_state}, is a distributive lattice.
Moreover, $S \lessdot S'$ if and only if $S'$ is obtained from $S$ by a counterclockwise plane transposition.
\end{theorem}

\begin{proof}
It only remains to prove the statement about covering. Let $M,M' \in \mathscr{M}_{G_0}$ be matchings corresponding to $S,S' \in \mathscr{S}_U$. By \refprop{reduced_spine_lattice}, $M \lessdot M'$ if and only if $M'$ is obtained from $M$ by twisting up at a non-outer face of $G_0$, i.e. by twisting up on a negative elementary cycle of $G_0$. By \reflem{transpositions_and_twisting} this occurs if and only if there is a clockwise plane transposition taking $S$ to $S'$.
\end{proof}

\section{Application to universes}
\label{Sec:applications_to_universes}

Consider a Kauffman universe $(U, \mathscr{F})$ (\refdef{universe}), which as discussed in \refsec{string_universes} is equivalent to a universe in string form, or a universe on a disc $(U, D, \mathscr{F})$  (\refdef{universe_on_disc}). Whichever of these ways we view $U$, we obtain an isomorphic set of states $\mathscr{S}$, and relation $\leqslant$ (\refdef{leq_on_states}) based on Kauffman transpositions (\refdef{transposition}). By Kauffman's clock theorem (\refthm{clock_universe}), this relation endows $\mathscr{S}$ with a distributive lattice structure.

However, such a $(U, D, \mathscr{F})$ is also a planar multiverse.  As $U$ is connected, it is 2-cell embedded in $D$. Hence as discussed in \refsec{spines_framings}, it has a unique framing $G$ up to isotopy. Throughout this section, we regard a string universe $U$ as a framed planar multiverse in this way. The set of states $\mathscr{S}$ then has a relation $\leqslant$ as in \refdef{leq_on_plane_multiverse_state}, based on plane transpositions (\refdef{n_transposition}). By the planar clock theorem (\refthm{clock_on_genus_0}), this relation also provides $\mathscr{S}$ with a distributive lattice structure.

Thus we have two relations on the set of states of $U$. In this section we prove the following precise version of \refthm{Kauffman_plane_equivalence}.
\begin{thm}
\label{Thm:clock_lattices_isomorphic}
Two states $S,S'$ of $U$ satisfy $S \leqslant S'$ with respect to the relation $\leqslant$ of \refdef{leq_on_states}, if and only if $S \leqslant S'$ with respect to the relation $\leqslant$ of \refdef{leq_on_plane_multiverse_state}.
\end{thm}
Thus, the distributive lattices of the two clock theorems, \refthm{clock_universe} and \refthm{clock_on_genus_0}, are isomorphic. One direction is straightforward.

\begin{lem}
Let $U$ be a string universe. Then any counterclockwise (resp. clockwise) Kauffman transposition is a plane $2$-transposition.
\end{lem}

\begin{proof}
As discussed in \refsec{states_contours} after \refdef{partial_transposition}, a Kauffman transposition is a contour $2$-transposition, with contour $\gamma$ as illustrated in \reffig{Kauffman_as_contour_transposition}. As discussed in \refsec{spines_framings}, since $U$ is 2-cell embedded, all contours are framed. As the corners involved in a Kauffman transposition are adjacent, there are no interior corners of $\gamma$, so their requirement to be forbidden is vacuously satisfied. Hence \refdef{n_transposition} of a plane $2$-transposition is satisfied, with the same direction of rotation of markers as the Kauffman transposition.
\end{proof}

The rest of this section is devoted to the following, which completes the proof of \refthm{clock_lattices_isomorphic} and \refthm{Kauffman_plane_equivalence}.
\begin{prop}
\label{Prop:universe_n-transpositions}
Let $U$ be a string universe. Any plane transposition between two states of $U$ is a Kauffman transposition.
\end{prop}
In particular, there are no plane $n$-transpositions on $U$ for $n \neq 2$. 

Throughout the rest of this section, $(U, D, \mathscr{F}, G)$ is a string universe, regarded as a 2-cell embedded planar multiverse on the disc $D$, with $G$ a framing unique up to isotopy. In the notation of \refdef{vef_labels} and \reflem{multiverse_boundary_vertices_even}, $V_\partial = 2$, $N=1$. There are two outer faces, which are precisely the starred faces of $\mathscr{F}$.

\subsection{Alternating cycles in spines of universes}

Let $T$ be a Tait graph containing $G$. We now prove two technical lemmas.

\begin{lem}
\label{Lem:long_alternating_cycles_require_escape}
Let $M \in \mathscr{M}_G$, and let $C$ be a vertex-simple alternating cycle of $G$ relative to $M$, of length $2L>0$. Let $E_{\partial W}$ be the number of edges of $T$ departing from a white vertex of $C$ into the exterior of $C$. Then $E_{\partial W} = L+2$.
\end{lem}

\begin{proof}
Since $C$ is alternating relative to $M$, $M$ pairs each vertex of $C$ to another vertex of $C$ via an edge of $C$. Thus, each vertex of $G$ in the interior of $C$ must be paired by $M$ to another vertex in the interior of $C$. 

Let $D' \subset D$ be the closed disc consisting of $C$ and its interior.
By \reflem{where_Tait_spine_coincide}, $G$ and $T$ coincide on $D'$. 
Let $G'$ be the subgraph of $G$ (or $T$) lying in $D'$ (i.e. in the interior or on the boundary of $D'$). Then $M$ restricts to a matching of $G'$. Since the matching pairs black and white vertices then the numbers of black and white vertices in $G'$ must be equal; let this number be $V$.

Let the number of edges of $G'$ be $E$. Each edge of $G'$ is incident to precisely one white vertex. Each white vertex in $T$ has degree $4$. Each white vertex of $G'$ in the interior of $C$ thus has degree $4$, but the white vertices of $C$ may have degree less than $4$ in $G'$, because the edges of $T$ counted by $E_{\partial W}$ are not in $G'$. Counting edges of $G'$ via degrees of white vertices we thus have $4V = E + E_{\partial W}$.

On the other hand, by \reflem{spine_is_2-cell}, each non-outer face of $G$ is a (possibly degenerate) quadrilateral, so each face of $G'$ in $D'$ is also a quadrilateral. Each edge of $G'$ has a face on both sides in $D'$ (possibly the same face twice, in a degenerate quadrilateral), except for the $2L$ edges along $C$, which only have a face in $D'$ on one side. Letting $F$ be the number of faces of $G'$ in $D'$, we have $4F = 2E - 2L$.

As $G'$ provides a 2-cell decomposition of the disc $D'$ with $2V$ vertices, $E$ edges and $F$ faces, Euler's formula gives $2V-E+F = 1$. 

The three equations of the previous three paragraphs can be written as equating $E$ to $4V - E_{\partial W} = 2F+L = 2V+F-1$ respectively. The first equality gives $4V-2F = E_{\partial W} + L$, and the second equality gives $4V-2F = 2L+2$. Thus we obtain $E_{\partial W} + L = 2L+2$ and hence the desired result. 
\end{proof}

The idea of the following lemma is that, under certain circumstances, if $G$ contains an alternating cycle which is not elementary, then there is a smaller alternating cycle inside it (but possibly for a different matching).

\begin{lem}
\label{Lem:universe_key_construction}
Let $M \in \mathscr{M}_G$. Suppose there is a vertex-simple alternating cycle $C$ of $G$ relative to $M$,
containing a white vertex $v_0$ with an incident edge in the interior of $C$. 
Then there exist $M' \in \mathscr{M}_G$ and a directed path $P$ in $G$, say $(e_1, v_0, v_1), \ldots, (e_n, v_{n-1}, v_n)$, of length $n \geq 1$, starting at $v_0$, such that the following conditions hold.
\begin{enumerate}
\item The edges $e_j$ are all distinct and all lie in the interior of $C$.
\item The vertices $v_0, v_1, \ldots, v_n$ are all distinct.
\item The vertices $v_1, \ldots, v_{n-1}$ all lie in the interior of $C$.
\item The terminal vertex $v_n$ of $P$ lies on $C$.
\item $M'$ agrees with $M$ on the edges of $C$ and in the exterior of $C$.
\item $P$ is alternating with respect to $M'$.
\end{enumerate}
\end{lem}

\begin{proof}
We algorithmically construct $P$, possibly adjusting $M$ along the way.

Let $\mathscr{P}$ be the set of all directed paths $P$ given by $(e_1, v_0, v_1), \ldots, (e_n, v_{n-1}, v_n)$ for some $n \geq 1$, such that (i)--(iii) of the statement hold, and such that $v_n$ lies on or in the interior of $C$.
Clearly $\mathscr{P}$ is a finite set. Roughly, it consists of directed paths, without self-intersections, which start at $v_0$, proceed inside $C$, and which stop if they hit $C$. 

We introduce a total order $\preceq$ on $\mathscr{P}$, in a lexicographic fashion, as follows: two paths $P_1, P_2$ in $\mathscr{P}$ satisfy $P_1 \preceq P_2$ iff at the first edge where they differ, $P_2$ proceeds counterclockwise of $P_1$. Roughly, the idea is that a path is lesser in this order if it turns ``more clockwise" or ``more to the right" at each step. 
For instance, the edge $e_1$ of a path in $\mathscr{P}$ is among the set of edges of $G$ departing $v_0$ into the interior of $C$. As there are no loops in $G$, this set may be totally ordered, say in counterclockwise order. (Note that as $v_0$ is a white vertex of the spine, incident to two edges of $C$, there are at most two possibilities for $e_1$.)
The path will be lesser in $\preceq$ if $e_1$ takes the ``more clockwise" or ``more rightwards" option, and greater  if it takes the ``more counterclockwise" or ``more leftwards" option.
Similarly, after choosing an initial sequence of $m$ directed edges $(e_1, v_0, v_1), \ldots, (e_{m}, v_{m-1}, v_{m})$, the next edge $e_{m+1}$ is among the set of edges departing $v_m$ other than $e_m$. This set of edges may again be totally ordered counterclockwise. 

Our approach is to try to build up elements of $\mathscr{P}$, adding one edge at a time, making our path alternating with respect to some matching $M'$, which initially agrees with $M$ but may be adjusted at each step, but so that $M'$ only ever differs from $M$ in the interior of $C$. Eventually our path will hit $C$ and satisfy conditions (i)--(vi) as required. We build the path by using the matching where we must, and a greedy algorithm otherwise, turning as clockwise as possible (hence $\preceq$-minimally) at each step. Roughly, the idea is that if this algorithm produces a path which intersects itself, then we can adjust the matching $M'$, and obtain a path which is lesser with respect to $\preceq$, which can be continued without self-intersection. Since $\preceq$ is a total order on a finite set, this process must terminate by hitting $C$, yielding the desired path.

So, start with $M' = M$; we construct $e_1$ to be the $\preceq$-minimal (i.e. most clockwise) edge departing $v_0$ into the interior of $C$. In the matching $M=M'$, $v_0$ is matched via an edge of $C$. Thus $e_1 \notin M'$.  Let $v_1$ be the vertex of $e_1$ other than  $v_0$. As $G$ is bipartite, $v_1$ is black (hence $v_1 \neq v_0$). If $v_1$ is a vertex of $C$ then we may take $P$ to consist simply of $(e_1, v_0, v_1)$, which is trivially alternating with respect to $M'$, and we are done.

Otherwise, $v_1$ lies in the interior of $C$. We choose $e_2$ to be the unique edge of $M'$ incident to $v_1$. Since $v_1$ is in the interior of $C$, so is $e_2$. Let the white endpoint of $e_2$ be $v_2$. Then $v_2$ also lies in the interior of $C$, since vertices of $C$ are matched via edges of $C$. Thus $v_2$ is distinct from $v_0$.

Note that $v_2$ must have degree $4$ in $G$, since it is a white vertex in the interior of $C$ (where $G$ coincides with $T$ by \reflem{where_Tait_spine_coincide}). Hence there exist $3$ edges of $G$ incident to $v_2$ other than $e_2$. We choose $e_3$ to be $\preceq$-minimal (i.e. most clockwise) among them. As $e_2 \in M'$ and $e_3$ share the vertex $v_2$ we have $e_3 \notin M'$. By construction, there is no edge of $G$ departing $v_2$ clockwise of $e_3$ and counterclockwise of $e_2$. Let $v_3$ be the black vertex of $e_3$. If $v_3$ is a vertex of $C$ then we may take $P$ to consist of $e_1, e_2, e_3$ and $M' = M$, and we are done. Otherwise, $v_3$ lies in the interior of $C$ and we continue the process.

Indeed, suppose we have constructed a directed path $(e_1, v_0, v_1), \ldots, (e_{2k-1}, v_{2k-2}, v_{2k-1})$ for some integer $k \geq 1$, and a matching $M'$ differing from $M$ only in the interior of $C$, such that the following conditions hold.
\begin{enumerate}
    \item[(a)] The edges $e_1, \ldots, e_{2k-1}$ are all distinct and in the interior of $C$.
    \item[(b)] The vertices $v_0, \ldots, v_{2k-1}$ are all distinct.
    \item[(c)] The initial vertex $v_0$ is the chosen white vertex of $C$.
    \item[(d)] The vertices $v_1, \ldots, v_{2k-1}$ all lie in the interior of $C$.
    \item[(e)] $e_j \in M'$ for $j$ even, and $e_j \notin M'$ for $j$ odd.
    \item[(f)] For each integer $j$ such that $1 \leq j \leq k-1$, there is no edge of $G$ departing $v_{2j}$ clockwise of $e_{2j+1}$ and counterclockwise of $e_{2j}$.
\end{enumerate}  
Condition (f) says that $e_{2j+1}$ is chosen to be $\preceq$-minimal at each $v_{2j}$. 
We attempt to extend the path, adding new distinct vertices $v_{2k}, v_{2k+1}$ and edges $e_{2k}, e_{2k+1}$ similarly.

To add $e_{2k}$, we note that since $e_{2k-1} \notin M'$, the edge of $M'$ containing the black vertex $v_{2k-1}$ is not among the edges already chosen, so we let $e_{2k}$ be the edge of $M$ containing $v_{2k-1}$, and let $v_{2k}$ be its white endpoint. As $v_{2k-1}$ lies in the interior of $C$, so does $e_{2k}$. The vertex $v_{2k}$ cannot be equal to any of the previous vertices $v_0, \ldots, v_{2k-1}$. For if $v_{2k} = v_j$ for some $j < 2k$, then $e_{2k}$ is the edge of $M$ incident with $v_j$; but all the edges of $M$ incident to previous $v_j$ for $1 \leq j \leq 2k-2$ have already been chosen as previous edges, and the edge of $M$ incident with $v_0$ lies on $C$, while $e_{2k}$ is in the interior of $C$. Further, $v_{2k}$ lies in the interior of $C$, since it is matched via $M$ to $v_{2k-1}$ via $e_{2k}$, and interior vertices are matched with interior vertices. Hence, appending the directed edge $(e_{2k}, v_{2k-1}, v_{2k})$, we have an alternating directed path with respect to $M'$.

Now consider adding $e_{2k+1}$. As $v_{2k}$ is a white vertex in the interior of $C$, the edge $e_{2k+1}$ must be chosen among those edges of $G$ incident to $v_{2k}$ other than $e_{2k}$. By \reflem{where_Tait_spine_coincide}, $v_{2k}$ has degree $4$ in $G$. Thus, there are $3$ choices for $e_{2k+1}$, and we choose $e_{2k+1}$ to be $\preceq$-minimal (most clockwise) among them. Let $v_{2k+1}$ be the black vertex of $e_{2k+1}$. The edge $e_{2k+1}$ by construction is distinct from all previous edges $e_1, \ldots, e_{2k}$. Since $e_{2k} \in M'$ and $e_{2k+1}$ share the vertex $v_{2k}$ we have $e_{2k+1} \notin M'$. Moreover, by construction, there is no edge of $G$ departing $v_{2k}$ clockwise of $e_{2k+1}$ and counterclockwise of $e_{2k}$.

If $v_{2k+1}$ is distinct from all $v_0, \ldots, v_{2k}$ and in the interior of $C$, then we have an alternating directed path
\begin{equation}
\label{Eqn:extended_path}
(e_1, v_0, v_1), \ldots, (e_{2k+1}, v_{2k}, v_{2k+1})
\end{equation}
satisfying (a)--(f) above, with $k$ replaced with $k+1$, and we proceed inductively.

If $v_{2k+1}$ lies on $C$, then the black vertex $v_{2k+1}$ is disjoint from all $v_0, \ldots, v_{2k}$, since $v_0$ is white and all $v_1, \ldots, v_{2k}$ all lie  in the interior of $C$. Then the directed path \refeqn{extended_path} satisfies the desired conditions (i)--(vi), and we are done.

The remaining case is that $v_{2k+1}$ coincides with some previous vertex $v_j$ in the interior of $C$. As the vertices alternate in colour, we have $v_{2k+1} = v_{2j+1}$ for some $0 \leq j < k$. Thus the directed path $(e_{2j+2}, v_{2j+1}, v_{2j+2}), \ldots, (e_{2k+1}, v_{2k}, v_{2k+1} = v_{2j+1})$ forms a non-null vertex-simple directed cycle $C'$, alternating relative to $M'$. Note that $C'$ lies entirely in the interior of the cycle $C$. Let $D'$ be the closed disc consisting of $C'$ and its interior, so $D'$ lies entirely in the interior of $C$.

Now $C'$ may be oriented clockwise or counterclockwise around the boundary of $D'$. We claim in fact that the orientation is clockwise. Suppose to the contrary that the orientation is counterclockwise; see \reffig{if_cycle_counterclockwise}. Then, by construction, at each white vertex $v$ of $C'$, 
every incident edge of $G$ (and $T$) either forms part of $C'$ or departs into the interior of $C'$. 
Thus the vertex-simple non-null alternating cycle $C'$ and the matching $M'$ thus satisfy the hypotheses of \reflem{long_alternating_cycles_require_escape}, and $E_{\partial W} = 0$. Thus the length of $C'$ is negative, a contradiction.

\begin{figure}
\begin{center}
\includegraphics[width=\textwidth]{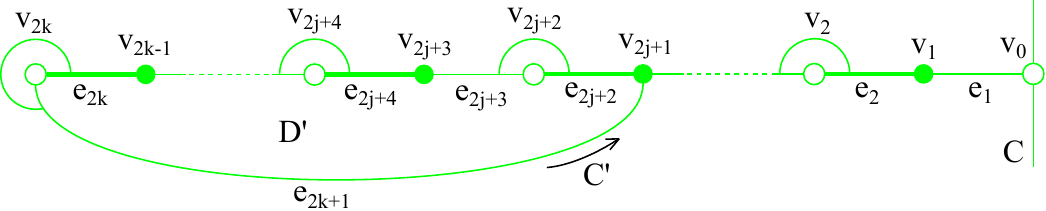}
\end{center}
    \caption{Counterclockwise $C'$. Edges of $M'$ are thickened. Semicircular arcs around white vertices indicate that no edges depart the vertex in that direction.}
    \label{Fig:if_cycle_counterclockwise}
\end{figure}

Hence $C'$ is oriented counterclockwise. We then adjust $M'$ by twisting at $C'$ (\refdef{twisting}), as shown in \reffig{if_cycle_clockwise_change_matching}. Note that $M'$ only changes in the interior of $C$. 

\begin{figure}
\begin{center}
\includegraphics[width=\textwidth]{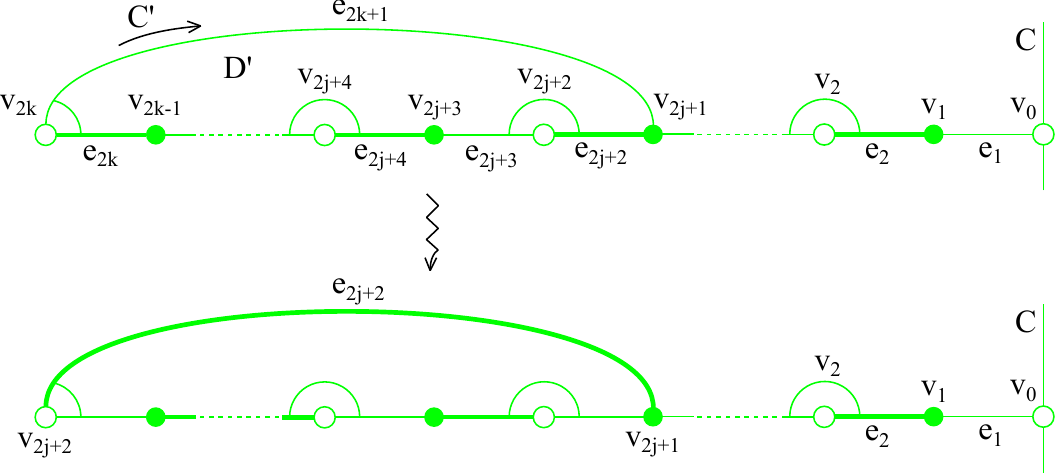}
\end{center}
    \caption{Clockwise $C'$. We twist $M'$ on $C'$, and adjust the path.}
    \label{Fig:if_cycle_clockwise_change_matching}
\end{figure}

We now restart our algorithm on the adjusted matching $M'$. Because the change is only in $C'$, the algorithm applied to the adjusted $M'$ will produce the same first $2j+1$ directed edges $(e_1, v_0, v_1), \ldots, (e_{2j+1}, v_{2j}, v_{2j+1})$. However, in choosing the next edge $e_{2j+2}$ the algorithm will choose the edge of the adjusted matching (which is the edge previously known as $e_{2k+1}$). This is clockwise of, i.e. $\preceq$-less than, the previous choice for $e_{2j+2}$, and hence any path so obtained form the adjusted matching will be lesser with respect to $\preceq$ than a path obtained from the matching prior to its adjustment. We then proceed with the algorithm.

Thus, applying the algorithm, the algorithm either terminates, or we adjust the matching and restart, obtaining paths which are lesser with respect to $\preceq$. Since $\preceq$ is a total order on a finite set, this process will in fact terminate and we obtain a path with the desired properties.
\end{proof}

\subsection{Plane transpositions on universes are Kauffman transpositions}

\begin{proof}[Proof of \refprop{universe_n-transpositions} and \refthm{clock_lattices_isomorphic}]
By \reflem{transpositions_and_twisting}, a plane $n$-transposition on a state of $U$ is equivalent to a twisting up or down on a positive or negative cycle of $G_0$ of length $2n$. By \reflem{Kauffman_transposition_twisting}, a Kauffman transposition on a state of $U$ is equivalent to twisting on an alternating cycle $A$ of $G$ of length $4$, such that at each white vertex $v$ of $A$, no edges of $G$ depart from $v$ into the interior of $C$. 

So, let $C$ be a positive cycle of $G_0$. We will show that $C$ has length $4$, and there is no edge of $G$ departing from a white vertex of $C$ into the interior of $C$.

This $C$ is an elementary cycle of $G_0$, forming the outer boundary of some bounded face of $G_0$. It is also an alternating cycle of $G_0$ with respect to a matching $M \in \mathscr{M}_G = \mathscr{M}_{G_0}$, hence vertex-simple by \reflem{alternating_elementary_cycles_super-simple}.

As $C$ is elementary in $G_0$, it is clear that there is no edge of $G_0$ departing from a white vertex into the interior of $C$. It is less clear that there is no edge of $G$ doing so, and this is what we will prove.

Suppose for contradiction that there exists a white vertex $v_0$ of $C$ at which an edge of $G$ departs into the interior of $C$.  Then we may apply \reflem{universe_key_construction}, so there exists a matching $M'$ on $G$, differing from $M$ only in the interior of $C$, and a directed path $P$, alternating relative to $M'$, departing from $v_0$, through the interior of $C$, passing through distinct vertices and edges $e_1, \ldots, e_n$, to arrive at a distinct vertex $v_n$ of $C$. 

All vertices in $C$ are matched in $M$, hence also in $M'$, via edges of $C$. Since the extremal edges $e_1, e_n$ of $P$ are interior to $C$ but incident to vertices of $C$, they do not lie in $M'$. In particular, $n$ is odd and $v_n$ is black.

The vertices $v_0, v_n$ split the cycle $C$ into two paths, both alternating relative to $M'$. As $v_0, v_n$ have opposite colour, these two paths have odd length. Hence one of these two paths begins and ends with edges of $M'$; denote it by $P_1$. Joining $P$ and $P_1$ we thus obtain an alternating cycle $C^!$ relative to $M'$, enclosing a strictly smaller number of faces of $G$. As $P$ and $P_1$ contain distinct vertices except at their common endpoints, $C^!$ is vertex-simple. Hence by \reflem{when_twisting_works} we may adjust $M'$ by twisting at $C^!$ to obtain a matching $M^!$.

Note that the edges $e_1, e_n$ are not in $M'$ but lie in the cycle $C^!$. Hence both $e_1, e_n$ lie in $M^!$. In particular, $e_1$ lies in some matching of $G$. Thus $e_1$ is not removed in constructing the reduction of $G$. So $e_1$ is also an edge of $G_0$. But $e_1$ is not in the cycle $C$, and in fact departs from the vertex $v_0$ of $C$ into the interior of $C$. Hence $C$ is not an elementary cycle of $G_0$. This is a contradiction.

Hence, at each white vertex $v$ of $C$, no edge of $G$ departs into the interior of $C$. In the Tait graph $T$, each white vertex has degree $4$, and $T$ coincides with $G$ inside $C$ (\reflem{where_Tait_spine_coincide}). At $v$ there are $2$ incident edges from the cycle $C$. Therefore there are precisely $2$ edges of $T$ departing from $v$ into the exterior of $C$. Letting $2L$ be the length of $C$, we have $L$ white vertices on $C$, and hence $E_{\partial W} = 2L$ edges of $T$ departing from white vertices of $C$ into the exterior of $C$. Applying \reflem{long_alternating_cycles_require_escape} to $C$ we then obtain $2L=L+2$, so $L=2$, and $C$ has length $4$. 

We have now shown that $C$ has length $4$, and that at each white vertex, no edges of $G$ depart into the interior of $C$. The argument for negative cycles is similar, completing the proof.
\end{proof}

\section{Multiverses in positive genus}
\label{Sec:positive_genus}

\subsection{Orientations on a dual of a spine}
\label{Sec:orientations_on_dual_spine}

In \refsec{lattices_matchings} we gave an outline of Propp's proof of \refthm{Propp_matching}, which provides a distributive lattice structure on the set of matchings of certain plane bipartite graphs $G$ (generalised to all finite plane bipartite graphs in \refprop{finite_bipartite_plane_lattice}). The proof proceeds by considering orientations on the dual $G^\perp$ of the bipartite plane graph $G$: it has a \emph{standard} orientation, which orients each edge of $G^\perp$ clockwise around black vertices of $G$ and counterclockwise around white vertices of $G$; and given a matching $M$ of $G$ there is a \emph{prescribed} orientation on $G^\perp$, obtained from the standard orientation by switching orientations on edges of $G^\perp$ dual to those of $M$. In a similar way, we now define standard and prescribed orientations on a dual of a spine of a multiverse, using the \emph{basic cycles} of $G^\perp$ of \refdef{basic_cycle}. 

Throughout this section, $(U, \Sigma, \mathscr{F}, G)$ is a framed multiverse, and $G^\perp$ is a dual of the spine $G$.

As noted in \refsec{dual_of_spine}, every edge $e^\perp$ of $G^\perp$ appears in two basic cycles, namely in the basic cycles around $v$ and $w$, the endpoints of the dual edge $e$ of $G$. As $G$ is bipartite, one of $v,w$ is black and the other is white. In \refdef{orientations_basic_cycles} we defined \emph{clockwise} and \emph{counterclockwise} orientations on a basic cycle. If $e^\perp$ is oriented as in the clockwise basic cycle about $v$, then it will be counterclockwise in the basic cycle around $w$, and vice versa.

\begin{defn}[Standard, prescribed orientations on dual of spine] 
\label{Def:standard_orientation} \
\begin{enumerate}
\item 
The \emph{standard orientation} $R_0$ of $G^\perp$ is the orientation where each basic cycle around a black vertex of $G$ is clockwise, and each basic cycle around a white vertex of $G$ is counterclockwise.
\item 
For each matching $M$ of $G$, the \emph{prescribed orientation} $R_M$ of $G^\perp$ corresponding to $M$ is obtained from $R_0$ by switching the orientations of edges of $G^\perp$ dual to edges of $M$.
\end{enumerate}
\end{defn}
In light of (i) above, we can refer to the two orientations on an edge of $G^\perp$ as its \emph{standard} and \emph{nonstandard} orientations.

Recall in \refdef{orientations_notation} we defined $\mathscr{R}_X^c$ to be the set of orientations on a graph $X$ with circulation $c$. We add a $\mathscr{P}$ to denote ``prescribed" as follows.
\begin{defn}[Set of prescribed orientations]
\label{Def:prescribed_orientations}
Let $c$ be a feasible circulation function on $G^\perp$.
\begin{enumerate}
\item 
The set of all prescribed orientations of $G^\perp$ with circulation $c$ is denoted $\mathscr{PR}^c_G$.
\item
The set of all prescribed orientations of $G^\perp$ is denoted $\mathscr{PR}_G$.
\end{enumerate}
\end{defn}
Thus $\mathscr{PR}^c_G \subseteq \mathscr{R}^c_{G^\perp}$.
(We drop the $\perp$ from $\mathscr{PR}_{G^\perp}$ to avoid cumbersome notation.)
In other words,
\begin{equation}
\label{Eqn:prescribed_orientations}
\mathscr{PR}_G = \left\{ R_M \mid M \in \mathscr{M}_G \right\}
\quad \text{and} \quad
\mathscr{PR}_G = \bigsqcup_c \mathscr{PR}^c_G,
\end{equation}
the disjoint union over feasible circulation functions $c$ on $G^\perp$.

\begin{lemma}\label{Lem:matchings_and_orientations}
The map $M \mapsto R_M$ yields a bijection $\mathscr{M}_G \To \mathscr{PR}_G$.
\end{lemma}

\begin{proof}
By definition we have a surjection,
and from a prescribed orientation $R_M$, the matching $M$ can be recovered by observing where $R_M$ differs from the standard orientation $R_0$.
\end{proof}

Combining
\reflem{states_and_matchings} and
\reflem{matchings_and_orientations}, we have bijections
\begin{equation}
\label{Eqn:triple_bijection}
\mathscr{S}_U \cong \mathscr{M}_G \cong \mathscr{PR}_G.
\end{equation}

\subsection{Viable circulation functions}

Any orientation $R$ on a graph yields a circulation function (\refdef{circulation_function}) $c_R$. Recall a function $\mathscr{C}_X \to \Z$ is a \emph{feasible circulation function} if it is the circulation function $c_R$ (\refdef{circulation}) of some orientation $R$ on $X$. 

As usual, let $(U, \Sigma, \mathscr{F}, G)$ be a framed multiverse, and $G^\perp$ a dual of $G$.
In light of the bijections \refeqn{triple_bijection}, we can associate circulation functions to states and matchings as well as orientations.
\begin{defn}[Circulation of a state or matching]
\label{Def:circulation_of_state_matching}
Suppose $S \in \mathscr{S}_U$,  $M \in \mathscr{M}_G$ and $R \in \mathscr{PR}_G$ correspond under \refeqn{triple_bijection}. Let $c \colon \mathscr{C}_{G^\perp} \To \Z$ be the circulation function of $R$.
\begin{enumerate}
\item
The \emph{circulation} of $S$ is $c$. The set of all states of $U$ with circulation $c$ is denoted $\mathscr{S}^c_U$.
\item 
The \emph{circulation} of $M$ is $c$. The set of all matchings of $G$ with circulation $c$ is denoted $\mathscr{M}_G^c$.
\end{enumerate}
\end{defn}
Thus
\begin{equation}
\label{Eqn:states_matchings_orientations}
\mathscr{S}_U = \bigsqcup_c \mathscr{S}_U^c
\quad \text{and} \quad
\mathscr{M}_G = \bigsqcup_c \mathscr{M}_G^c,
\end{equation}
the disjoint union over feasible circulation functions $c$ on $G^\perp$.

We now define a useful subset of the feasible circulation functions on $G^\perp$.

\begin{defn}[Viable circulation]
\label{Def:viable_circulation}
A function $c \colon \mathscr{C}_{G^{\perp}} \To \Z$ is a \emph{viable} circulation function if it is the circulation of a prescribed orientation of $G^\perp$. 
\end{defn}
Thus a function $\mathscr{C}_{G^\perp} \To \Z$ is a viable circulation function if and only if it is the circulation of a state, or a matching.
By definition, a viable circulation function is feasible. 
In \refeqn{prescribed_orientations} and \refeqn{states_matchings_orientations}, the nonzero sets are those with viable $c$, so the disjoint unions can in fact be taken over viable $c$.

All viable circulation functions are the same on basic cycles, as we now see. A more general version of the following lemma is stated by Propp in \cite{Propp}, in the planar context;
we give the version we need.
\begin{lemma}\label{Lem:standard_circulation}
Let $C$ be a basic cycle of $G^\perp$ around a vertex $v$ of $G$ with degree $d_v$, oriented counterclockwise. Let $c$ be a viable circulation function on $G^\perp$.
 Then
\[
c \left( C \right) = \left\{
\begin{array}{ll} 2-d_v & \text{if $v$ is black}, \\
d_v-2 & \text{if $v$ is white}.
\end{array} \right.
\]
\end{lemma}
Although viable circulations are fixed on basic cycles, they may differ on other cycles. See for example \reffig{Hasse_diagram_genus-1_example}, which shows all matchings on $G$ and prescribed orientations on $G^\perp$ on a specific multiverse; all orientations not connected by arrows have distinct circulations.

\begin{figure}
\begin{center}
\includegraphics[width=\textwidth]{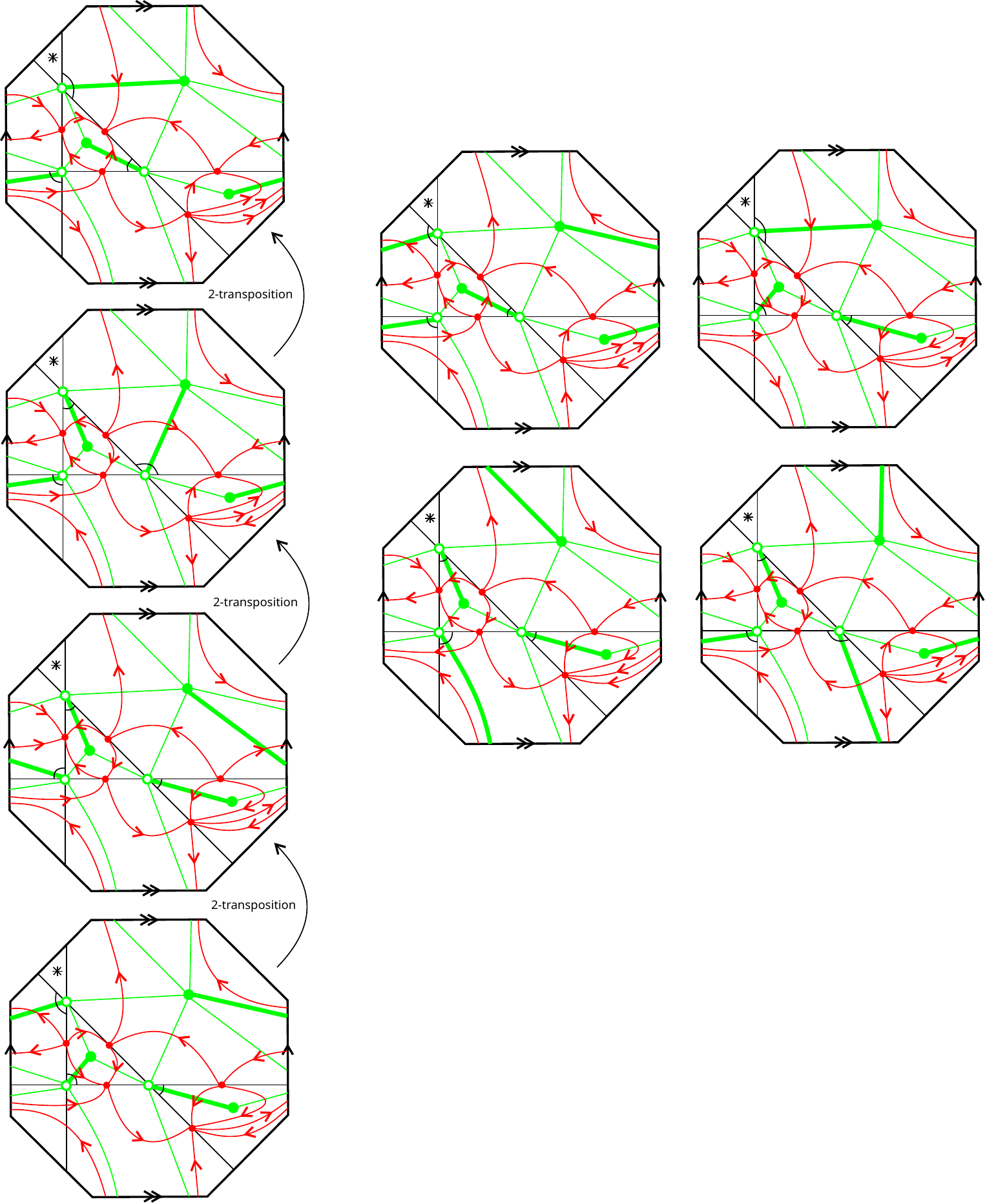}
\end{center}
    \caption{Hasse diagram of a multiverse on a punctured torus. }
    \label{Fig:Hasse_diagram_genus-1_example}
\end{figure}

\begin{proof}
Let $M \in \mathscr{M}^c_G$ correspond to $R=R_M \in \mathscr{PR}^c_G$.
In $M$ there is precisely one edge incident to $v$, say $e$; let its dual edge in $G^\perp$ be $e^\perp$. 
The length $|C|$ of $C$ is $d_v$.
So if $v$ is black, then in $R$, all edges of $C$ are oriented clockwise except $e^\perp$, hence (in the notation of \refdef{circulation}) $| C_{R}^+ | = 1$ and $| C_{R}^- | = d_v - 1$, thus $c (C) = 2-d_v$.
Similarly, if $v$ is white, then in $R$, all edges are oriented  counterclockwise except $e^\perp$, so $c(C) = d_v - 2$. 
\end{proof}

\begin{lem}
\label{Lem:viable_orientation_prescribed}
If $c$ is a viable circulation function then $\mathscr{PR}_G^c = \mathscr{R}_G^c$. 
\end{lem}
Since by definition $\mathscr{PR}_G^c \subseteq \mathscr{R}_G^c$, the content of this lemma is that any orientation of $G^\perp$ with viable circulation is a prescribed orientation. In other words, if $R \in \mathscr{R}_G^c$, then $R=R_M$ for some $M \in \mathscr{M}_G$. 

\begin{proof}
Let $R \in \mathscr{R}_G^c$, where $c$ is viable; we explicitly construct a matching $M \in \mathscr{M}_G$ such that $R=R_M$. Consider a black vertex $b$ of $G$, with degree $d_b$, and the counterclockwise basic cycle $C$ of $G^\perp$ about $b$. By \reflem{standard_circulation}
$c(C) = 2-d_b$. So
$|C_R^+|-|C_R^-|=2-d_b$ and $|C_R^+|+|C_R^-|=d_v$, hence $|C_R^+|=1$ and $|C_R^-|=d_b - 1$. Hence, there is precisely one directed edge in $C$ that is forward relative to $R$. Since $C$ is assigned a counterclockwise direction around $b$, there is precisely one edge $e^\perp$ in $C$ that is directed counterclockwise around $b$ in $R$. We add the edge $e$ of $G$ dual to $e^\perp$ to the matching $M$.

Similarly, consider a white vertex $w$ of degree $d_w$ and let $D$ be the basic counterclockwise cycle of $G^\perp$ around $w$. Again 
$c(D)=d_w-2$ and 
$|D_R^+|+|D_R^-|=d_w$ so 
$|D_R^+|=d_w-1$ and $|D_R^-|=1$. Thus, there is exactly one edge in $D$ that is backward relative to $R$. Since $D$ is counterclockwise direction, there is exactly one edge $e^\perp$ in $D$ that is directed clockwise around $w$ in $R$. We add the edge $e$ of $G$ dual to $e^\perp$ to the matching $M$.

As $G$ is bipartite, every edge has a black and white vertex. A directed edge of $G^\perp$ lies in two basic cycles, about a black and a white vertex. It is oriented clockwise with about one and counterclockwise about the other. Thus the collection of edges chosen for $M$ via their white vertices coincides with the  collection of edges chosen for $M$ at black vertices. Moreover,  each vertex of $G$ is adjacent to exactly one edge of $M$, so $M$ is a matching. By construction, $R=R_M$ as desired.
\end{proof}

\subsection{Forced and forbidden corners and edges}
Again let $(U, \Sigma, \mathscr{F}, G)$ be a framed multiverse, and $G^\perp$ a dual of $G$.

In \refdef{forced_forbidden_edges} we defined directed edges of a graph as $c$-forced and $c$-forbidden, for a feasible circulation function $c$. A feasible circulation function $c \colon \mathscr{C}_{G^\perp} \To \Z$ on $G^\perp$ thus determines $c$-forced and $c$-forbidden directed edges of $G^\perp$.

In \refdef{forced_forbidden} we defined the notion of forced and forbidden edges, based on matchings, so $G$ may have forced and forbidden edges. In \refdef{forced_forbidden_corner} we defined the notion of forced and forbidden corners in $U$, based on states. We showed in \reflem{forced_forbidden_corners_edges} that these notions of forced and forbidden on edges of $G$ and states of $U$ are equivalent.

We now also define corners of $U$ and edges of $G$ as forced and forbidden by a circulation function, following \refdef{forced_forbidden_edges}, using \refdef{circulation_of_state_matching} and the bijections of \refeqn{triple_bijection}.
\begin{defn}[$c$-forced and $c$-forbidden edges, corners]
\label{Def:forced_forbidden_edges_spine}
Let $c \colon \mathscr{C}_{G^\perp} \To \Z$ be a viable circulation function.
\begin{enumerate}
\item 
An edge $e$ of $G$ is
\begin{enumerate}
\item
\emph{$c$-forced} if $e$ lies in every matching $M \in \mathscr{M}^c_G$;
\item 
\emph{$c$-forbidden} if $e$ does not lie in any matching $M \in \mathscr{M}^c_G$.
\end{enumerate}
\item 
A corner $\alpha$ at an interior vertex of $U$ is
\begin{enumerate}
\item 
\emph{$c$-forced} if every state $S \in \mathscr{S}^c_U$ has a marker at $\alpha$;
\item 
\emph{$c$-forbidden} if no state $S \in \mathscr{S}^c_U$ has a marker at $\alpha$.
\end{enumerate}
\end{enumerate} 
\end{defn}
(Although it may seem that \refdef{forced_forbidden_edges} and \refdef{forced_forbidden_edges_spine}(i) provide conflicting definitions of $c$-forced and $c$-forbidden edges, the former is for directed edges, the latter for undirected edges. We apply the former to $G^\perp$ and the latter to $G$.)
The same proof as \reflem{forced_forbidden_corners_edges} immediately yields the following.
\begin{lem}
\label{Lem:c-forced_forbidden_corners_edges}
A corner $\alpha$ of $U$ in an unstarred face is $c$-forced (resp. $c$-forbidden) if and only if the corresponding edge $e$ of $G$ is $c$-forced (resp. $c$-forbidden).
\qed
\end{lem}

A corner of $U$ is forced (resp. forbidden) in the sense of \refdef{forced_forbidden_corner}, if and only if it is $c$-forced (resp. $c$-forbidden) for all viable circulation functions $c$. (The circulation of a state is always viable, by \refdef{viable_circulation} and subsequent comments.) Similarly, an edge of $G$ is forced (resp. forbidden) in the sense of \refdef{forced_forbidden}, if and only if it is $c$-forced (resp. $c$-forbidden) for all viable $c$.

Following Propp's \refprop{Propp's_prop}, which characterised forced or forbidden edges in terms of accessibility classes, we have the following.
\begin{prop}\label{Prop:all_nothing}
Let $c \colon \mathscr{C}_{G^\perp} \To \Z$ be a viable circulation function. 
Let $e$ be an edge of $G$, corresponding to a corner $\alpha$ of $U$, and let $e^\perp$ be the dual of $e$ in $G^\perp$. The following are equivalent.
\begin{enumerate}
\item 
$\alpha$ is $c$-forced or $c$-forbidden.
\item 
$e$ is $c$-forced or $c$-forbidden. 
\item
$e^\perp$, directed arbitrarily, is $c$-forced or $c$-forbidden.
\item 
The endpoints of $e^\perp$ belong to the same accessibility class of $c$.
\end{enumerate}
\end{prop}

\begin{proof}
The equivalence of (i) and (ii) is immediate from \reflem{c-forced_forbidden_corners_edges}.

Note that $e$ lies in a matching $M \in \mathscr{M}_G^c$ precisely when $e^\perp$ has nonstandard orientation in the corresponding prescribed orientation $R_M \in \mathscr{PR}_G^c$. Suppose $e$ is $c$-forced; the argument when $e$ is $c$-forbidden is similar. Then $e^\perp$ has nonstandard orientation in every prescribed orientation of $G^\perp$ with circulation $c$, i.e in every $R \in \mathscr{PR}_G^c$. But by \reflem{viable_orientation_prescribed}, $\mathscr{PR}_G^c = \mathscr{R}_G^c$. That is, every orientation of $G^\perp$ with circulation $c$ is a prescribed orientation. Thus $e^\perp$ has nonstandard orientation in every orientation of $G^\perp$ with circulation $c$, i.e. in every $R \in \mathscr{R}_G^c$. Hence $e^\perp$ with its standard orientation is $c$-forbidden, and $e^\perp$ with its nonstandard orientation is $c$-forced. Either way, $e^\perp$ is $c$-forced or $c$-forbidden.  Thus (ii) implies (iii). 

For the converse, suppose $e^\perp$ has the same direction in all orientations $R$ of $G^\perp$ with circulation $c$, i.e. standard or non-standard. All such orientations are prescribed orientations, and accordingly $e$ either lies in none or all of the corresponding matchings $M$, so is $c$-forced or $c$-forbidden.

\refprop{Propp's_prop} gives the equivalence of (iii) and (iv). 
\end{proof}

\subsection{Twisting along subsurfaces}
\label{Sec:twisting_surfaces}

In \refdef{pos_neg_cycle} we defined positive and negative cycles on a plane bipartite graph $G$ relative to a matching $G$. We now generalise this to multiverses. On a multiverse, a face need not have a unique outermost boundary component, and so we instead define \emph{subsurfaces} to be positive or negative.

We consider subsurfaces of a multiverse surface which are, roughly, unions of faces of the spine. However, we want compact surfaces, and taking a closure of a union of faces of $G$ might not be a subsurface. For instance, if two faces share a common vertex $v$ but no a common edge, then their closure contains $v$ but not a neighbourhood of $v$. Thus we make the following definition.
\begin{defn}[Face subsurface]
\label{Def:face_subsurface}
Let $G$ be a finite graph embedded in the interior of a connected compact orientable surface $\Sigma$, possibly with boundary. A \emph{face subsurface} of $G$ is a connected subsurface $\Delta$ of $\Sigma$ (possibly with boundary) which is the closure in $\Sigma$ of a union of faces $\Delta_1, \ldots, \Delta_n$ of $G$, i.e. $\Delta = \overline{\bigcup_{j=1}^n \Delta_j}$.
\end{defn}
As $G$ is embedded in the interior of $\Sigma$, each boundary component of a face subsurface $\Delta$ is either a boundary component of $\Sigma$, or a
cycle of $G$. Additionally, every edge $e$ of $G$ lying in a face subsurface $\Delta$ of $G$ either lies entirely on $\partial \Delta$, or has its interior entirely in the interior of $\Delta$.
\begin{defn}
\label{Def:face_subsurface_bdy_edges}
A boundary component of a face subsurface $\Delta$ is:
\begin{enumerate}
\item 
A \emph{$\Sigma$-boundary component} if it is a component of $\partial \Sigma$. These components are denoted $\partial_\Sigma \Delta$.
\item
A \emph{$G$-boundary component} if it is a 
cycle of $G$. These components are denoted $\partial_G \Delta$.
\end{enumerate}

An edge $e$ of $G$ in $\Delta$ is:
\begin{enumerate}
\item 
A \emph{$G$-boundary edge}, if it lies in $\partial \Delta$ (indeed $\partial_G \Delta)$.
\item 
An \emph{interior edge}, if it lies in the interior of $\Delta$.
\end{enumerate}
\end{defn}
An edge $e$ of $G$ in $\Delta$ is a $G$-boundary or interior edge accordingly as it has faces $\Delta_j$ of $\Delta$ on just one side, or on both sides.

\begin{lem}
\label{Lem:face_subsurface_boundary_cycles}
Let $\Delta$ be a face subsurface of $G$, with $G$-boundary components given by cycles of $G$-boundary edges $C_1, \ldots, C_n$. Then each $C_j$ is vertex-simple, and the sets of vertices arising in the $C_j$ are pairwise disjoint.
\end{lem}
In particular, any $G$-boundary component is a vertex-simple cycle of $G$.
Note that a $G$-boundary component can be a cycle of length $1$, i.e. a loop in $G$.

\begin{proof}
Let $v$ be a vertex on some $C_j$. As a boundary point of the surface $\Delta$, a neighbourhood of $v$ in $\Delta$ is a half disc, bounded by $v$ and two ends of edges of $G$. These two ends of edges must be the ends of the adjacent edges to $v$ in $C_j$ (they may be the ends of a single edge which is a loop). Thus $v$ cannot lie in $C_j$ more than once, nor can it lie in any $C_k$ with $k \neq j$.
\end{proof}

Now let $(U, \Sigma, \mathscr{F}, G)$ be a framed multiverse, $G^\perp$ a dual of $G$, and $M \in \mathscr{M}_G$. Obviously the above definitions apply to $G$ and $\Sigma$. 
\begin{defn}[Alternating subsurface]
\label{Def:alternating_subsurface}
An \emph{alternating subsurface} of $G$ relative to $M$ is a face subsurface $\Delta$  of $G$ all of whose $G$-boundary components are alternating cycles relative to $M$.
\end{defn}

We now introduce a generalised notion of \emph{twisting}, in the spirit of \refdef{twisting}. Suppose $\Delta$ is an alternating surface of $G$ relative to $M$. We now consider removing from $M$ all edges which lie in $\partial \Delta$, and adding to $M$ the edges of $\partial \Delta$ which do not belong to $M$, to obtain a set of edges $M'$. In similar notation to \refsec{matchings} and \refeqn{matching_notation}, regarding $G$-boundary components $\partial_G \Delta$ as sets of edges,
\[
M' = \left( M \setminus \partial_G \Delta \right) \cup \left( \partial_G \Delta \setminus M \right) = M + \partial_G \Delta.
\]
\begin{defn}[Twisting]
\label{Def:twisting_on_alternating_subsurface}
Let $\Delta$ be an alternating subsurface relative to $M$.
The operation of replacing the matching $M$ with the matching $M' = M + \partial_G \Delta$ as above is called \emph{twisting} $M$ at $\partial \Delta$.
\end{defn}

As a multiverse surface $\Sigma$ is orientable, it determines an orientation of a face subsurface $\Delta$. This orientation in turn determines an orientation of its boundary $\partial \Delta$, which we call the \emph{boundary orientation} of $\partial \Delta$. We denote by $\partial_G \Delta \cap M$ the edges of $G$-boundary components of $\Delta$ that belong to $M$. Recall also that $G$ is bipartite, so every edge has a black and white vertex. We can now introduce a notion of positive and negative, in the spirit of \refdef{pos_neg_cycle}; we first deal with one boundary component at a time.
\begin{defn}[Positive, negative boundary component]
\label{Def:pos_neg_subsurface_bdy_component}
Let $\Delta$ be an alternating subsurface of $G$ relative to $M$, and let $C$ be a $G$-boundary component of $\Delta$.
\begin{enumerate}
\item 
$C$ is \emph{positive} relative to $M$ if every edge of $C \cap M$, oriented from black to white, agrees with the boundary orientation of $\partial \Delta$.
\item 
$C$ is \emph{negative} relative to $M$ if every edge of $C \cap M$, oriented from black to white, disagrees with the boundary orientation of $\partial \Delta$.
\end{enumerate}
\end{defn}
Equivalently, $C$ is positive (resp. negative) relative to $M$ if, when oriented as $\partial \Delta$, the edges in $M$ are oriented from black to white (resp. white to black).

\begin{defn}[Positive, negative subsurface]
\label{Def:pos_neg_subsurface}
Let $\Delta$ be alternating subsurface of $G$ relative to $M$. 
$\Delta$ is \emph{positive} (resp. \emph{negative}) relative to $M$ if each of its $G$-boundary components is positive (resp. negative) relative to $M$.
\end{defn}
Thus, $\Delta$ is positive (resp. negative) relative to $M$ if every edge of $\partial_G \Delta \cap M$, oriented from black to white, agrees (resp. disagrees) with the boundary orientation of $\partial \Delta$. Equivalently, $\Delta$ is positive (resp. negative) relative to $M$ if, when oriented as $\partial \Delta$, the edges of $M \cap \partial \Delta$ are oriented from black to white (resp. white to black).

When $U$ is a planar multiverse, a positive (resp. negative) cycle in the sense of \refdef{pos_neg_cycle} and \refsec{elementary_on_plane_surface} bounds a  positive (resp. negative) surface.

Note that if $M'$ is obtained from $M$ by twisting along $\Delta$, then $M$ is positive (resp. negative) if and only if $M'$ is negative (resp. positive).
Applying this idea, we now define the notion of \emph{surface twisting up and down}, generalising the notion of twisting up and down from \refdef{twisting_up_down}, but also requiring a condition on interior edges. Just as twisting up and down in \refdef{twisting_up_down} was performed  on elementary cycles (\refdef{elementary_cycle}), which form the outer boundary of non-outer faces, we also require the surfaces involved to avoid the outer boundary of our multiverse.
\begin{defn}[Twisting surface]
\label{Def:twisting_surface}
Let $c$ be the circulation of $M$. A \emph{twisting surface} relative to $M$ is a positive or negative subsurface $\Delta$ relative to $M$ such that 
\begin{enumerate}
\item 
$\Delta$ is disjoint from the outer boundary $\partial \Sigma_0$ of $\Sigma$, and 
\item 
each interior edge of $\Delta$ is either $c$-forced or $c$-forbidden.
\end{enumerate}
\end{defn}

Although the definition of twisting surface involves a significant amount of detail, they can be straightforwardly identified as follows.
\begin{lem}
\label{Lem:twisting_surface_G-boundary}
Two twisting surfaces are equal if and only if they have the same $G$-boundary components.
\end{lem}

\begin{proof}
Let $C = \{ C_1, \ldots, C_n\} $ be the set of $G$-boundary cycles of a twisting surface $\Delta$. Observe that $C$ forms a separating collection of curves on $\Sigma$, and $\Delta$ is one of the complementary components. Indeed, $\Delta$ is the unique complementary component of $C$ containing all the $C_j$ as boundary components and not containing the outer boundary $\partial \Sigma_0$: if there were two such components, then they would form the whole of $\Sigma$ yet not contain the outer boundary, a contradiction. Thus two twisting surfaces with $G$-boundary $C$ are equal. The converse is clear.
\end{proof}

\begin{defn}[Surface twisting]
\label{Def:big_twisting}
Let $\Delta$ be a twisting surface relative to $M$.  
The operation of twisting $M$ at $\partial \Delta$ to obtain a matching $M' = M + \partial_G \Delta$ is called \emph{surface twisting} along $\Delta$. It is called 
\begin{enumerate}
\item 
\emph{surface twisting down}, if $\Delta$ is positive relative to $M$ (hence negative relative to $M'$);
\item 
\emph{surface twisting up}, if $\Delta$ is negative relative to $M$ (hence positive relative to $M'$).
\end{enumerate}
\end{defn}

Note that a twisting surface $\Delta$ may simply be a single (non-outer) face homeomorphic to a disc, in which case surface twisting generalises twisting at a positive/negative cycle encircling a face of a connected bipartite plane graph.
However, when a twisting surface contains several faces, or has several boundary components, surface twisting imposes numerous further requirements, quite distinct from twisting up/down on a negative/positive cycle in a bipartite plane graph.

\subsection{Surface twisting and orientations on the dual}
\label{Sec:surface_twisting_orientations}

Again let $(U, \Sigma, \mathscr{F}, G)$ be a framed multiverse, $G^\perp$ a dual of $G$, and $M \in \mathscr{M}_G$. Let $\Delta$ be a face subsurface of $G$.

Each edge $e^\perp$ of $G^\perp$ dual to a $G$-boundary edge $e$ of $\Delta$ goes across $e$, from one side to the other. Thus $e^\perp$ crosses the boundary of $\Delta$, and is dual to an edge in $\partial_G \Delta$. We refer to such edges repeatedly, so give them a name.
\begin{defn}
The \emph{$\partial_G \Delta^\perp$ edges} are the edges of $G^\perp$ dual to $G$-boundary edges of $\Delta$.
\end{defn}
Since $\partial_G \Delta^\perp$ edges cross the boundary of $\Delta$, they can be oriented into or out of $\Delta$.

When $\Delta$ is  positive or negative relative to  $M$, the prescribed orientation of $M$ orients $\partial_G \Delta^\perp$ edges in a definite manner.
\begin{lem}[Orientations dual to positive/negative surfaces]
\label{Lem:dual_orientation_in_out_surface}
Let $R_M \in \mathscr{PR}_G$ be the prescribed orientation of $G^\perp$ corresponding to $M$. 
\begin{enumerate}
\item 
$\Delta$ is positive relative to $M$ if and only if $R_M$ orients each $\partial_G \Delta^\perp$ edge into  $\Delta$.
\item 
$\Delta$ is negative relative to $M$ if and only if $R_M$ orients each $\partial_G\Delta^\perp$ edge out of $\Delta$.
\end{enumerate}
\end{lem}

\begin{proof}
Let $C$ be a $G$-boundary component of $\Delta$, oriented as $\partial \Delta$. By \reflem{face_subsurface_boundary_cycles}, $C$ is a vertex-simple directed cycle. 
As $G$ is bipartite, we can denote the black and white vertices of $C$ by $w_j$ and $b_j$ for $j \in \Z/n\Z$ respectively, so that in order around $C$ they are $w_1, b_1, \ldots, w_n, b_n$. Denote each edge of $C$ by (white vertex, black vertex).

In the standard orientation $R_0$ of $G^\perp$ (\refdef{standard_orientation}), the edges of the basic cycle around each $w_j$ and $b_j$ (\refdef{basic_cycle}) are oriented counterclockwise and clockwise (\refdef{orientations_basic_cycles}) respectively. Thus in $R_0$ the edges of $G^\perp$ dual to edges of $C$ are directed alternately into and out of $\Delta$. Each edge $(w_j,b_j)^\perp$ of $G^\perp$ is directed into $\Delta$, and each edge $(w_j,b_{j-1})^\perp$ is directed out of $\Delta$. See \reffig{from_alpha_to_beta}.

\begin{figure}
	\centering
	\includegraphics[width=0.6\textwidth]{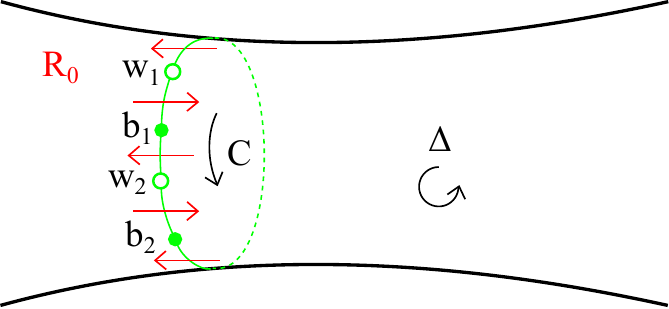}
	\caption{A $G$-boundary component $C$ of a face subsurface $\Delta$, with boundary orientation indicated, and standard orientation $R_0$ on dual edges of $G^\perp$. $G$ and $G^\perp$ are shown in green and red respectively.}\label{Fig:from_alpha_to_beta}
\end{figure}

The prescribed orientation $R_M$ differs from $R_0$ precisely on the edges of $G^\perp$ dual to edges of $M$. Thus $R_M$ points entirely into or out of $\Delta$ along $C$ precisely when $C$ is alternating relative to $M$. 

If $\Delta$ is positive (resp. negative) then, as noted after \refdef{pos_neg_subsurface}, the edges of $\partial_G \Delta \cap M$, oriented as $\partial \Delta$, are oriented from black to white (resp. white to black).
Thus $C$ is positive when the $(w_j, b_{j-1})$ belong to $M$, and $R_M$ orients each edge of $G^\perp$ dual to $C$ into $\Delta$. Similarly, $C$ is negative when the $(w_j, b_{j})$ belong to $M$, and $R_M$ orients each edge of $G^\perp$ dual to $C$ out of $\Delta$. See \reffig{K_is_maximal}.

\begin{figure}
	\centering
	\includegraphics[width=0.95\textwidth]{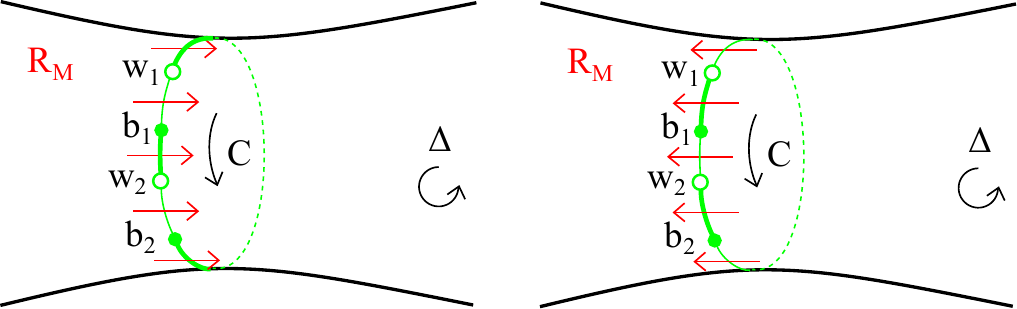}
	\caption{Left: $C$ positive, $R_M$ points into $\Delta$. Right: $C$ negative, $R_M$ points out of $\Delta$. Edges of $M$ are thickened.}\label{Fig:K_is_maximal}
\end{figure}

Thus, $\Delta$ is positive (resp. negative) if and only if each $G$-boundary cycle of $\Delta$ is positive (resp. negative), if and only if $R_M$ orients all 
$\partial_G \Delta^\perp$ edges into (resp. out of) $\Delta$.
\end{proof}

\begin{lemma}[Surface twisting preserves circulation]
\label{Lem:twisting_gives_same_circulation}
Let $\Delta$ be a positive or negative surface relative to $M$. Suppose $M' \in \mathscr{M}_G$ is obtained from $M$ by twisting $M$ at $\partial \Delta$, i.e. $M' = M + \partial_G \Delta$. Let $R,R'$  be the prescribed orientations of $G^\perp$ corresponding to $M, M'$ respectively. Then we have the following.
\begin{enumerate}
\item
$R$ and $R'$ differ precisely on the $\partial_G \Delta^\perp$ edges.
\item 
$R$ and $R'$ have the same circulation, i.e. $c_{R} = c_{R'}$. 
\end{enumerate}
\end{lemma}
(In the notation of \refdef{standard_orientation} here $R=R_M$ and $R'=R_{M'}$. Our notation here avoids repeated sub- and superscripts.) 

\begin{proof}
The edges of $M$ and $M'$ differ precisely 
on 
the $G$-boundary edges of $\Delta$. Thus the orientations $R,R'$ 
differ precisely on the $\partial_G \Delta^\perp$ edges. 
This gives (i).

For (ii), 
suppose $\Delta$ is positive relative to $M$, hence negative relative to $M'$; the argument is similar in the opposite case. Then, by \reflem{dual_orientation_in_out_surface}, for a $\partial_G \Delta^\perp$  edge $e^\perp$ of $G^\perp$,
$R$ orients $e^\perp$ into $\Delta$, and $R'$ orients $e^\perp$ out of $\Delta$.

Now consider a directed cycle $D$ in $G^\perp$, and its circulation $c_R (D), c_{R'}(D)$ with respect to $R,R'$ respectively. The orientations $R,R'$ differ precisely on the $\partial_G \Delta^\perp$ edges. 
So each non-$\partial_G \Delta^\perp$ edge of $D$ 
contributes equally to $c_R (D)$ and $c_{R'}(D)$.

The $\partial_G \Delta^\perp$ edges
are precisely those
which enter or exit $\Delta$. As $D$ is a cycle, the number of times $D$ enters $\Delta$ must be equal to the number of times $D$ exits $\Delta$.
In $R$, the edges into $\Delta$ are forward and the edges out of $\Delta$ are backward; in $R'$, the edges into $\Delta$ are backward and the edges out of $\Delta$ are forward.
The $\partial_G \Delta^\perp$ edges of $D$
thus contribute zero to $c_R (D)$, and zero to  $c_{R'}(D)$. 

Therefore $c_R (D) = c_{R'}(D)$, 
and $R$ and $R'$ have the same circulation. 
\end{proof}

\reflem{twisting_gives_same_circulation} immediately implies the following.
\begin{lem}
\label{Lem:orientations_related_by_surface_twisting}
Suppose $M' \in \mathscr{M}_G$ is obtained from $M$ by surface twisting along $\Delta$. Then the corresponding orientations $R,R' \in \mathscr{PR}_G$ have the same circulation, and differ precisely on $\partial_G \Delta^\perp$ edges. 
\end{lem}

\begin{proof}
Surface twisting is performed along a twisting surface, which is positive or negative.
\end{proof}

\subsection{Accessibility class pushing in the dual of the spine}
\label{Sec:pushing_in_dual_of_spine}

As usual, let $(U, \Sigma, \mathscr{F}, G)$ be a framed multiverse, $G^\perp$ a dual of $G$, and $M \in \mathscr{M}_G$.

Recall from \refdef{outer_vertex} that $G^\perp$ has a distinguished \emph{outer} vertex, dual to the outer face of $G$, adjacent to the outer boundary $\partial \Sigma_0$ of $\Sigma$.
\begin{defn}[Outer accessibility class]
\label{Def:outer_accessibility_class}
Let $R$ be an orientation of $G^\perp$. The \emph{outer accessibility class} of $R$ is the accessibility class of $R$ containing the outer vertex of $G^\perp$.
\end{defn}

\begin{lemma}\label{Lem:big_twisting_and_pushing}
Let $M,M' \in \mathscr{M}_G$ correspond to $R,R' \in \mathscr{PR}_G$. The following are equivalent.
\begin{enumerate}
\item
There is a surface twisting up (resp. down) $T$  on a negative (resp. positive) twisting surface $\Delta$ relative to $M$, taking $M$ to $M'$.
\item 
There is a pushing up (resp. down) $P$ on a minimal (resp. maximal) non-outer accessibility class $K$ of $R$, taking $R$ to $R'$.
\end{enumerate}
Moreover, 
$K$ consists of the vertices of $G^\perp$ dual to the faces of $\Delta$.
\end{lemma}
Similar to \refsec{twisting_transpositions}, we denote by $\STw_M$ the finite set of all surface twisting up operations that can be done on $M$, and by $\Pu_R$ the finite set of all pushing up operations on non-outer accessibility classes of $R$. Then \reflem{big_twisting_and_pushing} provides a bijection
\[
\STw_M \stackrel{\cong}{\To} \Pu_R
\]
which takes $T$ to $P$.

\begin{proof}
We prove that surface twisting up operations $T$ correspond to pushing up operations $P$; the arguments for twisting down and pushing down are similar.

Suppose a surface twisting up operation $T$ on the negative twisting surface $\Delta$ takes $M \mapsto M'$, and let $R,R'$ be the corresponding prescribed orientations of $G^\perp$. By \reflem{orientations_related_by_surface_twisting}, $R$ and $R'$ differ precisely on $\partial_G \Delta^\perp$ edges of $G^\perp$, and have the same circulation, which we denote by $c$. By \reflem{accessibility_same_circulation} then $R$ and $R'$ have the same accessibility classes. 

We construct the desired pushing up operation $P$ taking $R \mapsto R'$ in several steps as follows.
\begin{enumerate}
\item 
\emph{Construction of the set $K$.}
Let $\Delta_1,\dotsc,\Delta_m$ be the faces of $\Delta$.
Let $V_i$ be the vertex of $G^\perp$ dual to $\Delta_i$, and let $K = \{V_1, \ldots, V_m\}$. Let $K^c$ be the complement of $K$ in the vertex set of $G^\perp$. 
\item
\emph{$K$ forms an accessibility class of $c$.}
Let $V_a, V_b$ be two vertices in $K$ joined by an edge $e^\perp$ whose dual in $G$ is denoted $e$. Since $e$ is adjacent to $\Delta_a$ on one side and $\Delta_b$ on the other, it is an interior edge of $\Delta$ (\refdef{face_subsurface_bdy_edges}). Being an interior edge of a twisting surface (\refdef{twisting_surface}), $e$ is $c$-forced or $c$-forbidden.
By \refprop{all_nothing} then $e^\perp$ joins two vertices in the same accessibility class of $c$.
Hence, $V_a$ and $V_b$ belong to the same accessibility class of $c$.
As a face subsurface, $\Delta$ is connected (\refdef{face_subsurface}), hence the vertices of $K$ are connected by edges in $G^\perp$. Thus all vertices of $K$ lie in the same accessibility class.

Now consider a $\partial_G \Delta$ edge $e$ of $G$, and its dual $\partial_G \Delta^\perp$ edge $e^\perp$ of $G^\perp$. By \reflem{orientations_related_by_surface_twisting}, $R$ and $R'$ differ on $e^\perp$, so $e^\perp$ is not $c$-forced or $c$-forbidden. Hence by \refprop{all_nothing} again, the endpoints of $e^\perp$ belong to distinct accessibility classes of $c$.

Thus, the vertices of $G^\perp$ in $\Delta$ (namely $K$) lie in the same accessibility class, but if we proceed from $K$ out of $\Delta$, across a $\partial_G \Delta^\perp$ edge, then we arrive at a vertex in a different accessibility class. Thus the vertices of $K$ form an accessibility class.

\item
\emph{All edges between $K$ and $K^c$ are oriented from $K$ to $K^c$ in $R$.} Indeed, $K$ consists precisely of the vertices of $G^\perp$ in $\Delta$, and $K^c$ consists of the vertices of $G^\perp$ outside $\Delta$. So the edges between $K$ and $K^c$ are precisely the edges of $G^\perp$ crossing the boundary of $\Delta$, i.e. the $\partial_G \Delta^\perp$ edges. By \reflem{dual_orientation_in_out_surface}, $R$ orients all $\partial_G \Delta^\perp$ edges out of $\Delta$, i.e. from $K$ to $K^c$, as claimed. By a similar argument, all edges between $K$ and $K^c$ are oriented from $K^c$ to $K$ in $R'$.
\end{enumerate}
Thus, $K$ is a minimal accessibility class in $R$ and a maximal accessibility class in $R'$, and the two orientations differ precisely on the edges between $K$ and $K^c$. Hence, $R'$ is obtained from $R$ by pushing up on $K$. As a twisting surface (\refdef{twisting_surface}), $\Delta$ is disjoint from the outer boundary of $\Sigma$, and so $K$ does not contain the outer vertex of $G^\perp$, 
hence is a non-outer accessibility class. 
So the pushing up $P \in \Pu_R$ on the non-outer accessibility class $K$ takes $R$ to $R'$ and is as claimed.

Thus we have a map $r \colon \STw_M\to\Pu_R$ taking $T \mapsto P$, and it remains to construct an inverse.

So suppose a pushing up $P$ on a minimal non-outer accessibility class $K$ takes $R \mapsto R'$. 
Then $K$ is a maximal accessibility class for $R'$. Moreover $R,R'$ have the same viable circulation $c$, 
so by \reflem{viable_orientation_prescribed}, both $R,R'$ are prescribed orientations of matchings $M,M'$.
The orientations $R,R'$ differ precisely on the edges of $G^\perp$ between $K$ and $K^c$,  hence $M,M'$ differ precisely on the corresponding dual edges of $G$.

Let $K = \{ V_1,\dotsc,V_m \}$,  
and let $\Delta_i$ be the face of $G$ dual to $V_i$. 
The edges of $G^\perp$ between $K$ and $K^c$ are thus the edges between a vertex $V_j$ and a vertex not among the $V_j$. The corresponding edges of $G$ are those which have a face $\Delta_j$ on one side but not the other; these are precisely the edges on which $M,M'$ differ. 

Let $\Delta = \overline{\bigcup_{j=1}^m \Delta_j}$.
Then $M,M'$ differ precisely on those edges of $G$ which have $\Delta$ on one side but not the other. We temporarily call these the \emph{boundary edges} of $\Delta$.

We claim $\Delta$ is a negative twisting surface for $M$. We prove this in several steps.
\begin{enumerate}
\item
\emph{$\Delta$ is a subsurface of $\Sigma$.} As the closure of a union of faces, $\Delta$ fails to be a surface precisely if there is a vertex $v$ of $G$ whose neighbourhood in $\Delta$ is not a disc or half disc. See \reffig{vertex_surface_problem}. The number of ends of boundary edges at $v$ is even, and a neighbourhood of $v$ in $\Delta$ fails to be a disc or half disc when this number is $4$ or more. So suppose for contradiction that we have $4$ ends of boundary edges $e_1, e_2, e_3, e_4$ incident to a $v$ of $\Delta$. (These four ends of edges $e_j$ at $v$ are distinct, but a priori two of these ends could be the ends of a single edge, forming a loop.) As each $e_j$ is a boundary edge, it is in precisely one of $M$ or $M'$. Hence one of $M$ or $M'$ must contain at least $2$ of the $e_j$. But then we have a matching where more than one endpoint of an edge is chosen at $v$, contradicting \refdef{matching}. 
So $\Delta$ is indeed a subsurface. 

\begin{figure}
\begin{center}
\includegraphics[width=0.5\textwidth]{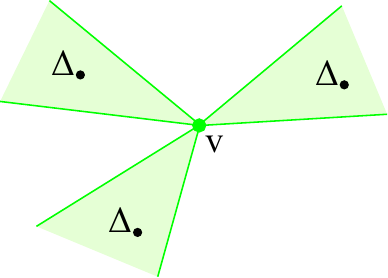}
\end{center}
    \caption{Vertex of $\Delta$ violating surface condition. Faces $\Delta_\bullet$ are shaded.} 
    \label{Fig:vertex_surface_problem}
\end{figure}

\item 
\emph{$\Delta$ is connected and hence a face subsurface of $G$} (\refdef{face_subsurface}). As $K$ is an accessibility class, the vertices $V_j$ of $G^\perp$ are all joined by edges of $G^\perp$, hence the faces $\Delta_j$ of $\Delta$ are all joined by common edges along their boundaries, and $\Delta$ is connected.

As $\Delta$ is a face subsurface of $G$, the boundary edges discussed above are in fact its $G$-boundary edges. The matchings $M,M'$ differ precisely on $\partial_G \Delta$.
\item 
\emph{$\Delta$ is an alternating subsurface} of $G$ relative to $M$ (\refdef{alternating_subsurface}) and in fact a \emph{negative subsurface} (\refdef{pos_neg_subsurface}).
As $K$ is a minimal accessibility class relative to $R$, each edge $e^\perp$ of $G^\perp$ between $K$ and $K^c$ points out of $K$ in $R$.
The edges of $G^\perp$ between $K$ and $K^c$ are precisely 
the $\partial_G \Delta^\perp$ edges. As $R$ orients $\partial_G \Delta^\perp$ edges out of $\Delta$, by \reflem{dual_orientation_in_out_surface} $\Delta$ is negative relative to $M$ as claimed.

\item
\emph{Each interior edge $e$ of $\Delta$ is $c$-forced or $c$-forbidden.} Such an $e$ is an edge of $G$ with a face $\Delta_j$ on either side, and hence its dual edge $e^\perp$ in $G^\perp$ joins two vertices among the $V_j$. Thus $e^\perp$ joins two vertices of $K$, which is an accessibility class. By \refprop{all_nothing} then $e$ is $c$-forced or $c$-forbidden.
\item 
\emph{$\Delta$ is disjoint from the outer boundary} $\partial \Sigma_0$ of $\Sigma$. As $K$ is a non-outer accessibility class, it does not include the outer vertex of $G^\perp$. Hence the outer face of $G$ does not appear as a face $\Delta_i$ of $\Delta$, so $\Delta$ is disjoint from $\partial \Sigma_0$.

This concludes the proof of the claim, and $\Delta$ is a negative twisting surface.
\end{enumerate}

Therefore, starting from $P$ we have found a surface twisting up operation $T$ on $M$ at $\Delta$. Since $M$ and $M'$ differ precisely on $\partial_G \Delta$, the operation results in $M'$ as desired. Moreover, $K$ consists of the vertices of $G^\perp$ dual to the faces of $\Delta$ as claimed. 
This map $P \mapsto T$ recovers the twisting down operation $T$ mapped to $P$ under $r$, and we have constructed the required inverse to $r$.
\end{proof}

\subsection{Surface transpositions}
\label{Sec:surface_transpositions}

We now define a version of transpositions for general multiverses, which we call \emph{surface transpositions}. Unlike the plane transpositions of \refdef{n_transposition}, these need not be along a single curve. Rather, a surface transpositions consists of a collection of contour transpositions (\refdef{partial_transposition}) performed along a collection of transposition contours (\refdef{transposition_contour}) which bound a surface with certain properties, which we define below.

Throughout this section, let $(U, \Sigma, \mathscr{F})$ be a multiverse, $S \in \mathscr{S}_U$ a state, and $c$ its circulation.

Recall a transposition contour $\gamma$ passes through various vertices and faces of $U$, denoted $v_j, F_j$ in \refdef{transposition_contour}. With multiple transposition contours, we require them to be disjoint as follows.
\begin{defn}
\label{Def:independent_contours}
A set of transposition contours $\gamma_1, \ldots, \gamma_n$ for $S$ is \emph{independent} if their vertices and faces are all disjoint.
\end{defn}
By the finiteness of a multiverse graph, an independent set of transposition contours is finite.

Let $\{\gamma_1, \ldots, \gamma_n\}$ be an independent set of contours for $S$, and denote by $Z_j$ the operation of contour transposition along $\gamma_j$. Each $Z_j$ changes distinct markers of the state $S$, so the operations $Z_j$ can be performed on $S$ in any order 
to obtain the same resulting state, which we denote $S'$.

We will perform a surface transposition along a surface $\Psi \subset \Sigma$, which has an independent set of transposition contours $\gamma_j$ for $S$
as its boundary. Like previous notions of transposition, surface transposition
must also avoid the outer boundary. Since each transposition contour $\gamma_j$ is a boundary component of  surface $\Psi$, the interior of $\Psi$ lies to one side of $\gamma_j$. At a vertex $v$ of a $\gamma_j$, the corners which lie entirely in the interior of $\Psi$ are called the \emph{interior corners} of $v$. At a vertex $v$ of $U$ in the interior of $\Psi$, all corners are interior corners. The \emph{interior corners} of $\Psi$ are the interior corners at all vertices of $U$ in $\Psi$.

At a vertex $v$ of a contour $\gamma_j$, let $\alpha, \alpha'$ be the corners with markers in the states $S,S'$ respectively. The contour transposition $Z_j$ replaces each $\alpha$ with $\alpha'$. There is a direction, clockwise or counterclockwise, such that the rotation  from $\alpha$ to $\alpha'$ about $v$ in this direction passes through the interior of $\Psi$. By construction, this direction is the same at all vertices of a single contour $\gamma_j$. See \reffig{clockwise_transposition_contour}. We refer to the contour $\gamma_j$ in $\partial \Psi$ as \emph{clockwise} or \emph{counterclockwise} accordingly. 

\begin{figure}
\begin{center}
\includegraphics[width=0.55\textwidth]{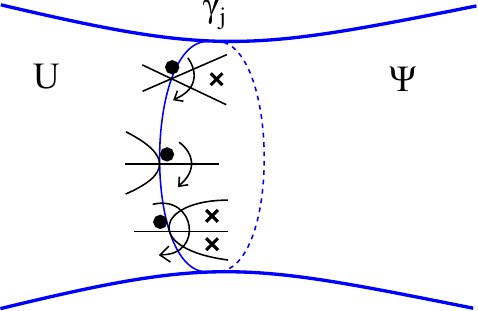}
\end{center}
    \caption{A clockwise contour transposition for a state $S$, at a contour $\gamma_j$ on the boundary of a subsurface $\Psi \subset \Sigma$. Dots are markers for $S$. Interior corners at vertices are noted with crosses.}
    \label{Fig:clockwise_transposition_contour}
\end{figure}

We require all $Z_j$ to be clockwise or counterclockwise; this requirement is analogous to a positive or negative surface (\refdef{pos_neg_subsurface}) having each $G$-boundary component positive or negative. 
\begin{defn}[Transposition surface]
\label{Def:transposition_surface}
A \emph{transposition surface} for $S$ is an compact subsurface $\Psi$ of $\Sigma$ such that the following hold: 
\begin{enumerate}
\item
The boundary of $\Psi$ consists of a (possibly empty) set of non-outer boundary components of $\Sigma$, and a nonempty independent set $\Gamma = \{\gamma_1, \ldots, \gamma_n\}$ of transposition contours of $S$.
\item 
Each interior corner of $\Psi$
is $c$-forced or $c$-forbidden.
\item 
All contours in $\Gamma$ are clockwise, or all contours in $\Gamma$ are counterclockwise.
\end{enumerate}
We say $\Psi$ is \emph{clockwise} or \emph{counterclockwise} accordingly with the contours of $\Gamma$.
\end{defn}
We call a boundary component of $\Psi$ a \emph{$\Sigma$-boundary component} if it is a boundary component of $\Sigma$, and a \emph{boundary contour} if it is a contour. There may be no $\Sigma$-boundary components.
As $\Sigma$ is compact, $\Psi$ contains finitely many boundary components. 

Condition (ii) here may be compared to the condition requiring forbidden corners on plane transpositions. However, the conditions are quite different. \refdef{n_transposition} only requires that that interior corners at vertices \emph{on} the contour are forbidden; condition (ii) here applies throughout the interior of $\Psi$,  refers only to the circulation $c$, and allows for $c$-forced as well as $c$-forbidden corners.

Just as transposition contours can be framed (\refdef{framed_contour}), so too can transposition surfaces.
\begin{defn}[Framed transposition surface]
\label{Def:framed_transposition_surface}
Let $(U, \Sigma, \mathscr{F}, G)$ be a framed multiverse. A transposition surface is \emph{framed} if all its boundary contours are framed transposition contours.    
\end{defn}

\begin{defn}[Surface transposition] \label{Def:transposition_on_positve_genus}
Let $\Psi$ be a framed transposition surface for a state $S$, with boundary contours $\Gamma = \{\gamma_1, \ldots, \gamma_n\}$.

\emph{Surface transposition} of $S$ along $\Psi$ is contour transposition of $S$ along all contours of $\Gamma$, yielding the state $S'$.
Accordingly as $\Psi$ is clockwise or counterclockwise, we say the surface transposition is \emph{clockwise} or \emph{counterclockwise}.
\end{defn}
Note that, just as a plane transposition must be done along a framed transposition contour, a surface transposition must also be done along framed transposition contours.

Plane transpositions are quite different from surface transpositions. As mentioned in the introduction, \reffig{Hasse_diagram_example1_thm1} and \reffig{Hasse_diagram_example1_thm2} show the distinct operations on the same multiverse.
Moreover, a Kauffman transposition (\refdef{transposition}) on a Kauffman universe may fail to be a surface transposition. For instance, \reffig{Hasse_diagram_example2_thm1} shows a multiverse and its Kauffman transpositions, which differ from the surface transpositions shown in \reffig{Hasse_diagram_example2_thm2}.

\subsection{Twisting surfaces and transposition surfaces}

We saw in \refsec{alternating_cycles_contours}, specifically \reflem{contours_alternating_cycles_equiv}, an equivalence between transposition contours and (vertex-simple) alternating cycles. We now show a similar equivalence between transposition surfaces and twisting surfaces.

As usual, let $(U, \Sigma, \mathscr{F}, G)$ be a framed multiverse, $S \in \mathscr{S}_U$ a state, $M \in \mathscr{M}_G$ the corresponding matching, and $c$ their circulation function.

The following lemma is the surface analogue of \reflem{alternating_cycles_contours}.
\begin{lem}
\label{Lem:twisting_surface_as_transposition_surface}
A positive (resp. negative) twisting surface $\Delta$ of $M$, regarded as an subsurface of $\Sigma$, is a clockwise (resp. counterclockwise) framed transposition surface for $S$.
\end{lem}

\begin{proof}
We prove the result for positive $\Delta$; the negative case is similar. We show the positive twisting surface $\Delta$ verifies 
conditions (i)--(iii) of \refdef{transposition_surface}. 

The boundary of $\Delta$ consists of non-outer $\Sigma$-boundary components, together with $G$-boundary components, 
which are vertex-simple alternating cycles of $M$, say $C_1, \ldots, C_n$, containing pairwise disjoint vertices (\reflem{face_subsurface_boundary_cycles}). 
By \reflem{alternating_cycles_contours} each $C_i$, regarded as a simple closed curve $\gamma_i$, is a framed transposition contour. As the $C_i$ have disjoint vertices, the $\gamma_i$ 
are independent,
satisfying (i).

Interior corners $\alpha$ of $\Psi$ 
correspond to interior edges $e$ of $\Delta$. 
As $\Delta$ is a twisting surface for $M$, all interior edges of $\Delta$ are $c$-forced or $c$-forbidden, hence
by \reflem{c-forced_forbidden_corners_edges}, all interior corners $\alpha$ of $\Psi$ are $c$-forced or $c$-forbidden,
satisfying (ii).

\begin{figure}
\begin{center}
\includegraphics[width=0.55\textwidth]{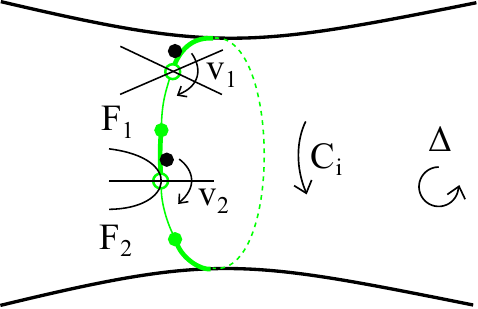}
\end{center}
    \caption{A boundary component $C_i$ of a positive twisting surface (green), regarded as a transposition contour. Edges of $M$ are thickened. The multiverse $U$ and state markers are shown in black.}
    \label{Fig:positive_twisting_surface_boundary}
\end{figure}

As $\Delta$ is positive, each $C_i$ is positive, so when oriented as $\partial \Delta$, the edges in $M$ are oriented from black to white (see \refdef{pos_neg_subsurface_bdy_component} and subsequent comment). Label the vertices of $C_i$, directed as $\partial \Delta$, in cyclic order as $v_1, F_1, \ldots, v_n, F_n$ where the $v_j, F_j$ are over $j \in \Z/n\Z$, the $v_j$ are white (hence vertices of $U$) and the $F_j$ are black (hence correspond to unstarred faces of $U$, which we also denote $F_j$), as in \reffig{positive_twisting_surface_boundary}; see also \reffig{K_is_maximal} (left). Then the edges in $M$ are precisely the edges oriented from black to white, i.e. $F_j$ to $v_{j+1}$. The edge in $M$ from $F_j$ to $v_{j+1}$ corresponds to a corner $\beta_{j+1}$ in the face $F_j$ at the vertex $v_{j+1}$ of $U$, where a marker appears in the state $S$. Contour transposition along $\gamma_j$ replaces the marker $\alpha_{j+1}$ at $v_{j+1}$ with the marker $\alpha'_{j+1}$ corresponding to the edge from $v_{j+1}$ to $F_{j+1}$. This is a clockwise transposition around $v_{j+1}$ through the interior of $\Psi$, as shown in figure \reffig{positive_twisting_surface_boundary}. Thus each contour $\gamma_i$ is clockwise, satisfying (iii).

Hence $\Delta$ is a clockwise transposition surface for $S$. As each $C_i$ is a framed contour, $\Delta$ is framed.
\end{proof}

For a converse and equivalence, we recall that the straightening (\refdef{straightening}) of a framed transposition contour $\gamma$ is a vertex-simple alternating cycle (\reflem{straightening_is_alternating_cycle}) isotopic to $\gamma$. Hence we have a surface analogue is as follows.
\begin{defn}
Let $\Psi$ a framed transposition surface with boundary contours $\Gamma$. The \emph{straightening} of $\Psi$ is the face subsurface of $G$ with the same $\Sigma$-boundary components as $\Psi$ and whose $G$-boundary components are straightenings of the contours $\Gamma$.
\end{defn}

\begin{lem}
\label{Lem:straightening_transposition_surface}
Let $\Psi$ a clockwise (resp. counterclockwise) framed transposition surface for $S$ with boundary contours $\Gamma$.
Then the straightening $\Delta$ of $G$ is a 
positive (resp. negative) twisting surface relative to $M$.
\end{lem}

\begin{proof}
We prove the result for clockwise $\Psi$; the counterclockwise case is similar.

First, note that as $\Psi$ is a transposition surface, the contours $\Gamma = \{\gamma_1, \ldots, \gamma_n\}$ are independent, so their straightenings $C = \{ C_1, \ldots, C_n\}$ are disjoint vertex-simple alternating cycles. Hence the isotopies taking each $\gamma_j$ to $C_j$ are disjoint, so $\Delta$ is actually a subsurface of $\Sigma$. As its boundary consists of components of $\partial \Sigma$ and edges and vertices of $G$, it is a face subsurface of $G$ (\refdef{face_subsurface}). As each $C_i$ is an alternating cycle relative to $M$, $\Delta$ is an alternating subsurface relative to $M$ (\refdef{alternating_subsurface}).

Now we show $\Delta$ is a positive twisting surface relative to $M$ verifying (i) and (ii) of \refdef{twisting_surface}. As a transposition surface, $\Psi$ is disjoint from the outer boundary of $\Sigma$, hence so is $\Delta$, verifying (i).

Consider an interior edge $e$ of $\Delta$. This corresponds to an interior corner $\alpha$ of $U$ in $\Psi$ in an unstarred face. As $\Psi$ is a transposition surface, $\alpha$ is $c$-forced or $c$-forbidden. 
Thus by \reflem{c-forced_forbidden_corners_edges}, 
$e$ is $c$-forced or $c$-forbidden, verifying (ii).

Consider a cycle $C_j$, the straightening of $\gamma_j$. As $\Psi$ is a clockwise transposition surface, contour transposition along $\gamma_j$ rotates each marker at each vertex of $\gamma_j$ clockwise through the interior of $\Psi$, as in \reffig{clockwise_transposition_contour}. Thus, when oriented as $\partial \Delta$, the edges of $M$ in $C_j$ are oriented from black to white, as in \reffig{positive_twisting_surface_boundary}. Thus $C_j$ is positive, and $\Psi$ is a positive twisting surface relative to $M$.
\end{proof}

For an equivalence statement, we introduce the following straightforward notion of isotopy of transposition surfaces, following definitions of isotopy for embedded graphs (\refdef{isotopy_embedded_graphs}), transposition contours (\refdef{isotopy_contour}), and spines (\refdef{isotopy_spine}).
\begin{defn}
Two transposition surfaces $\Psi_0, \Psi_1$ for a state $S$ are \emph{isotopic} if there is a continuous family of transposition surfaces $\Psi_t$ for $t \in [0,1]$ from $\Psi_0$ to $\Psi_1$.

Two surface transpositions are \emph{isotopic} is their transposition surfaces are isotopic.
\end{defn}
In the continuous family $\Psi_t$ of transposition surfaces, the $\Sigma$-boundary components must remain constant, and the restriction to a boundary contour $\gamma_0$ of $\Psi_0$ yields an isotopy of transposition contours $\gamma_t$ from $\gamma_0$ to a corresponding boundary contour $\gamma_1$ of $\Psi_1$. Indeed, an isotopy of transposition surfaces $\Psi_t$ is determined by its restriction to its boundary contours, and we have the following, corresponding to \reflem{twisting_surface_G-boundary} for twisting surfaces.
\begin{lem}
\label{Lem:transpositions_equal_boundary}
Two transposition surfaces are isotopic if and only if their boundary contours are isotopic.
\qed
\end{lem}

\begin{lem}
Straightening provides a bijection between isotopy classes of clockwise (resp. counterclockwise) framed transposition surfaces for $S$, and positive (resp. negative) twisting surfaces of $G$ relative to $M$.
\end{lem}

\begin{proof}
We prove the result for clockwise/positive surfaces; the other case is similar. 
Let $\Psi, \Psi'$ be  isotopic 
clockwise framed transposition surfaces for $S$, with framed boundary contours $\Gamma = \{\gamma_1, \ldots, \gamma_n\}$ and $\Gamma' = \{\gamma'_1, \ldots, \gamma'_n\}$, where each $\gamma_i$ is isotopic
to $\gamma'_i$,  
and each $\gamma_i$ and $\gamma'_i$ is clockwise. 

By \reflem{contours_alternating_cycles_equiv}, both $\gamma_i, \gamma'_i$ straighten to the same vertex-alternating cycle $C_i$ of $G$ relative to $M$. Thus the straightenings $\Delta, \Delta'$ of $\Psi, \Psi'$ on $G$ respectively are twisting surfaces of $G$ relative to $M$ with the same $G$-boundary cycles, and hence by \reflem{twisting_surface_G-boundary} $\Delta = \Delta'$. By \reflem{straightening_transposition_surface}, $\Delta = \Delta'$ is positive. 

This gives the desired map from isotopy classes of framed transposition surfaces to twisting surfaces. Regarding twisting surfaces as transposition surfaces then provides the desired inverse.
\end{proof}

As in \refsec{spines_framings}  and \refsec{alternating_cycles_contours}, when the faces involved are homeomorphic to discs, these considerations simplify. A multiverse $(U, \Sigma, \mathscr{F})$ has a unique framing $G$ up to isotopy. Then all transposition contours are framed, and all transposition surfaces are framed.

\subsection{Equivalence of surface twisting and surface transpositions}

We now use the equivalence of transposition surfaces and twisting surfaces of the previous section, to obtain an equivalence between surface transpositions and surface twisting.

\begin{lemma}\label{Lem:transposition_and_big_twisting}
Let $(U, \Sigma, \mathscr{F}, G)$ be a framed multiverse, 
an let $S,S' \in \mathscr{S}_U$ correspond to $M,M' \in \mathscr{M}_G$.
The following are equivalent.
\begin{enumerate}
\item
There is a counterclockwise (resp. clockwise) surface transposition $Z$ on a counterclockwise (resp. clockwise) framed transposition surface $\Psi$ for $S$, taking $S$ to $S'$.
\item 
There is a surface twisting up (resp. down) $T$ on a negative (resp. positive) twisting surface $\Delta$ relative to $M$, taking $M$ to $M'$.
\end{enumerate}
Moreover, $\Psi$ is isotopic to $\Delta$, and $\Delta$ is the straightening of $\Psi$.
\end{lemma}

As in \refsec{pushing_in_dual_of_spine} and \reflem{big_twisting_and_pushing} we denote by $\STw_M$ the finite set of surface twisting up operations that can be done on $M$. Similarly, we denote by $\Tr_S$ the finite set of isotopy classes of counterclockwise (resp. clockwise) surface transpositions that can be done on $S$. Then \reflem{transposition_and_big_twisting} provides a bijection
\[
\Tr_S \stackrel{\cong}{\To} \STw_M
\]
which takes (the isotopy class of) $Z$ to $T$.

\begin{proof}
We consider the counterlockwise/negative case; the other case is similar. 
Let $\Psi$ be the counterclockwise transposition surface of $Z$, and let $\Delta$ be its straightening. By \reflem{straightening_transposition_surface}, $\Delta$ is a negative twisting surface of $G$ relative to $M$. Letting the boundary contours of $\Psi$ be $\gamma_1, \ldots, \gamma_n$, in $\Delta$ they are straightened into disjoint vertex-simple negative (resp. positive) alternating boundary cycles $C_1, \ldots, C_n$ of $G$ relative to $M$. Surface transposition on $\Psi$ rotates each state marker $\alpha_i$ of $S$ at each corner $v_i$ of $\gamma_j$ counterclockwise through the interior of $\Psi$ to a corner $\alpha'_i$ of $v_i$, to obtain $S'$. Surface twisting up on $\Delta$ twists the matching $M$ on $\partial_G \Delta$ to obtain $M' = M + \partial_G \Delta$, which then corresponds to $S'$. See  \reffig{black_to_white_on_Z}.

\begin{figure}
	\centering
	\includegraphics[width=0.35\textwidth]{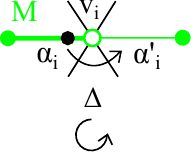}
	\caption{Rotating markers and matchings. The edges of $M$ are thickened.}
    \label{Fig:black_to_white_on_Z}
\end{figure}

Moreover, $\Delta$ is unique: any twisting surface $\Delta'$ for $M$
on which twisting takes $M \mapsto M'$ 
must have the same vertices and edges in its $G$-boundary as $\Delta$, and as the  $G$-boundary cycles of $\Delta'$ are disjoint vertex-simple cycles (\reflem{face_subsurface_boundary_cycles}), they must coincide with those of $\Delta$, so 
by \reflem{twisting_surface_G-boundary}, $\Delta=\Delta'$.

Conversely, if $\Delta$ is the negative twisting surface of $T$, then by \reflem{twisting_surface_as_transposition_surface} it may be regarded as a counterclockwise framed transposition surface  $\Psi$ with the properties claimed, and $\Delta$ is the straightening of $\Psi$ as required. Moreover, $\Psi$ is unique up to isotopy: any transposition surface $\Psi'$ for $S$ on which pushing up takes $S \mapsto S'$
must involve state markers in corners corresponding to the $G$-boundary edges of $\Delta$ in $M$. As the $G$-boundary edges form disjoint vertex-simple cycles, the framed boundary contours of $\Psi'$ must involve the same vertices, faces and corners of $\Psi$. Thus the framed boundary contours of $\Psi'$ are corner-equivalent to those of $\Psi$, so by \reflem{framed_contour_equivalent} they are isotopic. Hence by \reflem{transpositions_equal_boundary} $\Psi'$ is isotopic to $\Psi$, and to $\Delta$.
\end{proof}

\subsection{Clock theorem in arbitrary genus}

Combining \refeqn{triple_bijection}, \reflem{big_twisting_and_pushing} and \reflem{transposition_and_big_twisting} then immediately gives the following.
\begin{prop}\label{Prop:transposition_and_pushing}
Let $(U, \Sigma, \mathscr{F}, G)$ be a framed multiverse, 
and $G^\perp$ a dual of $G$. Then we have bijections
\[
\mathscr{S}_U \cong \mathscr{M}_G \cong \mathscr{PR}_G.
\]
For any $S \in \mathscr{S}$ corresponding to $M \in \mathscr{M}_G$ and $R \in \mathscr{PR}_G$, we have bijections
\[
\Tr_S \cong \STw_M \cong \Pu_R.
\]
\qed
\end{prop}

We can now define relations on these sets as follows.
\begin{defn}
\label{Def:general_clock_thm_relations}
Let $(U, \Sigma, \mathscr{F}, G)$ be a framed multiverse and $G^\perp$ a dual of $G$.
\begin{enumerate}
\item 
Define a relation $\leqslant$ on $\mathscr{S}_U$ by $S \leqslant S'$ if there exists a sequence of states $S=S_0, \ldots, S_m = S'$ of $\mathscr{S}_U$, for some $m \geq 0$, such that each $S_{j+1}$ is obtained from $S_j$ by a counterclockwise surface transposition.
\item 
Define a relation $\leqslant$ on $\mathscr{M}_G$ by $M \leqslant M'$ if there exists a sequence $M=M_0, \ldots, M_m = M'$ in $\mathscr{M}_G$, for some $m \geq 0$, such that each $M_{j+1}$ is obtained from $M_{j}$ by surface twisting up.
\item 
Define a relation $\leqslant$ on $\mathscr{PR}_G$ by $R \leq R'$ if there exists a sequence $R=R_0, \ldots, R_m = R'$ in $\mathscr{PR}_G$, for some $m \geq 0$, such that each $R_{j+1}$ is obtained from $R_j$ by pushing up on a non-outer accessibility class.
\end{enumerate}
\end{defn}
In other words, $S \leqslant S'$ if $S'$ is obtained from $S$ by a sequence of counterclockwise surface transpositions; $M \leqslant M'$ if $M'$ is obtained from $M$ by a sequence of surface twisting up operations; and $R \leqslant R'$ if $R'$ is obtained from $R$ by a sequnce of pushing up operations on non-outer accessibility classes.

As discussed in \refsec{surface_transpositions}, surface transpositions do not generalise Kauffman transpositions, and are quite different from surface transpositions. Hence, although it is in a similar vein,  the relation on $\mathscr{S}_U$ here is not strictly a generalisation of \refdef{leq_on_states} or \refdef{leq_on_plane_multiverse_state}.

Similarly, as discussed in \refsec{twisting_surfaces}, while surface twisting does generalise twisting on single faces of connected finite plane bipartite graphs, it is quite different in general. 
Thus the relation on $\mathscr{M}_G$ here is only a generalisation of \refdef{relation_on_matchings1} in a limited sense.

However, the relation on $\mathscr{PR}_G$ is essentially the same as that of \refdef{relation_on_orientations}, where the non-pushable accessibility class it the outer accessibility class. \refdef{relation_on_orientations} takes orientations one circulation at a time, but pushing up preserves circulation, and by \reflem{viable_orientation_prescribed}, the prescribed orientations are precisely the orientations with viable circulation.

\refprop{transposition_and_pushing} provides an isomorphism between the relations $\leqslant$ on $\mathscr{S}_U$, $\mathscr{M}_G$ and $\mathscr{PR}_G$. For a viable circulation function $c$ on $G^\perp$, by \reflem{orientations_related_by_surface_twisting} and \refdef{circulation_of_state_matching}, these isomorphisms restrict to isomorphisms between the relations on $\mathscr{S}_U^c$, $\mathscr{M}_G^c$ and $\mathscr{PR}_G^c = \mathscr{R}_G^c$.

Propp's theorem on orientations of graphs \refthm{Propp_pushing0} can now be applied to prove our general clock theorem.

\begin{theorem}[Clock theorem in arbitary genus] \label{Thm:clock_on_positive_genus}
Let $(U, \Sigma, \mathscr{F}, G)$ be a framed multiverse, 
$G^\perp$ a dual of $G$, and $c \colon \mathscr{C}_{G^\perp} \To \Z$ a viable circulation function.
Then $\mathscr{S}_U^c$, $\mathscr{M}_G^c$ and $\mathscr{PR}_G^c$, with the relations of \refdef{general_clock_thm_relations}, are isomorphic distributive lattices.

Moreover, for $S,S' \in \mathscr{S}_U^c$, $M,M' \in \mathscr{M}_G^c$, and $R,R' \in \mathscr{R}_G^c$,
\begin{enumerate}
\item
$S \lessdot S'$ if and only if $S'$ is obtained from $S$ by a counterclockwise surface transposition;
\item 
$M \lessdot M'$ if and only if $M'$ is obtained from $M$ by a surface twisting up;
\item 
$R \lessdot R'$ if and only if $R'$ is obtained from $R$ by pushing up on a non-outer accessibility class.
\end{enumerate}
\end{theorem}

If we take unions over viable circulations, then we obtain isomorphisms of partially ordered sets 
\[
\mathscr{S}_U \cong \mathscr{M}_G \cong \mathscr{PR}_G
\]
which we can regard as ``disconnected" distributive lattices, whose ``connected components" are the $\mathscr{S}_U^c$, $\mathscr{M}_G^c$ and $\mathscr{R}_G^c$ above.

When $U$ is 2-cell embedded, then as discussed in \refsec{spines_framings} after \reflem{framed_contour_equivalent}, a spine $G$ for $U$ is unique up to isotopy, so we obtain a distributive lattice for the multiverse $(U, \Sigma, \mathscr{F})$, by taking the spine $G$ arbitrarily.

\begin{proof}
The dual $G^\perp$ is a finite graph, and by \reflem{dual_of_spine_connected}, it is connected. Applying \refthm{Propp_pushing0} to $G^\perp$ and a viable circulation function $c$, with the unpushable accessibility class being the outer accessibilty class of $G^\perp$ (\refdef{outer_accessibility_class}), provides a distributive lattice structure on $\mathscr{R}_G^c = \mathscr{PR}_G^c$ as claimed. The isomorphisms with $\mathscr{S}_U^c$ and $\mathscr{M}_G^c$ then provide the other distributive lattices as desired.
\end{proof}

\small

\bibliography{refs}
\bibliographystyle{amsplain}

\end{document}